\documentclass[11pt,letterpaper]{amsart}
\usepackage{amsmath,amssymb,amsthm,
color, euscript, enumerate}
\usepackage{dsfont}                                                                          
\usepackage{longtable}
\usepackage{caption}
\usepackage{multirow}

\newcommand{\RR}{{\mathbb R}}

\newcommand{\CC}{{\mathbb C}}
\newcommand{\be}{\begin{equation}}
\newcommand{\ee}{\end{equation}}
\newcommand{\ba}{\begin{array}}
\newcommand{\ea}{\end{array}}
\newcommand{\bea}{\begin{eqnarray}}
\newcommand{\eea}{\end{eqnarray}}

\usepackage{graphicx}
\usepackage{color}

\usepackage{transparent}

\newtheorem{theorem}{Theorem} 

\newtheorem{lemma}{Lemma}[section]
\newtheorem{proposition}{Proposition}[section]
\newtheorem{definition}{Definition}

\DeclareMathOperator{\sgn}{sgn}

\numberwithin{equation}{section}
\numberwithin{table}{section}

\allowdisplaybreaks

\begin{document}
 
\begin{center}
{\bf{\LARGE Long time asymptotics for the KPII equation}}\\

\end{center}
\vskip.2in
\begin{center}
   \textrm{\large Derchyi Wu} \\
 \textrm{Institute of Mathematics}\\
\textrm{Academia Sinica, Taipei, Taiwan} \\
{\today}
\end{center}  
\medskip

\vskip 10pt
{\bf ABSTRACT}

\begin{enumerate}
\item[]{\small  We derive long-time asymptotic estimates for small
solutions of the Kadomtsev--Petviashvili II (KPII)
equation using inverse scattering theory, contour
deformation, and stationary phase analysis. Under
a suitable smallness condition and explicit
regularity and decay assumptions  
on the initial data, the estimates hold uniformly
in regimes separated from the transition region,
without imposing nonphysical conditions on the
scattering data. We establish reduced representation
formulas for the nonlinear contributions in terms
of ABF (Ablowitz--Bar Yaacov--Fokas) Cauchy integrals
with controlled remainders. Their analysis yields
$o_\epsilon(t^{-19/20+\epsilon})$ and
$o_\epsilon(t^{-11/12+\epsilon})$ estimates in the
positive and negative regimes, respectively, while
stationary phase analysis determines the asymptotics
of the linear contribution.}

\end{enumerate}

{\tableofcontents}
\section{Introduction}\label{S:Intro}

The Kadomtsev-Petviashvili II (KPII) equation 
\be\label{E:KPII-intro}
(-4u_{x_3}+u_{x_1x_1x_1}+6uu_{x_1})_{x_1}+3u_{{x_2}{x_2}}=0     
\ee
plays a significant role in plasma physics, water waves, and various other areas of mathematical physics. As one of the few physically relevant integrable systems in multiple spatial dimensions, the KPII equation has been the focus of extensive research. In particular, its global well-posedness and stability properties have been investigated through both partial differential equation (PDE) methods and   inverse scattering theory (IST). For a comprehensive account of KPII well-posedness,
stability, and inverse scattering theory, together
with an extensive bibliography, we refer the reader
to the monograph by Klein and Saut \cite{KS21}.

 Despite this progress, a complete description of the long-time behavior of KPII solutions remains largely open. Using PDE methods, the asymptotic behavior of small solutions to generalized KPII equations,  excluding the KPII equation itself, has been investigated in works   such as \cite{HNS99, N11}. {\color{black}  In addition, the long-time asymptotics of   the $x_1$-derivative of  KPII solutions were studied in \cite{HN14}}. Using the Marchenko integral equation, Anders,
Kotlyarov, and Khruslov constructed a class of
KPII solutions whose long-time asymptotics are
described by trains of curved solitons in specified
moving regions \cite{AKK94}. More directly related
to the present work, Kiselev formally derived
long-time asymptotic formulas for small KPII
solutions using the IST \cite{Ki06}. However, his analysis relies on non-physical and non-generic assumptions, particularly the integrability of $(1+|\lambda |)  s_c$  and the boundedness of $\partial_{\lambda_I}s_c,\, \partial^2_{\lambda_I}s_c$. Since the Lax operator associated with the KPII equation is the heat operator, 
  the scattering data $s_c$ is naturally differentiable and exhibits decay in $(\overline\lambda-\lambda,\overline\lambda^2-\lambda^2)$,  and the associated eigenfunction $m(x,\lambda)$ depends nontrivially  on the entire complex $\color{black}\lambda$-plane. 
 As a result, the assumptions imposed by Kiselev lead to highly degenerate scattering data along the real axis $\lambda_I=0$.

The goal of this paper is to rigorously establish the large-time asymptotic behavior  of  small solutions to the KPII equation via the IST approach, without imposing any non-physical assumptions. To state our results, define
\be\label{intro-M}
 \mathfrak M^{p,q}=\{f{\color{black} (x_1,x_2)}: \sum_{l_1,l_2\le { q }} |{\color{black}\partial_{x_1 }^{l_1 }\partial_{x_2 }^{l_2 }} (1+{\color{black}|x_1|+|x_2|})^pf |_  {L^1 \cap L^2  }<\infty\},\ee   and $C$ a uniform constant that is independent of  $x$, $\lambda$. 
\begin{theorem}\label{T:u} Fix $ \delta>0$. Suppose that $a=\frac{x_2^2-3x_1x_3}{3x_3^2}$, $\delta<|a|<1/\delta$,    $ {|x_2|}\le Ct$, $t=- x_3    $,  and 
\be\label{intro-ID}
\begin{gathered}
|u_0|_{\mathfrak M^{2, 6} }\ll 1,\qquad |u_0|_{\mathfrak M^{4, 8} } <\infty.
 \end{gathered}
 \ee Then, for   every sufficiently small   $ \epsilon >0 $, the solution $u $ to the Cauchy problem for \eqref{E:KPII-intro} with initial data $ u_0$ can be written as:
 \begin{align}
u(x )=&\ u_1(x)+u_{2,0}(x)+ u_{2,1}(x),\quad x=(x_1,x_2,x_3),\label{E:rep-dec}\\
u_1(x )=& -\frac 1{\pi i}\partial_{x_1}\iint e^{(\bar\lambda -\lambda )x_1+(\bar\lambda^2-\lambda^2)x_2+(\bar\lambda^3-\lambda^3)x_3  }   {  s}_c(\lambda   )    \ d\overline\lambda \wedge d\lambda ,\label{E:rep-dec-1}\\
u_{2,0}(x)=&-\dfrac 1{\pi i} \iint  e^{(\bar\lambda -\lambda )x_1+(\bar\lambda^2-\lambda^2)x_2+(\bar\lambda^3-\lambda^3)x_3  }   {  s}_c(\lambda   )      (\overline\lambda -\lambda )( m(x,{\color{black}\overline\lambda })-1)   \ d\overline\lambda \wedge d\lambda ,\label{E:rep-dec-20}\\
u_{2,1}(x)=&
 -\dfrac 1{\pi i} \iint e^{(\bar\lambda -\lambda )x_1+(\bar\lambda^2-\lambda^2)x_2+(\bar\lambda^3-\lambda^3)x_3  }  {  s}_c(\lambda   )\partial_{x_1}  m(x,{\color{black}\overline\lambda })   \ d\overline\lambda \wedge d\lambda  , \label{E:rep-dec-u21}
\end{align}
with the following asymptotics as 
  $t\to +\infty$,
\begin{itemize}
\item [$\blacktriangleright$] For $+1/\delta>a >+\delta>0$, 
\begin{gather}
u _1(x )=   |u_0|_{\mathfrak M^{1,1} }  o( t^{-1 }), \label{E:intro-+u1}\\ 
       u _{2,0}(x ),\ u _{2,1}(x )=     |u_0|_{\mathfrak M^{2, 6} }|u_0|_{\mathfrak M^{4,8} } o_{\epsilon}( t^{-\frac{19}{20}+\epsilon}).  \label{E:intro-+u2}
       \end{gather}
  \item [$\blacktriangleright$] For $\color{black}-1/\delta<a <-\delta<0$,
\begin{gather}
   u_1 (x )=  \frac{2  i e^{i 4     t  r^3    
  }}{3t  }     s_c (-\frac{x_2}{3x_3}+ ir   )
-\frac{2  ie^{-i   4  t  r^3    
  }}{ 3 t   }     s_c( -\frac{x_2}{3x_3}- ir  )
 + |u_0|_{\mathfrak M^{3, 3} }o( t^{-1 }),\label{E:intro--u1} \\
 u _{2,0}(x ),\  u _{2,1}(x )=      |u_0|_{\mathfrak M^{2, 6} }|u_0|_{\mathfrak M^{4, 8} }o_\epsilon( t^{-\frac{11}{12}+\epsilon }) .\label{E:intro--u2}
  \end{gather}
\end{itemize}
 Here, $s_c(\lambda)$ denotes the scattering data of $u_0$, $m(x,\lambda)$ is the eigenfunction of the Lax operator, $a$ characterizes the stationary points of the phase function, $r=\sqrt{\frac{-a}{3}}$,  and $t$ corresponds to the direction of KPII propagation. Finally, $\mathcal O_\epsilon,\,o_\epsilon$
indicate that the implied constants or convergence
rates may depend on $\epsilon$; the notation
$\mathcal O$ and $o$ indicates no such dependence. 
\end{theorem}

 The theorem follows from Theorems \ref{T:u-1}, \ref{T:+u-2}, \ref{T:-u-2}, \ref{T:+u-2-u21}, \ref{T:-u-2-u21}, which are established in the subsequent sections. We make several remarks about Theorem \ref{T:u}:
\begin{enumerate}
\item The condition $|a|>0$ is necessary to apply integration
by parts to obtain decay in $t$, or to keep the
stationary points separated from one another or
from $0$. With the uniform bound $|a|>\delta>0$, the asymptotic estimates obtained are uniform. By contrast, the condition $1/\delta>|a|$ is only
needed to derive uniform asymptotics for $u_{2,1}$,
whose representation formula lacks the factor
$\overline\lambda'-\lambda'$. Both constraints play crucial roles in removing Kiselev's nonphysical conditions.

\item Since the phase function can be transformed into $-\frac 1{\pi  }[\lambda_I'^3+(a -3\lambda_R'^2)\lambda_I']$, the asymptotics for $u_1(x)$, \eqref{E:intro-+u1} and \eqref{E:intro--u1}, will be derived via Airy analysis in $\lambda'_I$ and Gaussian stationary phase in $\lambda'_R$ (see Proposition \ref{P:u-11}). 
\item To study the asymptotics of $u_{2,0}(x)$ and $u_{2,1}(x)$, we apply the Plancherel theorem to the Cauchy integral $\widetilde{ \mathcal CT 1}$ to separate $\tilde s_c$ from the Cauchy kernel (yielding $\iint(\tilde s_c)^{\vee_{\xi'_1,\xi'_2}} \times (\frac{e^{2\pi it S_0}}{p_{\overline\lambda}})^{\wedge_{\xi'_1,\xi'_2}} dx_1'dx_2'$, see \eqref{E:CT-ABF}). Therefore,  we can exploit holomorphicity in $\xi'_2$ outside the normalized $\frac 14$-disks to deform the contours and reduce the  estimates to residue and contour computations (see the computations involving   $R_{in,\gtrless,+},\,P_{in,\gtrless,+},\, Q_{in,\gtrless,+}$ and $\mathfrak R_{in,\gtrless,+}$, $\mathfrak P_{in,\gtrless,+}$, $ \mathfrak Q_{in,\gtrless,+}$ in $\S\,$\ref{SS:CI-0} and $\S\,$\ref{SS:CI-1}). Inside the normalized $\frac 14$-disks, we rely mainly on measure-size estimates, obtained via the affine transformation introduced in \cite{W87} (see the computations involving $R_{in,\gtrless,-},\,P_{in,\gtrless,-},\, Q_{in,\gtrless,-}$ and $\mathfrak R_{in,\gtrless,-},\,\mathfrak P_{in,\gtrless,-},\, \mathfrak Q_{in,\gtrless,-}$ in $\S\,$\ref{SS:CI-0} and $\S\,$\ref{SS:CI-1}), together with integration by parts.   We prove that the contributions from the contour computations and the measure-size estimates on these normalized $\frac 14$-disks are bounded by 
\be\label{E:intro-obs}
|u_0|_{\mathfrak M^{2, 6} }|u_0|_{\mathfrak M^{4, 8} } \mathcal O_{\epsilon}( t^{- \frac {19}{20}+\epsilon }).
\ee Consequently, we obtain   reduced Cauchy integral representation  formulas consisting of the ABF Cauchy integrals (Ablowitz--Bar Yaacov--Fokas Cauchy integrals), which arise from the composition of the residue terms, together with $|u_0|_{\mathfrak M^{2, 6} }|u_0|_{\mathfrak M^{4, 8} } o_{\epsilon}( t^{ -\frac {19}{20}+\epsilon})$ remainders for $u_{2,0}$ and $u_{2,1}$ (see Theorems \ref{T:first-reduction} and \ref{T:u21-first-reduction}). These formulas reduce the asymptotic analysis
to that of the ABF Cauchy integrals.  

The obstruction, namely, \eqref{E:intro-obs}, to obtaining an $o(t^{-1+\epsilon})$ estimate arises from the estimates of $Q_{in,<,-}$ and $\mathfrak Q_{in,<,-}$, where two singularities of the Cauchy kernels collide and integration by parts cannot be applied (cf. \eqref{E:+0-small-19}, \eqref{E:q-<0}, \eqref{E:u20-0-small-3}, and \eqref{E:u20-small-14(Q<)} below for  $Q_{in,<,-}$). Moreover, our current approach for \eqref{E:intro-obs} does not distinguish between the cases $a \gtrless \pm\delta\gtrless 0$. We hope that a further exploitation of the nonstationary nature of the phase at $\xi'_1=0$ may improve the estimate \eqref{E:intro-obs} and reveal a distinction between the cases $a \gtrless \pm\delta\gtrless 0$.

To ensure the convergence of the series of  ABF Cauchy integrals and the corresponding remainder series, we impose the smallness condition 
$
|u_0|_{\mathfrak M^{2, 6} }\ll 1
$ (The admissible upper bound is determined by the
requirements in  Lemmas \ref{L:potential},  \ref{L:u20-large},  \ref{L:u20-small}, \ref{L:u20-0-small}, equations \eqref{E:n-CIO-ABF-finite}, \eqref{E:CIO-+new-u21}, and Definitions \ref{D:n-CIO}  and \ref{D:u21-CIO-1}). The computations leading to the reduced representation formulas also requires suitable control of $(x_1',x_2')$ at infinity  (see \eqref{E:large-2}, \eqref{E:small-2}, and \eqref{E:+0-small-2} for examples). This corresponds to the condition $|u_0|_{\mathfrak M^{4, 6} } <\infty$. For simplicity, we strengthen this condition to $|u_0|_{\mathfrak M^{4, 8} } <\infty$, in accordance with the assumption $q\geq 8$ in Theorem  \ref{T:cauchy-0} \cite{W87}.

\item Owing to (i) the lack of efficient estimates for higher derivatives of the ABF Cauchy integrals, and (ii) the fact that, regardless of how small the integration region is, the first derivatives of the ABF Cauchy integrals admit at best  $\mathcal{O}(1)$ bounds, the   $o_\epsilon \left(t^{-\frac{11}{12}+\epsilon}\right)$ estimate  for the ABF Cauchy integral series for $- 1/\delta< a<-\delta<  0$ is optimal within our approach (see  Theorem \ref{T:-u-2} and  Lemma \ref{L:-3ndReduction-u21}).  

\item Whether $o(t^{-1})$ estimates hold for $u_{2,0}$
and $u_{2,1}$ for generic initial data $u_0$
remains an open question.

For comparison, in the asymptotic theory of the KPI equation \cite{DLP25}, a $\frac{\pi}{2}$-phase shift is obtained. Moreover, one obtains $o(t^{-1})$ and
$\mathcal O(t^{-1})$ estimates for $ u_{2,0}$ and $u_{2,1}$ for $a \gtrless \pm\delta \gtrless 0$. These results rely on distinct analytical features: the associated Lax operator is the Schr$\ddot{\textrm {o}}$dinger operator, the phase function is antisymmetric in $k, l$, the scattering data lies in Sobolev spaces in $l$, and the eigenfunction $m$ depends only on $k \in \mathbb{R}$. 
\end{enumerate}
The paper is organized as follows. In Section \ref{S:Pre}, we present preliminary materials, including the IST for the KPII equation and an introduction to the stationary phase method. 

In Section \ref{S:sigma--}, we derive the asymptotic behavior of $u_1$ by applying the stationary phase method near the stationary points and using integration by parts away from them.

In Subsection  \ref{SS:technical}, we first introduce the ABF Cauchy integrals  $ \ \left[\mathfrak C\mathfrak T\right]^{0,(n)}$ which play an important role in  analyzing the asymptotics of $u_{2,0}$ in the subsequent subsections. 

To illustrate the structure of the  ABF Cauchy integrals, we note that $ \ \left[\mathfrak C\mathfrak T\right] ^{0,(1)}$, arises from taking the residue at \[\xi'_2=2\pi i\xi'^2_1+2\xi'_1\overline\lambda',\] and is a triple integral over the spatial variables $(x_1',x_2')$ and the spectral variable $\xi_1''$ where $ \xi'_1=\xi''_1+\frac{\lambda'_I}{2\pi}$. The $(x_1',x_2')$-integral is well behaved under sufficient regularity of the initial data $u_0 $. The $\xi_1''$-integral features an oscillatory Airy-type propagator $e^{2\pi i t \mathfrak G}$,  multiplied by an amplitude $\mathcal F$ which is the product of a bounded exponential factor and the decay factor $\frac{1}{1+\xi'^4_1 }$.  Consequently, the asymptotic behavior of the ABF Cauchy integrals   is obtained by applying the stationary phase method to the propagator $e^{2\pi i t \mathfrak G }$, and analyzing the behavior of the amplitude
$\mathcal F$ near the zeros of its exponent,
where exponential decay may fail.

In Subsection \ref{SS:CI-0}, we analyze the Cauchy integrals. The main techniques for analyzing the CIOs are integration by parts (away from the singularity $z_0$), measure-size estimates via an affine transformation (near $\xi'_{2,n}=0$), and the deformation method.  As an application, we derive a reduced representation formula for $u_{2,0}$ at the end of this subsection, which leads to our subsequent analysis of the ABF Cauchy integrals.

In Subsections \ref{SS:u+2} and \ref{SS:u-2}, we determine the asymptotic behavior of $u_{2,0}$ in the regimes $a\gtrless \pm\delta\gtrless  0$,  respectively. This is achieved by refining the decomposition of the ABF Cauchy integrals, establishing the integrability of $\partial_{\lambda'_I}\widetilde s_c$ or $(1+|\lambda'|)\widetilde s_c$ in various regimes, discarding terms with rapidly decaying amplitudes, and using several key tools: measure-size estimates,  the   factor    $(\overline\lambda'-\lambda')$, 
integration by parts, and the estimates developed in Subsection \ref{SS:technical}.

In Section \ref{S:m-1}, we adapt the approach from Section \ref{S:m} to investigate the Cauchy integrals  $\partial_{x_1} 
 \widetilde{\left(\mathcal CT\right)^n1}$ and derive the asymptotic behavior of $u_{2,1}$.  To apply integration by parts without imposing highly degenerate conditions on $\partial_{\lambda'_I}\widetilde s_c$ and $(1+|\lambda'|)\widetilde s_c$, we impose the additional condition $ | a| <1/\delta$ to ensure the required differentiability and integrability of $ \widetilde s_c$. Differentiation of the Cauchy integral operators (CIOs) introduces an additional factor $  2\pi i\xi'_{1,h }$ in the CIO's.   We decompose the CIOs and the ABF Cauchy integrals using the identity  \[  2\pi i\xi'_{1,h }=
2i\lambda'_{I,h }+ 2i(\pi\xi'_{1,h }-\lambda'_{I,h }) .\] The resulting estimates either reduce to those used for $u_{2,0}$ or benefit from milder singular behavior of the kernels.
 
Appendix A provides a list of the symbols used
throughout the paper.

{\bf Acknowledgments}.  I thank Jean-Claude Saut for suggesting the asymptotic problem for the KP equations. I am also grateful to Jiaqi Liu for insightful discussions that led to study the ABF Cauchy integrals, and {\color{black}to Barbara Prinari for thoroughly reading and discussing the manuscript, as well as for pointing out several sharp estimates. I further thank Theodoros Horikis and the Department of Mathematics at the University of Ioannina for their warm hospitality.} This work was supported by NSC 113-2115-M-001-007-. 

\section{Preliminaries}\label{S:Pre}

\subsection{The IST for KPII equations}\label{SS:KP-IST}Denote $x=(x_1,x_2,x_3)$, $l=(l_1,l_2,l_3)$, $\partial_x^l=\partial_{x_1}^{l_1}\partial_{x_2}^{l_2}\partial_{x_3}^{l_3}$,  $|l|=|l_1|+|l_2|+|l_3|$, and $\widehat f(\xi )=\widehat f(\xi_1,\xi_2)=\iint f(x)e^{-2\pi i (x_1\xi_1+x_2\xi_2)}dx_1dx_2$.
By establishing an IST, Wickerhauser solved the Cauchy problem of the KPII equation with a vacuum background: 
\begin{theorem}[The Cauchy Problem   \cite{W87}]\label{T:cauchy-0} Let $q\ge {\color{black} 8}$. If the initial data  $u_0\in \mathfrak M^{0,q}$ satisfies 
 \be\label{E:intro-ID} 
 u_0(x_1,x_2)=\overline{u_0(x_1,x_2)}, \qquad     |u_0  |_ {\mathfrak M^{0,0} } <1. 
\ee Then, we can construct the {\sl forward scattering transform}:
\be\label{E:intro-SD} 
\begin{split}
\mathcal S:u_0\mapsto  s _c(\lambda)=& \frac{\sgn(\lambda_I)}{2\pi i} \left[{u_0(\cdot)m_0(\cdot,\lambda)}\right]^\wedge(\frac{\overline\lambda-\lambda}{2\pi i},\frac{\overline\lambda^2-\lambda^2}{2\pi i}) 
\equiv \frac{\sgn(\lambda_I)}{2\pi i} \left[{u_0 m_0 }\right]^\wedge(\xi_1,\xi_2) ,
\end{split}
\ee satisfying the algebraic and analytic constraints:
 \begin{gather}
   s_c(\lambda)=  \overline{s_c( \overline\lambda)}, \quad    | (1+ |\xi  |)^{q }       s_c (\lambda(\xi))  |  _{
  L^\infty\cap L^2(d\xi_1d\xi_2)} 
    \le    {C |u_0|_{\mathfrak M^{0,q}}}.\label{E:intro-s-c-ana-c-0} 
 \end{gather} Here $m_0$ solves  the boundary value problem of the Lax equation:
\begin{gather}
(-\partial_{x_2}+\partial_{x_1}^2 +2\lambda\partial_{x_1}
+u_0(x_1,x_2))m_0(x_1,x_2 ,\lambda)=0, \ \label{E:intro-Lax}\ \
\lim_{|x|\to\infty}m_0(x_1,x_2,\lambda)= 1,  
\end{gather}
 
 Moreover, the solution $u$ to the KPII Cauchy problem  is given by: 
\begin{align}
u(x )= -\frac 1{\pi i}\partial_{x_1}\iint  T  m  \ d\overline\zeta\wedge d\zeta , & \label{E:intro-Lax-u-new-1-0}
\end{align} satisfying
\begin{align}
u(x)=\overline {u(x)},\quad |   u   |_{\mathfrak M^{0,q-{\color{black} 4}}}\le C |u_0|_{\mathfrak M^{0,q}}.& \label{E:intro-Lax-u-new-2-0}
\end{align}
Here $m$ solves the  Cauchy integral equation:
\begin{gather}   
    {   m} (x , \lambda) =1 +\mathcal C  T 
       m   ,\quad m_0(x_1,x_2,\lambda)=m(x_1,x_2,0,\lambda)  \label{E:intro-CIE-t}   
     \end{gather} with  $\mathcal C$ being the Cauchy integral operator, and $T$   the continuous scattering operator:
\begin{align}
\mathcal C \phi  (x,\lambda)
\equiv & -\frac{1}{2\pi i}\iint_{\CC}\frac{\phi(x,\zeta)}{\zeta-\lambda}d\overline\zeta\wedge d\zeta,\label{E:ct-operator-0}\\
 T  \phi  (x ,\lambda)
\equiv & T\circ \phi= e^{(\overline\lambda-\lambda)x_1+(\overline\lambda^2-\lambda^2)x_2 +(\overline\lambda^3-\lambda^3)x_3 }{  s}_c(\lambda  )\phi(x, \overline\lambda). \label{E:ct-operator-1}
\end{align}

\end{theorem}

\subsection{The stationary points}\label{SS:stationary}

Building upon Theorem \ref{T:cauchy-0}, we are going to investigate the long-time asymptotic behavior of the KPII solution using the stationary phase method (cf \cite{DLP25} for the corresponding analysis in the KPI case). The natural coordinates for applying this method are the variables $(\zeta_R',\zeta_I')$ introduced in \eqref{E:real-ima-new}.   To motivate their use, we define :
\be\label{E:real-ima} 
\begin{split}
t_1 =\frac{x_1}{t}, \quad t_2=\frac{x_2}{t}, &\quad t= -x_3,\\
2\pi i\xi_1= \overline\zeta-\zeta, & \quad 2\pi i\xi_2=\overline\zeta^2-\zeta^2,\\
\zeta=\frac {\xi_2}{2\xi_1}-i\pi\xi_1=\zeta_R+i\zeta_I,  &  \quad d\overline\zeta\wedge d\zeta=2i\,d\zeta_Rd\zeta_I=\frac{ i\pi}{ | \xi_1| }d\xi_1 d\xi_2.
\end{split}
\ee
and the phase function $\mathbb S_0$ by  
\be \label{E:phase} 
   \mathbb S_0(t_1,t_2 ;\zeta( \xi) )
\equiv  \frac{(\bar\zeta -\zeta )x_1+(\bar\zeta^2-\zeta^2)x_2+(\bar\zeta^3-\zeta^3)x_3 }{ 2\pi it}.
\ee
Notice that due to the propagation of the KPII equation \eqref{E:KPII-intro}, we will investigate the asymptotic of the KPII solution $u(x)$ as $t\to \infty$.

To simplify the computation by eliminating  quadratic terms, we introduce :
\begin{align}
(\zeta,\overline\zeta)=(\zeta' + \frac{t_2}{3},\overline\zeta' +
 \frac{t_2}{3}), &\quad (\xi'_1,\xi'_2)=(\xi_1,\xi_2-\frac {2t_2}3\xi_1),\label{E:real-ima-new}\\
2\pi i\xi'_1= \overline\zeta'-\zeta',  &\quad 2\pi i\xi_2'=\overline\zeta'^2-\zeta'^2,\nonumber\\
\zeta'=\frac {\xi_2'}{2\xi_1'}-i\pi\xi_1'=\zeta'_R+i\zeta'_I,    &\quad d\overline\zeta'\wedge d\zeta'=2i\,d\zeta_R'd\zeta_I'=\frac{i\pi
}{  |\xi_1'| }d\xi'_1 d\xi_2',\nonumber\\
\partial_{\zeta'_I}= -\frac 1{\pi}\partial_{\xi'_1}-\frac{1}{\pi}\frac {\xi'_2}{\xi'_1}\partial_{\xi'_2},&\quad
\partial_{\zeta'_R}=   2\xi'_1 \partial_{\xi'_2},\nonumber
\end{align}
 which induces 
 the definition, estimates
\be\label{E:scatter-new}
\begin{gathered}
f(\zeta)=   f(\zeta' + \frac{t_2}{3}  )\equiv \widetilde f(\zeta'),\\
{\color{black}(| \zeta_R|^{l_1 }  +|\zeta_I|^{l_2} )|\partial_{\zeta_R}^{j_1}\partial_{\zeta_I}^{j_2} s_c    |\sim   (|\zeta_R'|^{l_1 }  +|\zeta_I'|^{l_2} )|\partial_{\zeta'_R}^{j_1}\partial_{\zeta'_I}^{j_2}    \widetilde s_c | } 
   ,\quad \zeta'_I\ne 0,
\end{gathered}
\ee{\color{black}by $|t_2|<C$,} and changes the phase function to
 \be \label{E:phase-new}
 \begin{split}
\mathbb  S_0(t_1,t_2 ;\zeta(\xi))\equiv& \frac{1 }{2\pi i  }  [a(\overline\zeta'-\zeta') - (\overline\zeta'^3-\zeta'^3)]=-\frac 1{\pi  }(a\zeta_I'+\zeta_I'^3-3\zeta_I'\zeta_R'^2)\\
= & a\xi_1'+ \pi^2\xi_1'^3-\frac 34\frac{\xi_2'^2}{\xi_1'}  
 \equiv  S_0(a;\zeta'(\xi ')), 
 \end{split}
\ee with
\be\label{E:a}
 a= t_1+\frac 13 t_2^2 .
 \ee

\begin{definition}\label{D:stationary} Let the phase function $S_0(a;\zeta')$ be defined by \eqref{E:phase-new} and \eqref{E:a}. Define 
\begin{itemize}
\item For $a<0$, the stationary points of $S_0$ are purely imaginary:
\be\label{E:stationary-+}
\zeta'_R=0, \quad\zeta'_I=  \pm \sqrt{\frac {-a}3}\equiv \pm r,\quad r>0.
\ee
\item For $a>0$, the stationary points of $S_0$ are purely real:
\be\label{E:stationary--}
 \zeta'_R=\pm \sqrt{\frac{ a}3}\equiv \pm r,\quad\zeta'_I=  0,\quad r>0.
\ee
\end{itemize}
\end{definition}

Applying the Fourier theory to the Lax equation \eqref{E:intro-Lax}, for fixed $\lambda$  \cite{W87},  we derive
\begin{align}
m_0(x_1,x_2,\lambda)  
\equiv&1+G_{u_0}m_0,\label{E:wick-m}
\end{align}  where
\begin{gather}
 G_{u_0}f=   \iint e^{2\pi i(x_1\xi_1+x_2\xi_2)} \frac{\left[u_0f\right]^\wedge(\xi_1,\xi_2,\lambda)}{p_\lambda(\xi_1,\xi_2)}d\xi_1 d\xi_2,\label{E:AYF-G-infty-G-new}\\
|G_{u_0}f|_{L^\infty}\le C |u_0|_{\mathfrak M^{0,0}}|f|_{L^\infty} ,\label{E:AYF-G-infty-G} 
\end{gather}  
with
\be\label{E:p}
\begin{gathered}
p_{\lambda }(\xi _1,\xi _2)=(2\pi i\xi _1+\lambda )^2-(2\pi i\xi _2+{\lambda }^2),\\
\left|\frac 1{p_{\lambda }}\right|_{L^1(\Omega_{\lambda }, d\xi _1 d\xi_2 )}\le  \frac C{(1+|\lambda _I|^2)^{1/2}},\quad 
\left|\frac 1{p_{\lambda }}\right|_{L^2(\Omega_{\lambda }^c,d\xi _1 d\xi _2)}\le \frac C{(1+|\lambda _I|^2)^{1/4}},  \\
\Omega_{\lambda }\equiv \Omega_{\lambda ,1},\quad \Omega_{\lambda ,\rho}=\{(\xi _1,\xi _2)\in\RR^2 : |p_{\lambda }(\xi _1,\xi _2)|<\rho\}.
\end{gathered}
\ee

\begin{lemma}\label{L:sd-der} {\color{black} Suppose   $|t_2|<C$ and $\lambda'_I\ne 0$.}
\begin{align}
|\partial_{\lambda'_R}^j\widetilde s_c| \le & C(1+ |u_0|^2_{\mathfrak M^{j ,j}})(1+|\lambda'_I  |)^{j}  ,\label{E:sd-der-hot}\\
  |\partial_{\lambda'_I}^j \widetilde s_c|\le & C(1+ |u_0|^2_{\mathfrak M^{j ,j}})(1+|\lambda'_R  |)^{j} ,\label{E:sd-der-hot-1}\\
|\partial_{\lambda'_R} \partial_{\lambda'_I}  \widetilde s_c|\le &C (1+|u_0|^2_{\mathfrak M^{2,2}})(1+|\lambda' |)^{2} .\label{E:sd-der-hot-2}
\end{align}

\end{lemma}
 
\begin{proof}
Taking derives on the Lax equation \eqref{E:intro-Lax} and applying \eqref{E:wick-m}-\eqref{E:AYF-G-infty-G}, we obtain  
\be\label{E:sd-der-m}
 |\partial_{ {x_1}}^jm_0 |,\quad
|\partial_{\lambda_{R}}^j m_0 |,\quad
|\partial_{\lambda_{I}}^jm_0 |\le  C|u_0|_{\mathfrak M^{0,j}}.
\ee  
Therefore, the lemma is proved by using \eqref{E:intro-SD}.
\end{proof}

\section{Long time asymptotics of $u_1(x) $}\label{S:sigma--}

Throughout this section,   $a$, $r$, $t_i$, $t$ are as defined in    \eqref{E:real-ima}  and \eqref{E:a}-\eqref{E:stationary--}. 
Let $\psi$ be a non negative smooth   cutoff function such that $\psi(s)=1$ for $|s|\le \frac 12$ and $\psi(s)=0$ for $|s|\ge \frac 34$. Given $w_0\in\mathbb R$, define
\be\label{E:localize-fct}
\psi_{r,w_0}(s)=\psi\left(\frac{16(s-w_0)}{r}\right)+\psi\left(\frac{16(s+w_0)}{r}\right).
\ee Let   
\be\label{E:local}
 \chi( \zeta'  )= 
\left\{ 
{\begin{array}{ll}
 \psi_{r,r}(\zeta'_R)\psi_{r,0}(\zeta'_I),&\textit{ for $a>0$,}\\
\psi_{r,r}(\zeta'_I)\psi_{r,0}(\zeta'_R),&\textit{ for $a<0$.}
\end{array}}
\right.
\ee  Decompose the linearized term $u_1(x)$, defined by \eqref{E:rep-dec-1}, into 
\begin{align}
&\qquad u_1(x )= u_{1,1}(x)+ u_{1,2}(x),\label{E:rep-dec-1-dec}\\
u_{1,1}(x)=&
-\dfrac 1{\pi i} \iint  \widetilde {  s}_c(\zeta'  )e^{2\pi i tS_0 }    (\overline\zeta'-\zeta')\chi( \zeta'  )   \ d\overline\zeta'\wedge d\zeta',\label{E:rep-dec-11}
\\
 u_{1,2}(x)=&
 -\dfrac 1{\pi i} \iint  \widetilde {  s}_c(\zeta'  )e^{2\pi i tS_0 }    (\overline\zeta'-\zeta')(1-\chi( \zeta'  ))   \ d\overline\zeta'\wedge d\zeta' .\label{E:rep-dec-12}
\end{align}

We provide a quadratic growth estimate on the phase function away from stationary points:
\begin{lemma}\label{L:growth} If $a\ne 0$, then on the support of $1-\chi(\zeta')$, the phase function $S_0$ satisfies:
\begin{align}
|\nabla S_0|\equiv|(\partial_{\zeta'_R} S_0,\partial_{\zeta'_I} S_0)|\ge & {\color{black}\frac 1C} (|a|+|\zeta'|^2) ,\label{E:phase-d}\\
|\Delta S_0|\equiv|{\color{black}(\partial_{\zeta'_R}^2S_0,\partial_{\zeta'_I} ^2 S_0)}|\le & C|\zeta'|.  \label{E:phase-d-2}
\end{align}  
\end{lemma}
\begin{proof} We have
\be\label{E:partial-x}
\begin{split}
\partial_{\zeta'_R}S_0=&+\frac{6}{\pi}\zeta'_R\zeta'_I,\quad
\partial_{\zeta'_I}S_0= +\frac 1{ \pi}(-a+3({\zeta_R'}^2-{\zeta_I'}^2)).
\end{split}
\ee Therefore \eqref{E:phase-d-2} is justified and
\begin{align}
 |\partial_{\zeta'_R}S_0|^2+|\partial_{\zeta'_I}S_0|^2 
\ge &{\color{black} \frac{9}{\pi^2}}[{\zeta'_R}^4+2{\zeta'_R}^2 {\zeta'_I}^2 +({\zeta'_I}^2+\frac a3)^2] , \qquad a\le 0,\label{E:ge--}\\
|\partial_{\zeta'_R}S_0|^2+|\partial_{\zeta'_I}S_0|^2
\ge &{\color{black} \frac{9}{\pi^2}}[{\zeta'_I}^4+2{\zeta'_R}^2 {\zeta'_I}^2 +({\zeta'_R}^2-\frac a3)^2], \qquad a\ge 0 .\label{E:ge-+}  
\end{align} 

Since proofs are identical. We only give the proof of \eqref{E:phase-d} for $a<0$ for simplicity. 
By assumption, if $\psi_{r,r}(\zeta'_I)= 1$, then $\psi_{r,0}(\zeta'_R)\ne 1$. Namely,
\be\label{E:1-not-1}
\frac{||\zeta'_I|-r|}r \le \frac 1{32}< \frac{|\zeta'_R|}r,
\ee
along with $r\sim \pm\sqrt{\frac {-a}3}$ and \eqref{E:ge--}, implies that
\be\label{E:ima}
|\partial_{\zeta'_R}S_0|^2+|\partial_{\zeta'_I}S_0|^2 
\ge {\color{black}\frac 1C}( {\zeta'_R}^4+  {\zeta'_I}^4)\ge  {\color{black}\frac 1C}( {\zeta'_R}^4+  {\zeta'_I}^4+a^2).
\ee 

On the other hand, if  $\psi_{r,r}(\zeta'_I)\ne 1$, then there exists $\color{black}C_0>1$ such that  
\be \label{E:ge}
\textit{ either $|\zeta'_I|\le  \frac 1{\color{black}C_0}r$ or $
|\zeta'_I|\ge  {\color{black}C_0}r$ holds}.
\ee

Applying \eqref{E:ge--}, we have
\be\label{E:le-1}
\begin{split}
|\partial_{\zeta'_R}S_0|^2+|\partial_{\zeta'_I}S_0|^2 
\ge& {\color{black}\frac 1C}( {\zeta'_R}^4+  a^2)\ge   {\color{black}\frac 1C}( {\zeta'_R}^4+  {\zeta'_I}^4+a^2),\quad\,  |\zeta'_I|\le  \frac 1{\color{black}C_0}r,\\
|\partial_{\zeta'_R}S_0|^2+|\partial_{\zeta'_I}S_0|^2   
\ge &  {\color{black}\frac 1C}(  {\zeta'_R}^4+{\zeta'_I}^4)\ge   {\color{black}\frac 1C}( {\zeta'_R}^4+  {\zeta'_I}^4+a^2),\quad |\zeta'_I|\ge  {\color{black}C_0}r.
\end{split}
\ee 

\end{proof}

\begin{proposition}\label{P:u-12} Suppose that $|a|>\delta>0$.
\be \label{E:u-12}
   u_{1,2}(x) = |u_0|_{\mathfrak M^{1,1}} o(t^{-1}).
\ee

\end{proposition}
\begin{proof}  
Integration by parts, applying \eqref{E:intro-s-c-ana-c-0}, {\color{black} the factor $(\overline\zeta'-\zeta')$,} and Lemmas \ref{L:sd-der}, \ref{L:growth},  we have
\begin{align}
    |u_{1,2}(x) |
\le & \frac Ct|   \iint   e^{-2  it(a\zeta'_I+{\zeta'_I}^3-3\zeta'_I{\zeta'_R}^2) }\nabla\cdot \left( \widetilde{  s}_c(\zeta'  ) (\overline\zeta'-\zeta')(1-\chi)\frac{\nabla S_0}{|\nabla S_0|^2}\right) d\zeta'_R d\zeta'_I  | , \label{E:i-2-est-o}
\end{align}with
\be\label{E:integration-by-parts-o}
|\nabla\cdot \left(\widetilde{  s}_c(\zeta'  )(\overline\zeta'-\zeta')(1-\chi)\frac{\nabla S_0}{|\nabla S_0|^2}\right)|_{L^1(d\zeta'_R d\zeta'_I)}<C|u_0|_{\mathfrak M^{1,1}}.
\ee Here note that discontinuity of $\tilde s_c$ at $\zeta_I'=0$ can be neglected thanks to the factor $(\overline\zeta'-\zeta')$ and the uniform lower boundedness $|a|>\delta>0$ is used.

Setting $ \tilde\zeta_R   = \zeta'_I{\zeta'_R}^2    $, for $\zeta'_R\gtrless 0$, $ \zeta_I\gtrless 0$,
\begin{align}
& | u_{1,2}(x) | \label{E:i-2-est-o-new}\\
\le &  \frac Ct|  \int _{0}^\infty\int_{0}^\infty  e^{-2  it(a\zeta'_I+{\zeta'_I}^3-3\tilde\zeta_R ) }\nabla\cdot \left(\widetilde{  s}_c (\overline\zeta'-\zeta')(1-\chi)\frac{\nabla S_0}{|\nabla S_0|^2}\right)\frac{\partial(\zeta'_R,\zeta'_I)}{\partial(\tilde\zeta_R ,\zeta'_I)} d\tilde \zeta_R  d\zeta'_I |   \nonumber\\
+&  \frac Ct|  \int _{-\infty}^0\int_{-\infty}^0  e^{-2  it(a\zeta'_I+{\zeta'_I}^3-3\tilde \zeta_R ) }\nabla\cdot \left(\widetilde{  s}_c(\zeta  )(\overline\zeta'-\zeta')(1-\chi)\frac{\nabla S_0}{|\nabla S_0|^2}\right)\frac{\partial(\zeta'_R,\zeta'_I)}{\partial(\tilde \zeta_R ,\zeta'_I)}d\tilde \zeta_R  d\zeta'_I |   \nonumber
\end{align}where
\begin{multline}\label{E:change-variables}
\qquad|\nabla\cdot \left(\widetilde{  s}_c(\zeta'  )(\overline\zeta'-\zeta')(1-\chi)\frac{\nabla S_0}{|\nabla S_0|^2}\right)\times \frac{\partial(\zeta'_R,\zeta'_I)}{\partial(\tilde \zeta_R ,\zeta'_I)}|_{L^1(d\tilde \zeta_R  d\zeta'_I   )}\\
= |\nabla\cdot \left(\widetilde{  s}_c(\zeta'  )(\overline\zeta'-\zeta')(1-\chi)\frac{\nabla S_0}{|\nabla S_0|^2}\right)|_{L^1(d\zeta'_R d\zeta'_I)}<C|u_0|_{\mathfrak M^{1,1}}.\qquad
\end{multline}
Therefore \eqref{E:u-12} follows from Fubini's theorem and the Riemann-Lebesgue lemma.

\end{proof}

\begin{proposition}\label{P:u-11}  
As $t\to\infty$,
\begin{gather} 
   u_{1,1}(x)=    \frac{2  i e^{i 4t  r^3    
 }}{3t  }    \widetilde s_c ( + ir   )
-\frac{2  ie^{-i   4t  r^3    
  }}{ 3 t   }    \widetilde s_c(  - ir  )
 +  |u_0|_{\mathfrak M^{3,3}}  \mathcal O( { t^{- 3/2}}),\  \textit{  for $ a <-\delta<0$},
\label{E:u-11-}\\
   u_{1,1}(x)= |u_0|_{\mathfrak M^{1,1}} \mathcal O( t^{-4/3}),\qquad \textit{ for $ a>+\delta> 0$}.  \label{E:u-11+}
\end{gather}

\end{proposition}

\begin{proof} 
\noindent  {$\blacktriangleright$}  $\underline{\textbf{ Proof  of $\color{black}a <-\delta<0$ :}}$  Write
\begin{align}
  u_{1,1}(x) 
=& -\frac 2{ \pi}\int  d\zeta'_I e^{-2  i t(a\zeta'_I +{\zeta'_I}^3)}\psi_{r,r}(\zeta'_I)(\overline\zeta'-\zeta')\int  d\zeta'_R \ \ e^{ -\pi i t (-\frac 6\pi  \zeta'_I ){\zeta'_R}^2     } \psi_{r,0}(\zeta'_R)  \widetilde{  s}_c(\zeta'  ) .\label{E:u-11-asym-+}
\end{align}

Define the Fourier transforms  
\be\label{E:fourier}
\begin{gathered}
\phi^{\wedge_{\zeta'_R}}(\eta'_R,\zeta'_I)= \int e^{-2\pi i\zeta'_R\eta'_R}\phi(\zeta'_R,\zeta'_I)d\zeta'_ R,\ \ 
\phi^{\wedge_{\zeta'_I}}(\zeta'_R,\eta'_I)= \int e^{-2\pi i\zeta'_I\eta'_I}\phi(\zeta'_R,\zeta'_I)d\zeta'_I  .
\end{gathered}
\ee Setting $f  \equiv\psi_{r,r}(\zeta'_I)\psi_{r,0}(\zeta'_R) (\overline\zeta'-\zeta')\widetilde{  s}_c(\zeta'  )$, applying   Lemma \ref{L:sd-der}, $ a<-\delta<0$ (uniformly bounded above),
  and H$\ddot{\mbox{o}}$lder's inequality, we obtain successively: for $0\le j\le 3$, 
\begin{gather*}
   | \partial^j_{\zeta'_R}f |_{     L^2(d\zeta'_R)}  , \ 
|   (1+ |\eta'_R|^3)   f  ^{\wedge_{\zeta'_R}}|_{ L^2(  d\eta'_R)  } ,\ 
 |(1+ {\eta'_R}^2) f   ^{\wedge_{\zeta'_R}} |_{ L^1(d\eta'_R) }<C |u_0|_{\mathfrak M^{3,3}}. 
\end{gather*} 
Hence we can apply  the stationary phase theorem to get
\begin{align}
   u_{1,1}  
=& -\frac 2{ \pi}\frac{1}{\sqrt t}\int d\zeta'_I e^{-2  i t(a\zeta'_I +{\zeta'_I}^3)} e^{ \pi i \frac{\sgn(\zeta'_I)}4}\frac{1}{\sqrt{|\frac 6\pi \zeta'_I|}} \int d\eta'_R\left(1+ \mathcal O(\frac{{\eta'}^2_R}{t|\zeta'_I|})\right)f ^{\wedge_{\zeta'_R}} (\eta'_R,\zeta'_I  ) \label{E:stationary-r}\\
=& -\frac 2{ \pi}\frac{1}{\sqrt t}\int d\zeta'_ I e^{-2  i t(a\zeta'_I +{\zeta'_I}^3)}\frac{e^{+\pi i \frac{\sgn(\zeta'_I)}4}}{\sqrt{|\frac 6\pi \zeta'_I|}}\psi_{r,r}(\zeta'_I)(\overline\zeta'-\zeta')\widetilde{  s}_c(0,\zeta'_I  )+ |u_0|_{\mathfrak M^{3,3}}\mathcal O(\frac{1}{t^\frac 32}).\nonumber
\end{align}

Setting $g \equiv\frac{e^{+\pi i \frac{\sgn(\zeta'_I)}4}}{\sqrt{|\frac 6\pi \zeta'_I|}}\psi_{r,r}(\zeta'_I)(\overline\zeta'-\zeta')\widetilde{  s}_c(0,\zeta'_I  ) $,   applying $\color{black}a<-\delta<0$ (uniformly bounded above),  Lemma \ref{L:sd-der}, 
 and H$\ddot{\mbox{o}}$lder's inequality, for $0\le j\le 3$,  we have
\begin{gather} 
|\partial^j_{\zeta'_I}g |_{L^2(  d\zeta'_I)},\ 
| (1+ {\eta'_I}^3)  g ^{\wedge_{\zeta'_I}}|_{L^2(  d\eta'_I)},\ 
 |(1+ {\eta'_I}^2) g ^{\wedge_{\zeta'_I}}( 0, \eta'_I)|_{ L^1(d\eta'_I)}|< C|u_0|_{\mathfrak M^{3,3}} .\label{E:i-L1} 
\end{gather}

Besides, recall the Airy function
\be\label{E:airy}
Ai(z)=\frac{1}{2\pi}\int_{\mathbb R}e^{i(\frac{s^3}{3}+zs)}ds
\ee
which satisfies
\begin{gather}
|Ai(z)|\le C(1+|z|)^{-\frac{1}4},\quad z\in\mathbb R,\label{E:airy-bdd}\\
Ai(-x)=\frac{1}{\sqrt \pi x^\frac 14}\cos\left(\frac 23 x^\frac 32-\frac \pi 4  \right)+\mathcal O( x^{-\frac 74 }),\quad x\to\infty,
\label{E:airy-asy}\\
\left(e^{-2  i t(a\zeta'_I +{\zeta'_I}^3)}\right)^{\wedge_{\zeta'_I}}(-\eta'_I)= 
 \frac{2\pi}{ (6  t)^\frac 13} Ai\left(\frac{(2  t)^\frac 23}{\sqrt[3]{3}} (a-\frac{ \pi\eta'_I}{ t})\right) .\label{E:airy-app}
\end{gather}

Using  \eqref{E:i-L1}, the Fourier multiplication formula, \eqref{E:airy-bdd}, and \eqref{E:airy-app}, \eqref{E:stationary-r} turns into
\begin{align}
   u_{1,1}(x)
= &  -\frac 2{ \pi}\frac{1}{\sqrt t} \int d\eta'_I\left(e^{-2  i t(a\zeta'_I +
{\zeta'_I}^3)}\right)^{\wedge_{\zeta'_I}}(-\eta'_I) g^{\wedge_{\zeta'_I}}( 0,\eta'_I) + |u_0|_{\mathfrak M^{3,3}} \mathcal O(\frac{1}{t^\frac 32})\label{E:u-11-FirstReduction}\\
= & - \frac{4}{ (6 t)^\frac 13}\frac{1}{\sqrt t}\int  d\eta'_I\ Ai\left(\frac{ (2  t)^\frac 23}{\sqrt[3]{3}} (a-\frac{ \pi\eta'_I}{ t})\right)  g ^{\wedge_{\zeta'_I}}( 0  , \eta'_I) + |u_0|_{\mathfrak M^{3,3}}  \mathcal O(\frac{1}{t^\frac32}) . \nonumber
\end{align}

Moreover, let 
\be\label {E:eta-t}
z =\frac{ (2  t)^\frac 23}{\sqrt[3]{3}} (a-\frac{\pi \eta'_I}{ t}),\quad \eta'_I(t)=\frac t\pi(a+\frac{\sqrt[3]{3}}{(2  t)^{\frac 23}}).
\ee Note that $\eta'_I<-\frac tCr^2$ for $\eta'_I<\eta'_I(t)$ and $t\gg 1$. Hence from  \eqref{E:i-L1}, as $t\to\infty$,
\be\label{E:g-four}
\begin{split}
 |\theta(-\frac{t}{C}r^2-\eta'_I)g ^{\wedge_{\zeta'_I}}(0,\eta'_I)|_{L^1(d\eta'_I)} \le & C  |u_0|_{\mathfrak M^{3,3}} t^{- 5/2}  , \\  |(\frac{ {\eta'_I} }t )^j g ^{\wedge_{\zeta'_I}}(0,\eta'_I)|_{L^1(d\eta'_I)}\le& C |u_0|_{\mathfrak M^{3,3}} t^{-j},\ 0\le j\le 2  ,
\end{split}
\ee where $\theta$ is the Heaviside function. Consequently,  \eqref{E:u-11-FirstReduction} implies
\begin{align}
   u_{1,1}(x)
= &  - \frac{4}{ (6 t)^\frac 13}\frac{1}{\sqrt t}\int_{\eta'_I>\eta'_I(t)}  d\eta'_I\ Ai\left(\frac{ (2  t)^\frac 23}{\sqrt[3]{3}} (a-\frac{ \pi\eta'_I}{ t})\right)  g ^{\wedge_{\zeta'_I}}( 0  , \eta'_I) + |u_0|_{\mathfrak M^{3,3}} \mathcal O({\color{black}t^{- 3/2}}).\label{E:u-11-2ndReduction}
\end{align}

Finally, for $\eta'_I>\eta'_I(t)$, we have $z<-1$ and the Airy analysis \eqref{E:airy-asy} applies  to \eqref{E:airy-app}. Along with the mean value theorem and \eqref{E:g-four}, yields
\begin{align*}
&  u_{1,1}(x)\\
=& \frac{-2}{ (6  t)^\frac 13\sqrt t}  \int_{\eta'_I>\eta'_I(t)} d\eta'_I
 \frac{e^{i(\frac 23\left|\frac{(2 t)^\frac 23}{\sqrt[3]{3}} |a-\frac{\pi \eta'_I}{ t}|\right|^\frac 32-\frac \pi 4)}+c.c.}{\sqrt\pi\left[\frac{(2  t)^\frac 23}{\sqrt[3]{3}} |a-\frac{\pi \eta'_I}{ t}|\right]^\frac 14}g ^{\wedge_{\zeta'_I}}( 0   , \eta'_I) 
 + |u_0|_{\mathfrak M^{3,3}} \mathcal O({\color{black}t^{- 3/2}})\nonumber\\
=&-\frac{2}{ (6  t)^\frac 12}\frac{1}{\sqrt{\pi r t}} \int_{\eta'_I>\eta'_I(t)} d\eta'_I
   \left[e^{i(  4   t  r^3 (1-\frac 32\frac{ \pi \eta'_I}{ t{\color{black} |a|}})+\mathcal O(\frac { {\eta' _I}^2}t) -\frac \pi 4)}+c.c.\right]g ^{\wedge_{\zeta'_I}}( 0, \eta'_I)+|u_0|_{\mathfrak M^{3,3}}  \mathcal O({\color{black}t^{- 3/2}})\nonumber\\
 =& \frac{2  i e^{i 4t  r^3    
 }}{3t  }    \widetilde s_c ( + ir   )
-\frac{2  ie^{-i   4t  r^3    
  }}{ 3 t   }    \widetilde s_c(  - ir  )
 +|u_0|_{\mathfrak M^{3,3}}  \mathcal O({\color{black}t^{- 3/2}})  \nonumber
\end{align*}where c.c. denotes the complex conjugate of the preceding number.
Therefore, we prove \eqref{E:u-11-}.

\vskip.1in
\noindent $\blacktriangleright$ $\underline{\textbf{ Proof  of $ a>+\delta> 0$ :}}$ Using  $\color{black}a>\delta>0$, Lemma  \ref{L:sd-der},   the factor $(\overline\zeta'-\zeta')$, \eqref{E:intro-s-c-ana-c-0}, and integration by parts,
\begin{align}
 u_{1,1}(x)=  &-  \frac 1{ 3\pi t } \int d\zeta'_R\    \int d \zeta'_I\ e^{ 2 \pi i t S_0(a; \zeta')} \psi_{r,0}(\zeta'_I)
\partial_{ \zeta'_R} \left(\frac{1}{\zeta_R'} \psi_{r,r}(\zeta'_R) \widetilde{  s}_c(\zeta'  ) \right). \label{E:u-+-dec}
\nonumber
\end{align}
  Let $g_+= \psi_{r,0}(\zeta'_I)\partial_{ \zeta'_R} \left(\frac{1}{\zeta_R'} \psi_{r,r}(\zeta'_R) \widetilde{  s}_c(\zeta'  ) \right)$. Applying  Lemma  \ref{L:sd-der},   the factor $\chi$,  and $a>\delta>0$   (uniformly bounded below),                                                    
  \begin{gather*}
|g_+ |_{L^2(d \zeta'_I)}, \  |
\partial_{\zeta'_I}g_+ |_{L^2(  d\zeta'_I)}, \  |
 g_+^{\wedge_{\zeta'_I}} |_{L^1(  d\eta'_I)}<C|u_0|_{\mathfrak M^{1,1}}.
\end{gather*} Therefore, taking the Fourier transform,  and applying the Airy function analysis in the above proof, we obtain:
\be \label{E:u-11-asym---plancherel-0} 
   |u_{1,1} |
 \le    |   \frac{ C  }{    t ^\frac 43} \int d\zeta'_R \int  d\eta'_I \, Ai\left(\frac{ 2t(a-3{\zeta'_R}^2)-2\pi\eta'_I}{\sqrt[3]{6t}}  \right)   g_+ ^{\wedge_{\zeta'_I}}( \zeta_R'  , \eta'_I) | \le C|u_0|_{\mathfrak M^{1,1}}t^{-\frac 43 }.  
\ee

\end{proof}

We conclude this section by:
\begin{theorem}\label{T:u-1} Suppose that \eqref{E:intro-ID} holds. Then, as $t\to +\infty$,
\begin{align*}
 \blacktriangleright\ & u_{1 }(x)=\frac{2  i e^{i 4      t  r^3    
  }}{3t  }     s_c (+\frac{t_2}{3}+ ir   )
-\frac{2  ie^{-i   4  t  r^3    
  }}{ 3 t   }     s_c( +\frac{t_2}{3}- ir  )
 +|u_0|_{\mathfrak M^{3,3}} o( t^{-1 }),\  \textit{$ a<-\delta< 0$,}\\
 \blacktriangleright\ & u_{1 }(x)= |u_0|_{\mathfrak M^{1,1}}o( t^{-1 }), \qquad  \textit{ for $ a>+\delta> 0$.} 
  \end{align*}
\end{theorem}

\section{Long time asymptotics of    $u_{2,0}(x)$}\label{S:m}

Throughout this subsection,        $a$, $r$, $t_i$, $t$ are as defined in    \eqref{E:real-ima}  and \eqref{E:a}-\eqref{E:stationary--}. To adapt the approach of $u_1$ in Section \ref{S:sigma--} to study the asymptotics of $u_{2,0}$, 
 it reduce to analysing  $  \widetilde m(x,\overline\lambda')-1 $ and   $\nabla \widetilde m(x,\overline\lambda')$. From
\begin{align}
 m     =& 1+ \mathcal CT1+\cdots+  \left(\mathcal CT\right)^n1+\cdots,\label{E:m-CIE-t-AYF}  
\end{align} we are led to study the Cauchy integrals $(\widetilde{\left(\mathcal CT\right)^n1})(x,\overline\lambda')$ and their derivatives.

The first basic estimates about the Cauchy integrals is:
\begin{lemma}\label{L:wickerhauser}\cite{W87} {\color{black}Suppose the assumption of Theorem \ref{T:u}} holds.
\[
|  \widetilde{\mathcal CTf }|_{L^\infty}\le   C |u_0|_{\mathfrak M^{0,0}}|f|_{L^\infty}  .
\]

\end{lemma}
\begin{proof} The proof follows from \eqref{E:real-ima-new}, \eqref{E:scatter-new}, and 
\begin{align}
    \widetilde{\mathcal CTf }  = &   -\frac{1}{2\pi i}\iint\frac{e^{2\pi i tS_0 } \widetilde{  s}_c(\zeta ' ) }{\zeta'-\lambda'}\widetilde f(\zeta ' )  d\overline\zeta'\wedge d\zeta'   \nonumber\\
=&     -  \iint \frac{ e^{2\pi i tS_0(a;\zeta'(\xi'_1,\xi'_2))}  \widetilde{  s}_c(  \zeta' ) }  {p_{\lambda'}(\xi'_1,\xi_2')} \widetilde f(\zeta'(\xi'_1 ,\xi'_2 ) )  d\xi'_1  d\xi'_2  .  \nonumber
\end{align}  Therefore the lemma follows from \eqref{E:intro-s-c-ana-c-0} and \eqref{E:p}.

\end{proof}

We raise the $\xi'_1$-decay to assure the integrability (cf. Lemma \ref{L:CIO}, measure size estimates of $R_{in,\gtrless,-}$ as $\xi'_1\to\infty$, and \eqref{E:large-43}, reduction to estimates on $\Sigma_{<}$). Setting
\be\label{E:potential-0}
[(1+\frac{1}{16\pi^4}\partial_{ x_1} ^4)(u_0m_0) ]^\wedge=\mathbf s(\zeta),\qquad \mathbf v (x'_1,x'_2; t_2 )\equiv  \mathbf s  ^{\vee_{\xi_1 ,\xi_2 }}( x'_1-\frac {2t_2}3x_2',x'_2  ),
\ee applying the Fourier  theory, we formally   obtain     
\begin{align} 
& \widetilde{\mathcal CT1 }(t_1,t_2,t; \overline\lambda') 
=   -\frac{1}{2\pi i}\iint\frac{e^{2\pi i tS_0 } \widetilde{\mathbf s} (\zeta ' ) }{(1+\xi'^4_1)(\zeta'- \overline\lambda')}  d\overline\zeta'\wedge d\zeta'   \label{E:CT-ABF}\\
 = &      -\iint  [\,\frac{e^{+2\pi it(\pi^2{\xi'_1}^3-\frac 34\frac{{\xi'_2}^2}{\xi'_1})}}{(1+\xi'^4_1)p_{\overline \lambda'}(\xi'_1,\xi'_2)}\,]^{\wedge_{\xi'_1,\xi'_2}}( x_1', x_2')  \widetilde{\mathbf s}^{\vee_{\xi_1',\xi_2'}} ( x'_1 ,x'_2 )  dx_1'dx_2'\nonumber\\
 =&-\frac{1}{2\pi i}\iint dx_1'dx_2'\mathbf v (x'_1,x'_2; t_2 )  \iint d{\xi'_1}   d{\xi'_2} \frac{e^{ 2\pi i [+t(\pi^2{\xi'_1}^3-\frac 34\frac{{\xi'_2}^2}{\xi'_1})+[(ta-x_1'){\xi'_1}-x_2'{\xi'_2}] ]}}{(1+\xi'^4_1)({\xi'_2} -(2\pi i{\xi'_1}^2+2{\xi'_1}\overline \lambda')) }.   \nonumber 
\end{align}

\begin{lemma}\label{L:potential}
\[
|\mathbf v |_{L^1(dx_1dx_2)}<C|u_0|_{\mathfrak M^{2,6}}.
\]
\end{lemma}
\begin{proof}
Taking derivatives of \eqref{E:intro-Lax} and applying \eqref{E:AYF-G-infty-G}, we obtain
\[
| \left[\frac{  m_0 (x_1,x_2,\zeta )   }{(1+ x_1^2 + x_2^2 )}\right]^{\wedge_{x_1,x_2}}   |_{L^1(d\xi_1d\xi_2)}<  | \left[\frac{  m_0 (x_1,x_2,\zeta ) }{(1+ x_1^2 + x_2^2 )}\right]^{\wedge_{x_1,x_2}} (1+|\xi_1|+|\xi_2|)^2   |  <  C|u_0|_{\mathfrak M^{0,2}}  . 
\]  

Hence,
\begin{align*}
&|\{ [u_0(x_1,x_2)(1+ x_1^2 + x_2^2 )\frac{  m_0 (x_1,x_2,\zeta ) }{(1+ x_1^2 + x_2^2 )} ]^{\wedge_{x_1,x_2}} \}^{\vee_{\xi_1,\xi_2}}|_{L^1(d\xi_1d\xi_2)}\\
=&|\{\left[u_0(x_1,x_2)(1+ x_1^2 + x_2^2 )\right]^{\wedge_{x_1,x_2}}
\ast   \left[\frac{   m_0 (x_1,x_2,\zeta )  }{(1+ x_1^2 + x_2^2 )} \right]^{\wedge_{x_1,x_2}}  \}^{\vee_{\xi_1,\xi_2}}|_{L^1(d\xi_1d\xi_2)} \\
=&|u_0(x'_1,x'_2)(1+ {x'_1}^2 + {x'_2}^2 ) \{ [\frac{  m_0 (x_1,x_2,\zeta )  }{(1+ x_1^2 + x_2^2 )} ]^{\wedge_{x_1,x_2}}  \}^{\vee_{\xi_1,\xi_2}} |_{L^1(d\xi_1d\xi_2)} \\
\le &C|u_0|_{\mathfrak M^{2 ,2}}, 
\end{align*}and
\[|\mathbf v(x'_1,x'_2)|_{L^1(dx_1dx_2)} =|\{ [(1+ \frac{1}{16\pi^4}\partial_{ x_1} ^4)u_0 m_0 ]^{\wedge_{x_1,x_2}} \}^{\vee_{\xi_1,\xi_2}}|_{L^1(d\xi1 d\xi_2)}  \le C|u_0|_{\mathfrak M^{2,6}}. 
\]

\end{proof}

The plan of this section is: We will provide notations and introduce the {\it ABF Cauchy integrals (Ablowitz-Bar Yaacov-Fokas Cauchy integrals)} in the   Subsection \ref{SS:technical}.  In Subsection \ref{SS:CI-0}, through stationary phase analysis and deformation arguments, we will rigorously establish the existence of the right-hand side of  \eqref{E:CT-ABF} and  the corresponding analysis for $\widetilde{(\mathcal CT)^n1}$, and obtain a {reduced representation formula of $u_{2,0}$} (see Theorem \ref{T:first-reduction}), which reduces the problem to the asymptotic analysis of the ABF Cauchy integrals. The asymptoics of $u_{2,0}$  are established in Subsections  \ref{SS:u+2} and \ref{SS:u-2} by completing the asymptotic analysis of ABF Cauchy integrals.

\subsection{ABF Cauchy integrals}\label{SS:technical}
Denote 
\be\label{E:potential-0-1}
\begin{gathered}
p_{\overline \lambda'}(\xi'_1,\xi'_2)=-2\pi i({\xi'_2} -z_0 ),\qquad\
z_0\equiv z_0(\xi'_1,\lambda')=2\pi i{\xi'_1}^2+2{\xi'_1}\overline \lambda',
\end{gathered}
\ee and the phase function  of \eqref{E:CT-ABF} as
\begin{gather}
\wp(\xi_1',\xi_2';a,t,x_1',x_2')= + (\pi^2{\xi'_1}^3-\frac 34\frac{{\xi'_2}^2}{\xi'_1})+[( a-\frac{x_1'}t){\xi'_1}-\frac{x_2'}t{\xi'_2}]\equiv  \wp_1+\wp_2,\label{E:wp}\\
 \wp_1
=     \pi^2{\xi'_1}^3+( a-\frac{x_1'}t){\xi'_1}  ,\quad\wp_2=-   [\frac 34\frac{{\xi'_2}^2}{\xi'_1} + \frac{x_2'}t{\xi'_2} ],\label{E:1-wp} \\
 \partial_{\xi'_1}\wp_1=3\pi^2{\xi_1'}^2 +a-\frac{x_1'}t, \quad\partial_{\xi'_2}\wp_2=   
 -\frac 32  \frac{\xi_2'}{\xi_1'}-\frac{x_2'}t  .\label{E:3-wp}
\end{gather}

As $ \frac {x_j}t\to 0$,  $\wp\to \mathbb S_0(\xi'_1,\xi'_2)$, 
  and
\begin{align}
\bullet\ &\textit{the stationary points of $\wp_1$ tend to }  \xi'_1 
 =   \pm \frac{r}{ \pi}  \textit{ when $a<-\delta<0$,}\label{E:wp-2}\\
&\textit{and there are no stationary points when $a>+\delta>0$. }
\nonumber\\
\bullet\ &  | \partial_{\xi'_1}\wp_1| \ge\frac { 1}C ({\xi'_1}^2 +|a|) \quad \textit{on } \left\{ 
{\begin{array}{ll}
  \mathbf R^2,&\textit{ for $a>+\delta>0$,}\\
  ||\xi'_1|-  \frac {r}{ \pi} |\ge   r_\ast  ,&\textit{ for $a<-\delta<0$,} 
\end{array}}
\right.\label{E:wp-3}
\end{align}where 
\be\label{E:wp-4}
r_\ast=
\left\{ 
{\begin{array}{ll}
1,&\textit{ for $a>+\delta>0$,}\\
 \min(\frac {r}{10\pi},1),&\textit{ for $a<-\delta<0$. } 
\end{array}}
\right.
\ee  

Let
\be\label{E:large-15}
\begin{gathered}
\rho_1=t^{-\frac 12+\epsilon_1},  \quad \rho_2=t^{-\frac 12+ \epsilon_2 },\quad
    \epsilon_2=2\epsilon_1>0, \\
    \epsilon_\flat=|u_0|_{\color{black}\mathfrak M^{2,6}},\qquad \epsilon_\sharp=|u_0|_{\mathfrak M^{ 4, 8}}.
\end{gathered} 
\ee   Denote  
\be\label{E:wp-5} 
\begin{gathered}
\theta_{s,w_0}(\xi'_1)= \theta(s-\left||\xi'_1|-|w_0|\right|),\\
   \Xi (\xi'_1) =  
\left\{ 
{\begin{array}{ll}
 \theta_{r_\ast,0}(\xi'_1),&\textit{ for $a>+\delta>0$,}\\
 \theta_{r_\ast,0}(\xi'_1)+ \theta_{r_\ast,\frac r\pi}(\xi'_1) ,&\textit{ for $a<-\delta<0$, } 
\end{array}}
\right.\quad
   \Xi^c (\xi'_1) =  
1-\Xi(\xi'_1) , 
\end{gathered}\ee
 and 
\begin{align}
 \Sigma_{ \lessgtr}=&\{(x_1',x_2'):\ \textit{$|x_j'|\lessgtr   t^{\frac 12}$ for   $  j= 1$ $\left\{\begin{array}{c}\textit{and}\\
 \textit{or}\end{array}\right\}$ $\,2$}\},\label{E:wp-6}\\
\mathfrak R (\lambda'_R,\xi'_1)=&[1-\theta_{\rho_1,0}(\lambda'_R)][1- \theta_{r_\ast,0}(\xi'_1)]+[1-\theta_{|\xi'_1|^{1/2}\rho_1,0}(\lambda'_R)]\theta_{r_\ast,0}(\xi'_1).\nonumber
\end{align}

Inspired by \cite{AYF83}, we introduce 
\begin{definition}[Ablowitz-Bar Yaacov-Fokas (ABF) Cauchy integrals ]\label{D:CIO-1}Define $[\mathfrak C \mathfrak T1]^{0,(1)}$ and $\mathfrak C \mathfrak T_{(1)}$ by:
\begin{align}
 & e^{ i\pi tS_0(a;\overline \lambda') }[\mathfrak C \mathfrak T1]^{0,(1)}
(t_1,t_2,t;\lambda')  \label{E:CIO-ABF-1}\\
&\quad\equiv   e^{- i\pi tS_0(a;\lambda') } \iint_{ \Sigma_{ <}} dx_1'dx_2' \   \mathbf v(x_1',x_2';t_2)  e^{-i\lambda'_I(x_1'+2\lambda'_Rx_2')}\nonumber\\
& \quad\times\int d\xi_1''   e^{2\pi i t \mathfrak G^\sharp}\Theta(t;\lambda';x_2';\xi_1'')\mathcal F(t;\lambda';x_2';\xi_1'')\equiv e^{- i\pi tS_0(a;\lambda') } \mathfrak C \mathfrak T_{(1)} 1 , \nonumber
\end{align} where 
\begin{align}
&\mathfrak S ^\sharp(a,t;x_1',x_2';\lambda'_R;\xi_1'')=\mathfrak S(a;\lambda_R';\xi_1'')-\frac {(x_1'+2\lambda'_R x_2')\xi_1''}t,\nonumber\\
&\mathfrak S(a;\lambda_R';\xi_1'')=4 \pi^2\xi_1''^3 +(a-3{\lambda'}_R^2  )\xi_1'',\label{E:CIO-ABF-2}\\
&\Theta=\Theta_{pole}\Theta_{amp}\mathfrak R(\lambda'_R,\xi''_1+\frac{\lambda'_I}{2\pi}),\nonumber\\
& \Theta_{amp}(t;\lambda';x_2';\xi_1'')=\sgn(\lambda'_R) \theta(-( x_2'+3  t\lambda'_R)(\xi_1''-\frac{\lambda'_I}{2\pi})  (\xi_1''+\frac{\lambda'_I}{2\pi}  )),\nonumber\\
&\Theta_{pole}(\lambda',\xi''_1)=\theta(\frac\pi 8-|\arg(z_0(\xi''_1+\frac{\lambda'}{2\pi}))|)+\theta(\frac\pi 8-\left|\pi-|\arg(z_0(\xi''_1+\frac{\lambda'}{2\pi}))|\right|),\nonumber\\
&\mathcal F(t;\lambda';x_2';\xi_1'')= \frac {1}{1+(\xi''_1+\frac{\lambda'_I}{2\pi})^4}  e^{4\pi^2 (  x_2'+3  t\lambda'_R)(\xi_1''-\frac{\lambda'_I}{2\pi})  (\xi_1''+\frac{\lambda'_I}{2\pi}  )}.\nonumber
\end{align}  
\end{definition} 

The ABF Cauchy integral will   play  a key role in the reduced representation formula for $u_{2,0}$.   Note that the cut-off function  $\Theta_{amp}$ identifies the region where $\mathcal F$ is non-increasing, whereas the cut-off function 
$\Theta_{pole}$ specifies the region in which the pole  $z_0$  of the CIO lies between the real axis   and the lines passing through the origin at angles $\pm\frac\pi 8$ in the complex plane $\xi'_2\in\CC$.

\begin{lemma}\label{L:CIO} 
\[
|[\mathfrak C \mathfrak T1]^{0,(1)} |<C\epsilon_\flat .
\]

\end{lemma}

\begin{proof} Follows from Definition \ref{D:CIO-1}, $|\mathcal F|_{L^1(d\xi'')}<C$, and  Lemma \ref{L:potential}.  

\end{proof}

\begin{definition}\label{D:stationary-g-frak}
Let the phase function $\mathfrak S (a;\lambda_R';\xi_1'')$ be defined by \eqref{E:CIO-ABF-2}. In view of 
\be\label{E:d-tilde-s-0}
\begin{split}
\partial_{\xi_1''}\mathfrak  S(a;\lambda_R';\xi_1'')=&+12\pi^2 {\xi_1''}^2+(a-3{\lambda'}_R^2  ),\\
\partial_{\xi_1''}^2\mathfrak  S(a;\lambda_R';\xi_1'')=&+24\pi^2 {\xi_1''},
\end{split}\ee  we have the definition for stationary points :
\begin{align*}
&\textit{If $a-3{\lambda'}_R^2>0$, there are no stationary points of $\mathfrak S$.}\\
&\textit{If $a-3{\lambda'}_R^2<0$, there are two stationary points  ${\xi_1''}=\pm b\gtrless 0$, $b^2=\frac{ 3{\lambda'}_R^2-a}{12\pi^2}$, of $\mathfrak S$.}
\end{align*}

\end{definition}

\begin{lemma}\label{L:cio-0}
For $\epsilon>0$,
\[ 
   |  [\mathfrak C \mathfrak T1]^{0,(1)}|\le  {\color{black}  C_\epsilon {\color{black} \epsilon_\sharp} t^{-\frac 12+\epsilon} } \quad \textit{as $t\to\infty$}.\]
  
\end{lemma}

\begin{proof}  We divide the analysis into two regimes:

   $\blacktriangleright$ {\color{black}$|\lambda_R'|\ge r/2$: Decompose
\be\label{E:ge-r}
\begin{split}
 |\mathfrak C\mathfrak T 1|\le &|\mathfrak C\mathfrak T  \theta(|x_2'|-t|\lambda_R'|)|+|\mathfrak C\mathfrak T [1-\theta( ||\xi_1''|-\frac{|\lambda_I'|}{2\pi} |-t^{ -\frac12+\epsilon})]\theta(t|\lambda_R'|-|x_2'|)|\\
 +&|\mathfrak C\mathfrak T \theta( ||\xi_1''|-\frac{|\lambda_I'|}{2\pi} |-t^{-\frac12+\epsilon})\theta(t|\lambda_R'|-|x_2'|)|\equiv I_1+I_2+I_3.
 \end{split}\ee
 
 Apparently, as $t\gg 1$, $I_1\equiv 0$ and $|I_2|\le Ct^{-\frac12+\epsilon}\epsilon_\flat $.   
Moreover, notice that 
\begin{multline*}
  \hskip.5in    \theta( |\lambda'_R|-r/2)\theta(t|\lambda'_R|-|x_2'|) \theta( ||\xi_1''|-\frac{|\lambda_I'|}{2\pi} |-t^{-\frac12+\epsilon})\ \\
  \times  |(  x_2'+3  t\lambda'_R)(\xi_1''-\frac{\lambda'_I}{2\pi})  (\xi_1''+\frac{\lambda'_I}{2\pi}  )|\ge Ct^{2\epsilon }.  \hskip.5in  
\end{multline*}Therefore, $|\Theta\mathcal F|_{L^1(d\xi_n'')}\sim C_\epsilon t^{-2}$ which implies $|I_3|= \epsilon_\flat o(t^{-1})$.}

  $\blacktriangleright$ $|\lambda_R'|\le r/2$: 
\begin{itemize}
\item  {\color{black}If $\color{black}a>+\delta>0$, then $|\partial_{\xi_1''}\mathfrak  S(a;\lambda_R';\xi_1'')|\ge \frac 1C ({\xi_1''}^2+a  )$.} Integration by parts with respect to $\xi''_1$, we obtain 
\be\label{E:cio-0-1}
 |   \mathfrak C\mathfrak T 1 |\le \mathcal O(t^{-1}){\color{black}\epsilon_\sharp} .
\ee 
\item If $a<-\delta<0$,  the stationary points $\pm b$ are well separated.  As a result, we obtain estimates for the measure  : 
\begin{align}
&{\color{black}|\Sigma_{t^{-1/2}}|\le C t^{-1/2} ,}\label{E:cio-0-2}
\end{align}  where $\Sigma_s(a;\lambda_R';\xi_1'')=\{\xi_1'': |\partial_{\xi_1''} \mathfrak  S(a;\lambda_R';\xi_1'')|\le s\}$. 
 Hence  using $|\lambda'_R|<r/2$, integration by parts, and  \eqref{E:cio-0-2},
\begin{align}
 |   \mathfrak C\mathfrak T 1 |
\le &{\color{black}|\mathfrak C\mathfrak T \theta(t^{-1/2}-|\partial_{\xi_1''} \mathfrak  S(a;\lambda_R';\xi_1'')|) | }+{\color{black}|\mathfrak C\mathfrak T\theta(|\partial_{\xi_1''} \mathfrak  S(a;\lambda_R';\xi_1'')|-t^{-1/2})) |}\label{E:cio-0-3}\\
\le &C t^{-1/2}\epsilon_\flat+\frac C{t }| \iint_{\Sigma_{<}} dx_1'dx_2' \  \mathbf v(x'_1,x'_2;t_2 ) e^{-i\lambda'_I(x_1'+2\lambda'_Rx_2')} 
\nonumber\\
\times& \int_{\xi_1''\in\Sigma_{t^{-1/2}}^c} d\xi_1''   e^{2\pi i t \mathfrak G }   \partial_{\xi_1''}   \frac{1}{ \partial_{\xi_1''}  \mathfrak S  }       \{  e^{-2\pi i (x_1'+2\lambda'_R x_2')\xi_1''}   \nonumber \\
\times&  \frac {\Theta}{1+(\xi''+\frac{\lambda'_I}{2\pi})^4}   e^{4\pi^2 (  x_2'+3  t\lambda'_R)(\xi_1''-\frac{\lambda'_I}{2\pi})  (\xi_1''+\frac{\lambda'_I}{2\pi}  )}\} |\le C {t^{-1/2}}{\color{black} \epsilon_\sharp}.\nonumber
\end{align} 
\end{itemize}

\end{proof}

\begin{lemma}\label{L:lambda-r}
Let $\kappa>0$. As $t\to\infty$,  
\begin{align}
|\partial_{\lambda'_I}\left[\mathfrak C \mathfrak T1\right]^{0, (1)}| \le&  C(1+|\lambda_R'|){\color{black} \epsilon_\sharp} ,      \label{E:lambda-i} \\
   \theta(|\lambda_R'|-1/\kappa) |
 \partial_{\lambda'_R}  \left[\mathfrak C \mathfrak T1\right]^{0, (1)}|  
\le  & C_{\kappa} (1+ |\lambda_I'| ){\color{black} \epsilon_\sharp}. \label{E:lambda-r}
\end{align}

Moreover,
\be
\textit{$[\mathfrak C \mathfrak T1]^{0,(1)}$ is holomorphic in $\lambda'_R\lambda'_I$ when $\lambda'_R\lambda'_I\ne 0$.}\label{E:CIO-holo}
\ee 
\end{lemma} 

\begin{proof}  In view of
\[
\begin{split}
\partial_{\lambda'_I}\theta(\frac\pi 8\mp\arctan\frac{\pi\xi'_1-\lambda'_I}{\lambda_R'})=&\pm\frac 1\pi\delta(\xi'_1-\frac{\lambda'_I\pm (\sqrt 2-1)\lambda'_R}{\pi}),
\\
\partial_{\lambda'_R}\theta(\frac\pi 8\mp\arctan\frac{\pi\xi'_1-\lambda'_I}{\lambda_R'})=&\frac {\sqrt 2-1}\pi\delta(\xi'_1-\frac{\lambda'_I\pm (\sqrt 2-1)\lambda'_R}{\pi}),
\end{split}
\] 
we obtain: 
\begin{align}
&|  \partial_{\lambda'_I}    \left[\mathfrak C \mathfrak T1\right]^{0,(1)}|_{L^\infty}\label{E:lambda-i-1}\\
\le& C |\iint_{ \Sigma_{<}} dx_{1 }'dx_{2 }' \      (x_{1 }'+2x_{2 }' \lambda'_R )     \mathbf v(x'_1,x'_2;t_2 )  e^{- i\lambda'_I(x_{1 }'+2\lambda'_Rx_{2 }')}  \int  d\xi_1''e^{ 2\pi i   t\mathfrak G^\sharp } \Theta\mathcal F     |\nonumber\\
+&C   \iint_{ \Sigma_{<}} dx_{1 }'dx_{2 }' \     |    \mathbf v(x'_1,x'_2;t_2 ) |  \nonumber\\
+&  C   \iint_{ \Sigma_{<}} dx_{1 }'dx_{2 }' |  \mathbf v(x'_1,x'_2;t_2 )  | |\int  d\xi_1''e^{ 2\pi i   t\mathfrak G^\sharp }  \Theta \lambda'_I(  x_{2 }'+3  t\lambda'_R)   \mathcal F  | \nonumber\\
\equiv & I _{1} + I _{2}+   I _{3}.\nonumber
\end{align}

Firstly,
\be\label{E:lambda-i-2}
\color{black}
  I  _{1}  
\le   C(1+|\lambda_R'|){\color{black} \epsilon_\sharp},\quad I _{2}\le C \epsilon_\flat.   \ee
For $ I _{3} $,  notice  that
{\color{black}\begin{align*}
  &  |\int_{\xi_1''\lambda'_I\gtrless 0} d\xi_1''\theta(-(x_{2 }'+3  t\lambda'_R)(\xi_1''-\frac{\lambda'_I}{2\pi})  (\xi_1''+\frac{\lambda'_I}{2\pi}  ))
     \lambda'_I(  x_{2 }'+3  t\lambda'_R)  e^{ 4\pi^2 (x_{2 }'+3  t\lambda'_R)(\xi_1''-\frac{\lambda'_I}{2\pi})  (\xi_1''+\frac{\lambda'_I}{2\pi}  )} |\\ 
     \le & 
  |\int_{\xi_1''\lambda'_I\gtrless 0} d\xi_1''\theta(\mp (x_{2 }'+3  t\lambda'_R)(\xi_1''\mp \frac{\lambda'_I}{2\pi})\lambda_I') )
     \lambda'_I(  x_{2 }'+3  t\lambda'_R) e^{ \pm 2\pi  (x_{2 }'+3  t\lambda'_R)(\xi_1''\mp \frac{\lambda'_I}{2\pi})\lambda_I'    } |\\
\le & 
 C |\int_{\xi_1''\lambda'_I\gtrless 0} d\xi_1''\theta(\mp (x_{2 }'+3  t\lambda'_R)(\xi_1''\mp \frac{\lambda'_I}{2\pi})\lambda_I')\partial_{\xi_1''} e^{ \pm 2\pi  (x_{2 }'+3  t\lambda'_R)(\xi_1''\mp \frac{\lambda'_I}{2\pi})\lambda_I'    } |
 \le C.  \nonumber  
\end{align*}}    Hence 
\be\label{E:lambda-i-3}  
   I _{ 3} 
\le   C \epsilon_\flat,
\ee and \eqref{E:lambda-i} is proved.

 To prove \eqref{E:lambda-r},     
\begin{align}
&{\color{black}|\partial_{\lambda'_R}  \mathfrak C \mathfrak T_{(1)}  |}  \nonumber\\
\le &  C   \iint_{ \Sigma_{<}} dx_{1 }'dx_{2 }' \      |x_{2 }'(|\lambda_{I}'|+ 1 )      \mathbf v(x'_1,x'_2;t_2 )   |   \theta(t|\lambda_R'|-|x_{2 }'|) \int  d\xi_1'' (1+|\xi''_1|)|e^{ 2\pi i   t\mathfrak G^\sharp } \Theta \mathcal F     |\nonumber\\
+&  C   \int dx_{1 }'   \     | t  \mathbf v(x'_{1 } ,\pm t\lambda_R';t_2 )  |   
+   C |\iint_{ \Sigma_{<}} dx_{1 }'dx_{2 }'            \mathbf v(x'_1,x'_2;t_2 )   e^{ i \lambda_I' (x_{1 }'+2\lambda'_Rx'_{2 } )}  \theta(t|\lambda_R'|-|x_{2 }'|)\nonumber\\
\times& \int  d\xi_1''e^{ 2\pi i   t\mathfrak G^\sharp }  \Theta \mathcal F    4\pi i  (x_{2 }'+3t\lambda'_R )  \xi_1''|
+   C |\iint_{ \Sigma_{<}} dx_{1 }'dx_{2 }'             \mathbf v(x'_1,x'_2;t_2 )  e^{- i \lambda_I' (x_{1 }'+2\lambda'_Rx'_{2 } )}  \nonumber\\
\times&\theta(t|\lambda_R'|-|x_{2 }'|)\int  d\xi_1''e^{ 2\pi i   t\mathfrak G^\sharp } \Theta  \mathcal F   t(\xi_1''-\frac{\lambda'_I}{2\pi}) (\xi_1''+\frac{\lambda'_I}{2\pi})  |\nonumber\\ 
\equiv&   II_{ 1} +II_{ 2}+  II_{ 3}+   II_{ 4}  . \nonumber
\end{align}

Hence 
\begin{gather}
   II_{ 1}  \le  C (1+ |\lambda_I'|)\epsilon_\sharp .\label{E:lambda-r-1}
\end{gather}

Besides,  
\begin{multline}\label{E:lambda-r-2}  
II_3\le    C  \iint_{ \Sigma_{<}} dx_1'dx_2'  | \mathbf v(x'_1,x'_2;t_2 ) |\,| \int  d\xi''_1e^{2\pi i   t\mathfrak G^\sharp }     \frac {\Theta }{1+(\xi''_1+\frac{\lambda'_I}{2\pi})^4}
   \\\times  \partial_{\xi''_1}  e^{4\pi^2 (  x'_{ 2 }+3  t\lambda'_R)(\xi''_1-\frac{\lambda'_I}{2\pi})  (\xi''_1+\frac{\lambda'_I}{2\pi}  )} |\le C\epsilon_\flat,
\end{multline}and,   using the factors $\theta(|\lambda_R'|-1/\kappa)$, $\theta(t|\lambda_R'|-|x_{2 }'|)$ and  writing $t$ as $\frac{t\lambda_R'}{\lambda_R'}$ in $II_2 $,   we obtain   the exponent $4\pi^2(x_{2 }'+3  t\lambda'_R)(\xi_1''-\frac{\lambda_I'}{2\pi}) (\xi_1''+\frac{\lambda_I'}{2\pi})$ is proportional to   $t(\xi''_1-\frac{\lambda_I'}{2\pi}) (\xi_1''+\frac{\lambda_I'}{2\pi})$ in $II_4 $ and
\be 
    \theta(|\lambda_R'|-1/\kappa)II_2,\quad \theta(|\lambda_R'|-1/\kappa) II_4   \le C_{\kappa}\epsilon_\flat.\label{E:lambda-r-3}  
\ee

The statement \eqref{E:CIO-holo} follows   by  using  $\lambda'_R\lambda'_I\ne 0$,     \eqref{E:lambda-r}, and the chain rule.
 
\end{proof}
The above lemma shows that taking  the derivatives of the ABF Cauchy integrals,   no matter how small a neighborhood is chosen around these points, the $(1 + |\lambda'|)\mathcal{O}(1)$ bounds on the $\lambda'$-derivatives of the Cauchy integrals cannot be improved. {\color{black}Moreover, higher derivatives of the Cauchy integrals correspond to higher orders in $t$.} This presents a fundamental obstruction to obtaining $o(t^{-1})$ estimates for $u_{2,0}$ and $u_{2,1}$ through our approach.
 
It will be shown in Lemmas \ref{L:CIO-large}, \ref{L:small}, and \ref{L:0-small} that $\left[\mathfrak C \mathfrak T 1\right]^{0,(1)}$  arises from  \[\textit{taking the residue at $\xi'_2=2\pi i\xi'^2_1+2\xi'_1\overline\lambda'$ and setting $\xi'_1=\xi''_1+\frac{\lambda'_I}{2\pi}$}.\] Using \eqref{E:real-ima-new} and
\be\label{E:aux}
\begin{gathered}
\left.\overline \zeta' _n\right|_{\xi'_{2,n}=2\pi i\xi'^2_{1,n}+2\xi'_{1,n}\overline\lambda' ,\ \xi'_{1,n}=\xi''_n+\frac{\lambda'_{I }}{2\pi}}=\lambda'_{R }+2\pi i\xi''_n,
\end{gathered}
\ee where we denote 
\be\label{E:zeta-n}
\zeta_n=\frac{\xi'_{2,n}}{2\xi'_{1,n}}-i\pi\xi'_{1,n} =\zeta'_{R,n}+i\zeta_{I,n}'
\ee the corresponding spectral variables of the $n$-th Cauchy integral in $\widetilde{\left(\mathcal CT\right)^n}1$.
We introduce
\begin{definition}[the ABF Cauchy integrals]\label{D:n-CIO}For $\color{black}n\ge 2$,    define
$\left[\mathfrak C \mathfrak T 1\right]^{0,(n)}$ and $\mathfrak C \mathfrak T_{(n)}$  by
\begin{align}
  &e^{\beta_ni\pi   tS_0(a;\overline \lambda')} \left[ \mathfrak C \mathfrak T 1 \right]^{0,(n)}  (t_1,t_2,t;\lambda') \nonumber\\ 
= & e^{-\beta_ni\pi   tS_0(a; \lambda')}\iint_{\Sigma_{<}}  dx_{1,n}'dx_{2,n}'\mathbf v(x'_{1,n},x'_{2,n};t_{2} ) e^{-\beta_n i\lambda'_I(x_{1,n}'+2\lambda'_Rx_{2,n}')}\nonumber\\
 \times &\int d\xi_n'' e^{\beta_n2\pi i t \mathfrak G^\sharp(a,t;x'_{1,n},x'_{2,n};\lambda_R';\xi_n'')}  
  \Theta (t;\lambda'; x_{2,n}';\xi_n'')\mathcal F^{(n)}  (t;\lambda'; x_{2,n}';\xi_n'') \left[ \mathfrak C \mathfrak T 1 \right]^{0,(n-1)} (t_1,t_2,t;   \overline\zeta'_n)\nonumber\\
  \equiv & e^{-\beta_ni\pi   tS_0(a; \lambda')}\mathfrak C \mathfrak T_{(n)}  \left[ \mathfrak C \mathfrak T 1 \right]^{0,(n-1)} (t_1,t_2,t;   {\lambda'_{R}}+2\pi i\xi_n''),\nonumber
\end{align}and
\begin{gather}
   \left[ \mathfrak C \mathfrak T 1 \right]^{0,(0)}=  1 , \  x_{1,1}'=x_1',\ x_{2,1}'=x_2',\nonumber\\
    \frac 12\le\beta_n=\frac 12(2-\beta_{n-1})\le 1 ,\ 
    \beta_1=1,\nonumber\\
\mathcal F^{(n)}(t;\lambda'; x_{2,n}';\xi_n'')= \frac {1}{1+(\xi''_n+\frac{\lambda'_I}{2\pi})^4}   e^{\beta_n4\pi^2 (  x_{2,n}'+3  t\lambda'_R)(\xi_n''-\frac{\lambda'_I}{2\pi})  (\xi_n''+\frac{\lambda'_I}{2\pi}  )} .\nonumber
\end{gather} 
\end{definition}

Applying \eqref{E:lambda-i} and adapting argument for proving Lemma \ref{L:CIO} and \ref{L:lambda-r}, inductively, 
\begin{lemma}\label{L:dCIO} Let $\epsilon, \, \kappa>0$, as $t> 1$,
\begin{gather}
|[\mathfrak C \mathfrak T1]^{0,(n)}| <C\epsilon_\flat^n,\qquad |[\mathfrak C \mathfrak T1]^{0,(n)}| <C_\epsilon \epsilon_\flat^{n-1 }\epsilon_\sharp t^{-\frac 12+\epsilon},\label{E:n-CIO-ABF-finite} \\
\textit{$[\mathfrak C \mathfrak T1]^{0,(n)}$ is holomorphic in $\lambda'_R\lambda'_I$ when $\lambda'_R\lambda'_I\ne 0$,}\label{E:n-CIO-ABF-holo}
\end{gather}
and
\begin{align}
|\partial_{\lambda'_I}\left[\mathfrak C \mathfrak T1\right]^{0, (n)}| \le&  C(1+|\lambda_R'|){\color{black}\epsilon_\flat^{n-1 }\epsilon_\sharp},  \label{E:n-lambda-i}  \\
\theta(|\lambda_R'|-1/\kappa) |
 \partial_{\lambda'_R} [\mathfrak C \mathfrak T1]^{0,(n)}] 
\le &   C_\kappa (1+ |\lambda_I'| ){\color{black}\epsilon_\flat^{n-1 }\epsilon_\sharp}. \label{E:n-lambda-r}
\end{align}

\end{lemma}

\subsection{The Cauchy integrals and a reduced representation formula for $u_{2,0}(x)$}\label{SS:CI-0}

Let $\mathbf v$, $\Xi^c $, $\theta_{s,w_0}$,  $z_0$, $\wp$,  $\rho_1,\,\rho_2,\,\epsilon_1,\,\epsilon_2,\,\epsilon_\flat$, and $\epsilon_\sharp$ be defined by \eqref{E:potential-0}, \eqref{E:potential-0-1}, \eqref{E:wp}, and \eqref{E:large-15}-\eqref{E:wp-5} throughout this subsection.

The main techniques for analyzing the CIOs are integration by parts away from the singularity $z_0$, measure-size estimates via affine transformation near $z_0$, and the deformation method. We first avoid applying integration by parts, in Lemmas \ref{L:CIO-large}, \ref{L:small}, and \ref{L:0-small}, because some terms exhibit decay rates significantly weaker than $\mathcal{O}(t^{-1})$ near $z_0$, where  a bootstrap argument is therefore necessary and integration by parts cannot be applied directly. In Lemmas \ref{L:u20-large}, \ref{L:u20-small}, and \ref{L:u20-0-small}, we implement the bootstrap argument and  integration by parts to obtain sharper estimates for $u_{2,0}$. As an application, we derive a reduced representation formula for $u_{2,0}$ at the end of this subsection, which leads to our subsequent analysis of the ABF Cauchy integrals in $\S$ \ref{SS:u+2} and $\S$ \ref{SS:u-2} for $a\gtrless \pm\delta\gtrless   0$ respectively.

\begin{lemma}\label{L:CIO-large} [$\textrm{Estimates of $\widetilde{\mathcal CT1 }$  on $\Xi^c$}$]     
\begin{align}
&-\frac 1{2\pi i}\iint dx_1'dx_2'  \mathbf v(x'_1,x'_2;t_2 ) \iint d{\xi'_1}d{\xi'_2} \frac{\Xi^c(\xi'_1)e^{ 2\pi i t\wp}}{(1+\xi'^4_1)({\xi'_2} -z_0) }\label{E:CIO-large}\\
&\hskip.8in=   e^{ -i\pi tS_0(a;\lambda') } \mathfrak C \mathfrak T_{(1)}     \Xi^c(\xi'_1)  +\epsilon_\flat\mathcal O_{\epsilon_1}   (  t^{ -\frac 12 +2\epsilon_2 })+\epsilon_\sharp \mathcal O(t^{-1}).\nonumber
\end{align}

\end{lemma} 
\begin{proof}

We will derive the estimate  by dividing the $(x'_1,x'_2)$  into $\Sigma_{\gtrless}$.

\noindent  {$\blacktriangleright$}  $ \underline{\textbf{ Step 1.  On the regime $\Sigma_{<} $ :}}$   
Decompose  
\begin{align}
 &\textit{LHS of \eqref{E:CIO-large}}   \label{E:large-16}\\
=&   \frac{-1}{2\pi i}  \iint_{\Sigma_{<}}  dx_1'dx_2' \mathbf v(x'_1,x'_2;t_2 ) \iint  d{\xi'_1}  d{\xi'_2} \frac{\Xi^c(\xi'_1) \frac{1}{1+\xi'^4_1}[1- \theta_{\rho_1,0}( \lambda'_R )] e^{ 2\pi i t \wp}}{{\xi'_2} -z_0 }\nonumber\\
+& \frac{ -1}{2\pi i}  \iint_{\Sigma_{<}}  dx_1'dx_2' \mathbf v(x'_1,x'_2;t_2 ) \iint  d{\xi'_1}  d{\xi'_2} \frac{\Xi^c(\xi'_1) \frac{1}{1+\xi'^4_1}\theta_{\rho_1,0}( \lambda'_R ) e^{ 2\pi i t \wp}}{{\xi'_2} -z_0 }   \nonumber\\
\equiv& R_{in,>}+  R_{in,<} .\nonumber
\end{align}

\begin{itemize}
\item [\bf{(A)}]      $ \underline{\textbf{ Estimates for  $R_{in,>}$ :}}$    Introduce the scaling
\be\label{E:large-17}
  \omega=\frac{\xi'_2}{ |\xi'_1| \rho_1 } , 
\ee   and decompose
\begin{align}
 R_{in,>}
=&\frac{-1}{2\pi i}  \iint_{\Sigma_{<}}  dx_1'dx_2' \mathbf v(x'_1,x'_2;t_2 ) \int  d{\xi'_1} \Xi^c(\xi'_1) \frac{1}{1+\xi'^4_1}e^{ 2\pi i t\wp_1 } \label{E:large-18}\\
\times&\int   d{\omega} [1-\theta_{\rho_1,0}( \lambda'_R )]\theta(|\omega|- \frac14 ) \frac{ e^{ 2\pi i [-\frac 34  {t {\xi'_1}\rho_1 ^2}     \omega ^2  -x_2'|\xi'_1|\rho_1\omega ]  }}{\omega - \frac{z_0}{|\xi'_1|\rho_1   }}\nonumber\\
-&\frac{ 1}{2\pi i}  \iint_{\Sigma_{<}}  dx_1'dx_2' \mathbf v(x'_1,x'_2;t_2 ) \int  d{\xi'_1}\Xi^c(\xi'_1) \frac{1}{1+\xi'^4_1}e^{ 2\pi i t\wp_1 } \nonumber\\
\times&\int   d{\omega} [1-\theta_{\rho_1,0}( \lambda'_R )]\theta(\frac14 -|\omega|) \frac{ e^{ 2\pi i [-\frac 34  {t {\xi'_1}\rho_1 ^2}     \omega ^2  -x_2'|\xi'_1|\rho_1\omega ]  }}{\omega - \frac{z_0}{|\xi'_1|\rho_1   }}\nonumber\\
\equiv &R_{in,>,+}+R_{in,>,-} .\nonumber
\end{align} 

\begin{itemize}
\item [\bf{(A1)}] {\bf Estimate for $R_{in,>,+}$:} In this case, estimates   will be derived via the deformation method.  
 Precisely, thanks to 
\begin{gather}
1\le|  \frac {z_0}{ |\xi'_1|  \rho_1 }  | ,\quad |t \xi'_1\rho_1^2  |\ge \frac 1Ct^{+2\epsilon_1},\quad |    {t \rho_1  }     |\ge  \frac 1C t^{\frac 12+ \epsilon_1}\gg t^{\frac12 }\ge   |   x_2'     | . \label{E:large-19} 
\end{gather}  

Therefore,   deforming the right hand side of \eqref{E:large-17} to   contours:
\begin{multline}\label{E:large-20}
\omega\mapsto    
 \left\{
 {\begin{array}{l}
  \{-\sgn(\xi')\frac \pi4\lessgtr\theta\lessgtr 0:  \frac14  e^{ i\theta}     \}\cup
  \{e^{-i\sgn(\xi_1')\frac\pi 4} s,  s> \frac14 \}, \\
  \{ \sgn(\xi')\pi\gtrless\theta\gtrless\sgn(\xi'_1)\frac{3\pi}4: \frac14  e^{ i\theta}  \}\cup
\{ e^{+i\sgn(\xi_1')  \frac{3\pi} 4 } s,   s> \frac14 \}, \end{array}}\right. \\
 {\begin{array}{l}
   \textit{when $ \omega>0$ and $|\arg(z_0)|\le \frac\pi 8$,}\\
    \textit{when $ \omega<0$ and $|\pi-|\arg(z_0)||\le \frac\pi 8$,}  
\end{array}} 
\end{multline}and
\begin{multline}\label{E:large-21}
\omega\mapsto   
 \left\{
 {\begin{array}{l}
  \{-\sgn(\xi')\frac \pi{16}\lessgtr\theta\lessgtr 0: \frac14 e^{ i\theta}     \}\cup
  \{e^{-i\sgn(\xi_1')\frac\pi {16}} s,  s>\frac14 \}, \\
   \{ \sgn(\xi')\pi\gtrless\theta\gtrless\sgn(\xi'_1)\frac{15\pi}{16}:  \frac14e^{ i\theta}  \}\cup
\{ e^{+i\sgn(\xi_1')  \frac{15\pi} {16} } s,   s>\frac14\}, \end{array}}\right. \\
 {\begin{array}{l}
   \textit{when $ \omega>0$ and $\arg(z_0)\ge \frac\pi 8$,}\\
   \textit{when $ \omega<0$ and $|\pi-|\arg(z_0)||\ge \frac\pi 8$;}  
\end{array}} 
\end{multline}applying the residue theorem, $\xi_1'=\xi_1''+\frac{\lambda'_I}{2\pi}$, 
\begin{align}
&2\pi i \left[(ta-x_1'){\xi'_1}+t\pi^2{\xi'_1}^3-x_2'{\xi'_2}-t\frac 34\frac{{\xi'_2}^2}{\xi'_1}\right]_{{\xi'_2}=2\pi i{\xi'_1}^2+2{\xi'_1}\overline\lambda'} \label{E:large-22}\\
=&2\pi i   t[4 \pi^2\xi_1''^3 +(
a-3{\lambda'}_R^2  )\xi_1''+\frac{\lambda'_I}{2\pi}(a-3{\lambda'}_R^2+{\lambda'}_I^2)]\nonumber \\
-&2\pi i(   x_1'+ 2\lambda'_Rx_2')(\xi_1''+\frac{\lambda'_I}{2\pi})+4\pi^2 (  x_2'+3  t\lambda'_R)(\xi_1''-\frac{\lambda'_I}{2\pi})  (\xi_1''+\frac{\lambda'_I}{2\pi}  ),   \nonumber 
\end{align}Definition \ref{D:CIO-1}, and Lemma \ref{L:CIO}, we obtain:
\begin{itemize}
\item When 
\be\label{E:large-23}
\begin{gathered}
(\xi_1''+\frac{\lambda'_I}{2\pi}) > 0,\quad \left\{{\begin{array}{lll}0> \arg(z_0)>- \frac\pi 8,&(\xi_1''-\frac{\lambda'_I}{2\pi})<0,&\lambda'_R>0 ,  \\
+\pi > \arg(z_0)>+\pi- \frac\pi 8,&(\xi_1''-\frac{\lambda'_I}{2\pi})>0,&\lambda'_R< 0 , \end{array}}\right.    
\end{gathered}
\ee or
\be
(\xi_1''+\frac{\lambda'_I}{2\pi}) < 0,\quad \left\{{\begin{array}{lll}0< \arg(z_0)<+ \frac\pi 8,&(\xi_1''-\frac{\lambda'_I}{2\pi})<0,&\lambda'_R<0 , \\
-\pi< \arg(z_0)<-\pi+ \frac\pi 8,&(\xi_1''-\frac{\lambda'_I}{2\pi})>0,&\lambda'_R> 0 , \end{array}}\right.   
\ee then 
\be  \label{E:large-24}
  R_{in,>,+}
=   
 e^{- i\pi tS_0(a;\lambda') }\mathfrak C \mathfrak T_{(1)} \Xi^c(\xi'_1) 
+\epsilon_\flat o_{\epsilon_1}( t^{-1}) .  
\ee

\item Otherwise,
\begin{align}
&R_{in,>,+}
= 
  \epsilon_\flat o_{\epsilon_1}( t^{-1}) .   
\label{E:large-25}
\end{align} 
\end{itemize}   

\item [\bf{(A2)}] {\bf Estimate for $R_{in,>,-}$:} 
In this case, estimates   will be derived via  measure estimates.  Precisely, we apply   the change of variables \cite[Lemma 2.II]{W85, W87} \eqref{E:large-26}.
\be\label{E:large-26}
\begin{gathered}
\xi'_1= \frac 1{2\pi}(\eta_1+\lambda'_I),\quad
\xi'_2= \frac 1{2\pi}(\eta_2+2\lambda'_R\eta_1+2\lambda'_R\lambda'_I),\quad
d\xi_1'd\xi'_2=\frac 1{(2\pi)^2}d\eta_1 d\eta_2,
\end{gathered}
\ee  to transform
\begin{gather}
p_{\overline\lambda'}(\xi_1',\xi'_2)\mapsto -\eta_1^2+{\lambda'_I}^2-i\eta_2 .\label{E:large-27} 
\end{gather}  Due to the factor $ [1-\theta_{\rho_1,0}( \lambda'_R )]\theta(\frac14 -|\omega|)$ and \eqref{E:large-17}, we obtain 
\begin{gather}
|\eta_2+4\pi \xi'_1 \lambda'_R|=2\pi|  \xi'_2 | < \frac \pi 2|\xi'_1|\rho_1 
,  \label{E:large-28} \\
|\lambda'_R|\ge \rho_1,\quad \frac 1C\le|\xi'_1|.\label{E:large-28-1-(add)}
\end{gather}
WLOG (without loss of generality), setting
$ \lambda'_I=b_0\ge 0 $, 
\begin{itemize}
\item [(a)] if  $b_0\ge \frac{ r_\ast}4 $ then
\begin{align}
& |R_{in,>,-}|\le   C\epsilon_\flat [1-\theta_{\rho_1,0}(\lambda'_R )]  \label{E:large-29-(0)}\\
 & \times\int_{\frac 1C\le|\eta_1+\lambda'_I|= 2\pi |\xi'_1|} \frac{d\eta_1}{1+\xi'^4_1}    \int_{|\eta_2+4\pi\xi'_1\lambda'_R |\le \frac\pi 2|\xi'_1|\rho_1}  \frac{d\eta _2}{\sqrt{(\eta_1^2-b^2_0)^2+\eta_2 ^2}}\nonumber \\
& \le  C\epsilon_\flat [1-\theta_{\rho_1,0}(\lambda'_R )] \int_{  |\xi'_1|\ge\frac 1C  }\frac{d\xi'_1}{1+\xi'^4_1}        \frac{\rho_1|\xi'_1| }{\sqrt{(\eta_1^2-b^2_0)^2+ (\rho_1 \xi'_1) ^2 }} \nonumber\\
& \le C\epsilon_\flat [1-\theta_{\rho_1,0}(\lambda'_R )] \rho_1  \int_{|\xi'_1|\ge\frac 1C }  \frac{d\xi'_1 }{(1+\xi'^4_1)\sqrt{(\xi'_1-2b_0)^2 + \rho_1  ^2}} \nonumber\\
&=  \epsilon_\flat [1-\theta_{\rho_1,0}(\lambda'_R )]\mathcal O( t^{-\frac 12  +   \epsilon_1}).\nonumber
\end{align}

\item [(b)]    If  $b_0\le  \frac{ r_\ast}4 $, then, using \eqref{E:large-26}, \eqref{E:large-28}, we obtain $(\eta_1^2-b^2_0)^2>\frac 1C$ and  
\begin{align}
 |R_{in,>,-}|\le & C\epsilon_\flat [1-\theta_{\rho_1,0}(\lambda'_R )]\int_ {  |\eta_1+\lambda'_I|>\frac 1C}\frac{d\eta_1}{1+\xi'^4_1} \int_{|\eta_2+4\pi\xi'_1\lambda'_R|\le \frac \pi2|\xi'_1|\rho_1 }  \frac{d\eta _2}{\sqrt{(\eta_1^2-b^2_0)^2+\eta_2 ^2}}\label{E:large-29-(2)}\\
   =  & \epsilon_\flat  [1-\theta_{\rho_1,0}(\lambda'_R )]\mathcal O( t^{-\frac 12  +  \epsilon_1} ). \nonumber
\end{align}    

\end{itemize}

\end{itemize}

Combining (A1), and (a),(b) of (A2), we obtain 
\be\label{E:large-30}  
  R_{in,> }
=  
e^{- i\pi tS_0(a;\lambda') }\mathfrak C \mathfrak T_{(1)} \Xi^c(\xi'_1) + \epsilon_\flat [1-\theta_{\rho_1,0}(\lambda'_R )]\mathcal O(t^{-\frac 12 + \epsilon_1} ) +\epsilon_\flat o_{\epsilon_1}( t^{-1}). 
\ee

\item [\bf{(B)}] $ \underline{\textbf{ Estimates for  $R_{ in ,<}$ :}}$ Introduce the scaling
\be\label{E:large-33}
  \omega=\frac{\xi'_2}{|\xi'_1|\rho_2}, 
\ee and decompose
\begin{align}
&R_{in,<}\label{E:large-32}\\
=& \frac{- 1}{2\pi i}  \iint_{\Sigma_{<}}  dx_1'dx_2' \mathbf v(x'_1,x'_2;t_2 ) \iint  d{\xi'_1}  d{\xi'_2} \frac{\frac{1}{1+\xi'^4_1} \Xi^c(\xi'_1)\theta_{\rho_1,0}(\lambda'_R ) e^{ 2\pi i t \wp}}{{\xi'_2} -z_0 }\theta(|\omega |-\frac 14)  \nonumber\\
-& \frac{  1}{2\pi i}  \iint_{\Sigma_{<}}  dx_1'dx_2' \mathbf v(x'_1,x'_2;t_2 ) \iint  d{\xi'_1}  d{\xi'_2} \frac{ \frac{1}{1+\xi'^4_1}\Xi^c(\xi'_1) \theta_{\rho_1,0}(\lambda'_R ) e^{ 2\pi i t \wp}}{{\xi'_2} -z_0}\theta(\frac 14-|\omega |)  \nonumber\\
\equiv& R_{ in,<,+  }+R_{ in,<,- }.\nonumber
\end{align}

\end{itemize}

\begin{itemize}
\item [\bf{(B1)}] $\textbf{Estimate for  $R_{in,<,+ }$:}$ In this case, estimates   will be derived by the deformation method.  
 Hence  
\begin{gather}
\int  d{\xi'_2}\frac{\theta(|\frac{\xi_2'}{\xi_1'\rho_2} |-\frac 14) e^{ 2\pi i (-t  \frac 34\frac{{\xi'_2}^2}{\xi'_1} -x_2'{\xi'_2})  }}{{\xi'_2} -z_0}
=\int _{|\omega|>\frac 14} d{\omega}\frac{ e^{ 2\pi i (-\frac 34{\xi'_1} {t\rho_2^2}     \omega ^2  -x_2'  |\xi'_1|\rho_2\omega )  }}{\omega - \frac{z_0}{|\xi'_1|\rho_2  }},\label{E:large-34} 
\end{gather}with
\be
  |  \frac {\textit{Re}(z_0)}{|\xi'_1| \rho_2}  | \ll 1,\quad t\rho_2^2 \ge  t^{+2 \epsilon_2},\quad t\rho_2  \ge t^{\frac 12+2\epsilon_2}\gg |x'_2|. \label{E:large-35}
\ee  

Therefore,   deforming the right hand side of \eqref{E:large-34} to   contours:
\be \label{E:large-36}
\omega\mapsto    
 \left\{
 {\begin{array}{ll}
  \{-\sgn(\xi')\frac \pi4\lessgtr\theta\lessgtr 0:   \frac 14e^{ i\theta}     \}\cup
  \{e^{-i\sgn(\xi_1')\frac\pi 4} s,  s>\frac 14\},&\textit{for $\omega>0$}, \\
  \{ \sgn(\xi')\pi\gtrless\theta\gtrless\sgn(\xi'_1)\frac{3\pi}4:   \frac 14e^{ i\theta}  \}\cup
\{ e^{+i\sgn(\xi_1')  \frac{3\pi} 4 } s,   s>\frac 14\},&\textit{for $\omega<0$,}  \end{array}}\right. 
\ee and, thanks to \eqref{E:large-33},
\be\label{E:large-37}
  \frac {z_0}{|\xi'_1| \rho_2 } \notin \textit{the region enclosed by contours defined above,} 
\ee  
we obtain
\be\label{E:large-38}
 R_{in,<,+ }= \epsilon_\flat o_{\epsilon_1}(t^{-1}).
\ee

\item [\bf{(B2)}] $\textbf{Estimate for  $R_{in,<,-}$:}$ In this case, estimates   will be derived via   measure estimates. WLOG, assuming $\lambda'_I=b_0\ge 0$ and applying   the change of variables \cite[Lemma 2.II]{W85, W87} \eqref{E:large-26}, 
\be \label{E:large-39}
\begin{gathered}
  |\eta_2+4\pi\xi'_1\lambda'_R|=2\pi|  \xi'_2 | \le \frac{\pi}{2} |\xi'_1|\rho_2,\\
 |\lambda'_R|\le\rho_1,\quad  \frac 1C\le |\xi'_1|. 
   \end{gathered} 
\ee Moreover,    we also need
\begin{gather}
\int_0^L\frac{1}{\sqrt{x^2+A^2}}dx=\texttt{arcsinh}\frac LA ,\qquad \texttt{arcsinh}x= \log (x+\sqrt{1+x^2}),
\label{E:large-29-(1)} \\
\int_0^L\texttt{arcsinh}\frac Axdx=L\texttt{arcsinh}\frac AL+A\texttt{arcsinh}\frac LA.\label{E:large-29-(1)-(add)}
\end{gather}

\begin{itemize}
\item [(a)] If $b_0\ge \frac{r_\ast}4 $ then
\begin{align}
&|R_{in,<,-}|\label{E:large-29-(2)-(add)}\\
 \le&  C\epsilon_\flat  \theta_{\rho_1,0}(\lambda'_R )  \int  \frac{d\eta_1}{1+\xi'^4_1} \int_{|\eta_2 |\le 3\pi\rho_2|\xi'_1|}  \frac{d\eta _2}{\sqrt{(\eta_1^2-b^2_0)^2+\eta_2 ^2}} \nonumber\\
\le&C\epsilon_\flat  \theta_{\rho_1,0}(\lambda'_R ) (\  \int_{|\xi'_1-2b_0|<\frac{b_0}2}+\int_{\RR/|\xi'_1-2 b_0|<\frac{b_0}2})\frac{d\xi'_1}{1+\xi'^4_1} \texttt{arcsinh} \frac{3\pi\rho_2|\xi'_1|}{|\xi'_1||\xi'_1-2b _0|}    \nonumber\\
=&C\epsilon_\flat  \theta_{\rho_1,0}(\lambda'_R ) \mathcal O(t^{-\frac 12 +2\epsilon_2}).\nonumber
\end{align}

\item [(b)]    If  $b_0\le  \frac{r_\ast} 4$, then, using argument used for \eqref{E:large-29-(2)},  
\be\label{E:large-40-(add)}
  R_{in,<,-} = \epsilon_\flat   \theta_{\rho_1,0}(\lambda'_R ) \mathcal O( t^{-\frac 12  +  \epsilon_2} ).  
\ee 

\end{itemize} Therefore, 
\be\label{E:large-40}
R_{in,<,-}= \epsilon_\flat   \theta_{\rho_1,0}(\lambda'_R ) \mathcal O( t^{-\frac 12  + 2 \epsilon_2} ).
\ee
\end{itemize}

Combining (B1), (B2), we obtain
\be\label{E:large-41}
 R_{in,<}=\epsilon_\flat\theta_{\rho_1,0}(\lambda'_R ) \mathcal O_{\epsilon_1}(t^{-\frac 12+2\epsilon_2}) .
\ee

\noindent  {$\blacktriangleright$}  $\underline{\textbf{ Step 2. On the regime $\Sigma_{>}$ :}}$ From Lemma \ref{L:potential}, 
\be\label{E:large-2}
|\iint_{\Sigma_{>}} dx_1'dx_2'  \mathbf v(x'_1,x'_2;t_2 )|= |u_0|_{\mathfrak M^{4,6}}\mathcal O(t^{-1}) .
\ee   Along with \eqref{E:p}, yields to showing 
\be\label{E:large-42}
|\iint _{(\xi'_1,\xi'_2),\,\xi'_2\in \Omega^c_{\lambda'}}d{\xi'_1} d{\xi'_2} \frac{\frac{1}{1+\xi'^4_1}e^{ 2\pi i t\wp_2  }}{{\xi'_2} -z_0 }|<C.
\ee

 To this aim, let
\[ \mathcal A= \{\xi'_2: |\partial_{\xi'_2}\wp_2|\le 1\}.
\]Then  $|A|\le C|\xi'_1|$, and
\begin{align}
 &|\int d\xi'_1 \int  _{\xi'_2\in \Omega^c_{\lambda'}}   d{\xi'_2} \frac{\frac{1}{1+\xi'^4_1}e^{ 2\pi i t\wp_2  }}{{\xi'_2} -z_0 }|  \label{E:large-43}\\
\le &| \int d\xi'_1 \frac{1}{1+\xi'^4_1}\int_{\xi'_2\in \Omega^c_{\lambda'}\cap \mathcal A^c}  d{\xi'_2} \frac{ e^{ 2\pi i t\wp_2  }}{{\xi'_2} -z_0}|+| \int d\xi'_1\frac{1}{1+\xi'^4_1}\int_{\xi'_2\in \Omega^c_{\lambda'}\cap \mathcal A }  d{\xi'_2} \frac{ e^{ 2\pi i t\wp_2  }}{{\xi'_2} -z_0}|\nonumber\\  
\le &| \int d\xi'_1 \frac{1}{1+\xi'^4_1}\int_{\xi'_2\in \Omega^c_{\lambda'}\cap \mathcal A^c}  d{\xi'_2} \frac{ e^{ 2\pi i t\wp_2  }}{{\xi'_2} -z_0}|+C. \nonumber
\end{align}

Moreover, on the regime $\xi'_2\in \Omega^c_{\lambda'}\cap \mathcal A^c$, we can adapt the deformation method used in {\bf Step 1 (B1)} and obtain
\be\label{E:large-44} 
| \int_{\xi'_2\in \Omega^c_{\lambda'}\cap \mathcal A^c}  d{\xi'_2} \frac{\frac{1}{1+\xi'^4_1}e^{ 2\pi i t\wp_2  }}{{\xi'_2} -(2\pi i{\xi'_1}^2+2{\xi'_1}\overline\lambda') }|=o(t^{-1}).
\ee
\end{proof}

We will need the following lemma for estimates of $u_{2,0}$  on $\Xi^c$.
\begin{lemma}\label{L:u20-large}  
 \begin{align} 
& -\dfrac 1{\pi i} \iint d\overline\lambda'\wedge d\lambda' e^{2\pi i tS_0 } \widetilde {  s}_c(\lambda'  )      (\overline\lambda'-\lambda')\widetilde{(\mathcal CT)^n  } \label{E:u20-large} \\
\times &\iint dx_1'dx_2' \mathbf v ( x'_1,x_2';t_2 )\iint d{\xi'_{1,1}}d{\xi'_{2,1}}\frac{ \Xi^c(\xi_{1,1}')e^{ 2\pi i t \wp}}{(1+\xi'^4_{1,1})({\xi'_{2,1}} -z_{0,1}(\xi'_{1,1},\lambda'_1) ) }  \nonumber\\ 
= & -\dfrac 1{\pi i} \iint d\overline\lambda'\wedge d\lambda' e^{2\pi i tS_0 } \widetilde {  s}_c(\lambda'  )      (\overline\lambda'-\lambda')\widetilde{(\mathcal CT)^n  }  \nonumber\\
\circ & \left(e^{- i\pi tS_0(a;\lambda'_1) } \mathfrak C \mathfrak T_{(1)}   \Xi^c(\xi'_{1,1})\right)(t_1,t_2,t;\overline\lambda'_1)+\epsilon^{n+1 }_\flat\epsilon_\sharp \mathcal O_{\epsilon_1} (t^{-1 +3\epsilon_2}) .\nonumber 
\end{align}   Here we have identified $\lambda'_1=\lambda'_{R,1}+i\lambda'_{I,1}=\frac{\xi'_{2,2}}{2\xi'_{1,2}}-i\pi\xi'_{1,2}$ and $z_{0,1}(\xi'_{1,1},\lambda'_1)= 2\pi i\xi'^2_{1,1}+2\xi'_{1,1}\overline\lambda'_1 $. 
\end{lemma}

\begin{proof}

Applying  Lemma \ref{L:wickerhauser}, 
\be\label{E:lambda-L1}
 |(\overline\lambda'-\lambda')\widetilde {  s}_c(\lambda'  )|_{L^1(d\lambda'_R d\lambda'_I)}\le C, 
\ee and Lemma \ref{L:CIO-large} (in particular, {\bf (A2) and (B2) in Step 1}), it suffices to study
\be\label{E:u20-large-1-0}
  \iint d\overline\lambda'\wedge d\lambda' e^{2\pi i tS_0 } \widetilde {  s}_c(\lambda'  )      (\overline\lambda'-\lambda') \widetilde{(\mathcal CT)^{n}  }\circ R_{in,\gtrless,-}(t_1,t_2,t;\overline\lambda_1) .\ee 

\noindent  {$\blacktriangleright$}  $\underline{\textbf{ Step 1.}}$  
In this step, we will   prove
\begin{align}
 R_{out,\gtrless}\equiv&  \theta_{  {r_\ast}/2, 0}(\frac{\lambda'_I }\pi) R_{in,\gtrless,-}= \epsilon_\sharp(1+|\lambda'_I|)o(t^{-1}).  \label{E:uCIO-large-4}
\end{align}

 Notice that  
\be \label{E:wp-3-new}
(\xi'_1,\xi'_2)\in \Omega^c_{\lambda', {r_\ast}/2},\qquad | \partial_{\xi'_1}\wp_1| \ge\frac { 1}C ({\xi'_1}^2 +|a|).
\ee

Moreover,
\begin{align}
&|\int_{(\xi'_1,\xi'_2)\in\Omega^c_{\lambda',r_\ast/2 }}  d{\xi'_2} \frac{e^{ 2\pi i t\wp_2  }}{{\xi'_2} -(2\pi i{\xi'_1}^2+2{\xi'_1}\lambda') }|\label{E:uCIO-large-5}\\
\le &|\int_{|   \frac 32 \sgn(\xi_1'){\eta}+x_2'\sqrt{\frac {|\xi'_1|}t} |<1} d{\eta}\frac{[1-\theta_{ \frac{r_\ast}2 \sqrt{\frac{t}{|\xi'_1|}},\sqrt{\frac t{|\xi'_1|} }z_0}(\eta)]  e^{ -2\pi i [   \frac 34 \sgn(\xi_1'){\eta}^2 +x_2'\sqrt{\frac{|\xi'_1|}{t}}{\eta}]  }}{{\eta} -\sqrt{\frac{t}{|\xi'_1|}}z_0 }|\nonumber\\
+&|\int_{|    \frac 32 \sgn(\xi_1'){\eta}  +x_2'\sqrt{\frac{|\xi'_1|}{t}} |>1}  d{\eta} \frac{[1-\theta_{ \frac{r_\ast}2 \sqrt{\frac{t}{|\xi'_1|}},\sqrt{\frac t{|\xi'_1|} }z_0}(\eta)] e^{ -2\pi i [\frac 34 \sgn(\xi_1'){\eta}^2+x_2'\sqrt{\frac  {|\xi'_1|} t}{\eta}]  }}{{\eta} -\sqrt{\frac t{|\xi'_1|} }z_0 }|\nonumber\\
<&C(|\frac{|\xi_1'|}t|^{\frac 12}+|\frac{\xi_1'}t|),\nonumber
\end{align} via decomposing the domain according to the stationary point of $\wp_2$, the change of variables $ \xi_2' \mapsto \eta=\sqrt{\frac t{|\xi_1'|}}\xi'_2$, and using   integration by parts w.r.t. $\eta$ away from stationary points.

Consequently, integration by parts with respect to $\xi'_1$, using    \eqref{E:wp-3-new},   \eqref{E:uCIO-large-5}, \eqref{E:wp-3},  
\begin{align}
\bullet\quad& \partial_{\xi'_1}\frac 1{{\xi'_2} -(2\pi i{\xi'_1}^2+2{\xi'_1}\overline\lambda') }
\label{E:uCIO-large-10}\\
 =&  \frac { i(2\pi  {\xi'_1}^2-2\xi_1'\lambda'_I)}{\xi'_1({\xi'_2} -(2\pi i{\xi'_1}^2+2{\xi'_1}\overline\lambda'))^2 }+\frac {2  \pi i\xi'_1 }{ ({\xi'_2} -(2\pi i{\xi'_1}^2+2{\xi'_1}\overline\lambda'))^2 }  \nonumber\\
 +&\frac {  - \xi'_2+2\xi_1'\lambda'_R }{\xi'_1({\xi'_2} -(2\pi i{\xi'_1}^2+2{\xi'_1}\overline\lambda'))^2 }  \nonumber\\
 +&\frac { -t\frac 32\frac{\xi_2'}{\xi'_1}-x'_2}{(-t\frac 32)({\xi'_2} -(2\pi i{\xi'_1}^2+2{\xi'_1}\overline\lambda'))^2 }+\frac { x'_2}{(-t\frac 32)({\xi'_2} -(2\pi i{\xi'_1}^2+2{\xi'_1}\overline\lambda'))^2 },\nonumber\\
\bullet\quad &\partial_{\xi'_1}e^{2\pi i [-t  \frac 34\frac{{\xi'_2}^2}{\xi'_1} -x_2'{\xi'_2}  ]}  \label{E:uCIO-large-11}\\
 =&   2\pi i (t\frac 34\frac{{\xi_2'}^2}{{\xi'_1}^2} )e^{2\pi i [-t  \frac 34\frac{{\xi'_2}^2}{\xi'_1} -x_2'{\xi'_2}  ]} \nonumber\\
=&  \frac 1{6\pi i  t}   \partial_{\xi_2'}^2e^{2\pi i [-t  \frac 34\frac{{\xi'_2}^2}{\xi'_1} -x_2'{\xi'_2}  ]}+\frac  {2  x_2'}  {3t} \partial_{\xi_2'} e^{2\pi i [-t  \frac 34\frac{{\xi'_2}^2}{\xi'_1} -x_2'{\xi'_2}]  }+\frac  {2\pi i { x_2'}^2}  {3t}   e^{2\pi i [-t  \frac 34\frac{{\xi'_2}^2}{\xi'_1} -x_2'{\xi'_2}  ]},\nonumber
\end{align}  
and integration by parts  with respect to $\xi'_2$, we obtain 
\be\label{E:uCIO-large-6}
R_{out,\gtrless,+} = \epsilon_\sharp o(t^{-1}).
\ee

 Therefore, the study of \eqref{E:u20-small-2-0} reduces to investigating 
\begin{align}
&  \iint d\overline\lambda'\wedge d\lambda' e^{2\pi i tS_0 } \widetilde {  s}_c(\lambda'  )      (\overline\lambda'-\lambda')\widetilde{(\mathcal CT)^{n}  }[1-\theta_{ \frac{r_\ast}2, 0}(\xi'_{1,2}  )]R_{in,\gtrless,-}.  \label{E:u20-large-1}
\end{align}

\noindent  {$\blacktriangleright$}  $\underline{\textbf{ Step 2: Proof of  \eqref{E:u20-large-1} about $R_{in,<,-} $ :}}$  When $n=0$, applying \eqref{E:large-40},  Lemma \ref{L:wickerhauser}, and \eqref{E:lambda-L1}, 
\begin{align}
& | \iint d\overline\lambda'\wedge d\lambda' e^{2\pi i tS_0 } \widetilde {  s}_c(\lambda'  )      (\overline\lambda'-\lambda')[1-\theta_{ r_\ast/2, 0}(\frac{\lambda'_I}{2\pi}  )] R_{in,<,-} | \label{E:u20-large-2} \\
\le &  \epsilon_\flat \mathcal O(t^{-\frac 12 +2 \epsilon_2} )|\iint d\overline\lambda'\wedge d\lambda' e^{2\pi i tS_0 } \widetilde {  s}_c(\lambda'  )      (\overline\lambda'-\lambda')[1-\theta_{ r_\ast/2, 0}(\frac{\lambda'_I}{2\pi}  )] \theta_{\rho_1,0}(\lambda'_R) |\nonumber\\
=& \epsilon_\flat^2\mathcal O(t^{-1 +3\epsilon_2}).\nonumber
\end{align}

When $n>0$,  applying  \eqref{E:large-40},  
\begin{multline}\label{E:u20-large-3}
  |\widetilde{(\mathcal CT)   }\circ [1-\theta_{ r_\ast/2, 0}(\frac{\lambda'_{I,1}}{2\pi}  )] R_{in,<,-}(t_1,t_2,t;\overline\lambda'_1)|
 \le \epsilon_\flat\mathcal O(t^{-\frac 12+2\epsilon_2}) |\iint dx'_{1,2}dx'_{2,2} \\
\times \mathbf v (x'_{1,2},x'_{2,2}) \iint d\xi'_{1,2}d\xi'_{2,2}\frac{[1-\theta_{ r_\ast/2, 0}(\xi'_{1,2}  )  \frac{1}{1+\xi'^4_1} \theta_{\rho_1,0}( \frac{\xi'_{2,1}}{2\xi'_{1,2}}  )  e^{ 2\pi i t \wp}}{{\xi'_2} -z_0 }|.   
\end{multline} 

Therefore, applying the change of variables \cite[Lemma 2.II]{W85, W87} 
\be \label{E:large-39-new}
  |\eta_{2,2}+4\pi\xi'_{1,2}\lambda'_R|=2\pi|  \xi'_{2,2} | \le C\rho_1 |\xi'_{1,2}| ,  
\ee  and adapting argument as used for $R_{in,\gtrless,-}$ (see \eqref{E:large-29-(0)}, \eqref{E:large-29-(2)}, \eqref{E:large-29-(2)-(add)}, \eqref{E:large-40-(add)}), we obtain 
\be\label{E:u20-large-4}
 \iint d\xi'_{1,2}d\xi'_{2,2}\frac{[1-\theta_{ r_\ast/2, 0}(\xi'_{1,2}  )  \frac{1}{1+\xi'^4_1} \theta_{\rho_1,0}( \frac{\xi'_{2,1}}{2\xi'_{1,2}}  )  e^{ 2\pi i t \wp}}{{\xi'_2} -z_0 } 
= \mathcal O(t^{-\frac 12+ 2 \epsilon_1}) .
\ee

Along with Lemma \ref{L:wickerhauser} and \eqref{E:lambda-L1}, yields, for  $n>0$,
\begin{multline}\label{E:u20-large-5}
  \iint d\overline\lambda'\wedge d\lambda' e^{2\pi i tS_0 } \widetilde {  s}_c(\lambda'  )      (\overline\lambda'-\lambda')\widetilde{(\mathcal CT)^{n}  }\circ [1-\theta_{ r_\ast/2, 0}(\frac{\lambda'_{I,1}}{2\pi}  )] R_{in,<,-}(t_1,t_2,t;\overline\lambda_1)\\= \epsilon_\flat^{n+2}\mathcal O(t^{-1 +3\epsilon_2}).\end{multline}

Combining \eqref{E:u20-large-2} and \eqref{E:u20-large-5}, we prove \eqref{E:u20-large-1} about $R_{in,<,-} $.

\noindent  {$\blacktriangleright$}  $\underline{\textbf{ Step 3: Proof of  \eqref{E:u20-large-1} about $R_{in,>,-} $ :}}$  We will sharpen the estimates \eqref{E:large-29-(0)} and  \eqref{E:large-29-(2)} for $R_{in,>,-} $ to prove the statement. Precisely, thanks to 
\be\label{E:u20-large-6}
\begin{gathered}
|\omega|=|\frac{\xi'_2}{\xi'_1\rho_1}|\le \frac{1}{4}, \quad [1-\psi_{\rho_1,0}(\lambda'_R)]|\frac{z_0}{\xi'_1\rho_1}|\ge 1,\\
{\color{black}|\partial_{\xi'_1}^j\frac{1}{2\pi it\wp_1}|\le\frac{C}{t(\xi'^2_1+|a|^2)^{j+1}}\le Ct^{-1},}
\end{gathered}
\ee we obtain
\begin{align}
\bullet\quad &|\partial_{\xi'_1}\frac{1}{\omega-\frac{z_0}{|\xi'_1|\rho_1}}|\le C|\frac{\partial_{\xi'_1}\frac{z_0}{|\xi'_1|\rho_1}}{(\omega-\frac{z_0}{|\xi'_1|\rho_1})^2}\le C \rho_1^{-1},\label{E:u20-large-7}\\
\bullet\quad &|\partial_{\xi'_1}e^{2\pi i[-\frac 34t\xi'_1\rho_1^2\omega^2-x'_2|\xi'_1|\rho_1\omega]}| \le Ct^{+2\epsilon_1}.\label{E:u20-large-8}
\end{align}
Together with integration by parts three times with respect to $\xi'_1$, yield
\begin{align}
&|R_{in,>,-}|
\label{E:u20-large-9}\\
\le& C| \iint_{\Sigma_{<}}  dx_1'dx_2' \mathbf v(x'_1,x'_2;t_2 ) \int  d{\xi'_1}\frac{1}{1+\xi'^4_1}   \Xi^ce^{ 2\pi i t\wp_1 } \left[\partial_{\xi'_1}\left( \frac{1}{3t\pi^2\xi'^2_1+(ta-x'_1)}\right.\right]^3\nonumber\\
\times&\left.\left.\left.\int   d{\omega} [1-\theta_{\rho_1,0}( \lambda'_R )]\theta(\frac14 -|\omega|) \frac{ e^{ 2\pi i [-\frac 34  {t {\xi'_1}\rho_1 ^2}     \omega ^2  -x_2'|\xi'_1|\rho_1\omega ]  }}{\omega - \frac{z_0}{|\xi'_1\rho_1|  }}\right)\right)\right)+\epsilon_\flat \mathcal O(t^{-1})\nonumber\\ 
\le &\epsilon_\flat \mathcal O(t^{{\color{black}3(-1)}-3(-\frac 12+\epsilon_1)}) +\epsilon_\flat \mathcal O(t^{-1}).\nonumber
\end{align}
Along with Lemma \ref{L:wickerhauser} and \eqref{E:lambda-L1}, we justify \eqref{E:u20-large-1} about $R_{in,>,-} $.

\end{proof}

The estimates for the Cauchy integrals near $\xi'_1=\pm r/\pi$ are similar to those as $|\xi'_1|\in \Xi^c$.   
\begin{lemma}\label{L:small} [$\textrm{Estimates of $\widetilde{\mathcal CT1}$  near $\xi'_{1,1}=\pm\frac {r}{ \pi}$}$]    
\begin{align}
&-\dfrac 1{2\pi i}   \iint dx_1'dx_2' \mathbf v( x'_1,x_2';t_2 ) \iint d{\xi'_1}d{\xi'_2} 
 \frac{\theta_{r_\ast,\frac {r}{ \pi}}(\xi_1')   e^{ 2\pi i t\wp}}{(1+\xi'^4_1)({\xi'_2} -z_0) }  \label{E:small-1}\\
&\hskip.8in=  
e^{- i\pi tS_0(a;\lambda') }  \mathfrak C \mathfrak T_{(1)}    \theta_{r_\ast,\frac {r}{ \pi}}(\xi_1')  \nonumber
+ \epsilon_\flat   \mathcal O_{\epsilon_1}  ( t^{ -\frac 12+2\epsilon_2 })+\epsilon_\sharp  \mathcal O(t^{-1}).\nonumber
\end{align}

\end{lemma} 
\begin{proof}  Following the same deformation approach used in {\bf Step 2 of Lemma \ref{L:CIO-large}}, we derive 
\be \label{E:small-2}
   \iint_{\Sigma_{>} }dx_1'dx_2' \mathbf v( x'_1,x_2';t_2 ) \iint d{\xi'_1}d{\xi'_2} 
 \frac{\theta_{r_\ast,\frac {r}{ \pi}}(\xi_1')   e^{ 2\pi i t\wp}}{(1+\xi'^4_1)({\xi'_2} -z_0) }  =
\epsilon_\sharp \mathcal O(t^{-1}).
\ee

Hence it reduces to studying estimates on the regime $\Sigma_{<}$. Write 
\begin{align}
&    \iint_{\Sigma_{<} } dx_1'dx_2' \mathbf v( x'_1,x_2';t_2 ) \iint d{\xi'_1}d{\xi'_2} 
 \frac{\theta_{r_\ast,\frac {r}{ \pi}}(\xi_1')   e^{ 2\pi i t\wp}}{(1+\xi'^4_1)({\xi'_2} -z_0) }\label{E:small-3}\\
=&      \iint_{\Sigma_{<} } dx_1'dx_2'\mathbf v( x'_1,x_2';t_2 )  \iint d{\xi'_1}d{\xi'_2}\frac{\theta_{r_\ast,\frac {r}{ \pi}}(\xi_1') [1-\theta_{\rho_1,0}(\lambda'_R) ]  e^{ 2\pi i t\wp}}{(1+\xi'^4_1)({\xi'_2} -z_0) } \nonumber\\
+ &    \iint_{\Sigma_{<} } dx_1'dx_2' \mathbf v( x'_1,x_2';t_2 )   \iint  d{\xi'_1}d{\xi'_2}\frac{\theta_{r_\ast,\frac {r}{ \pi}}(\xi_1')\theta_{\rho_1,0}(\lambda'_R) e^{ 2\pi i t\wp}}{(1+\xi'^4_1)({\xi'_2} -z_0) } \nonumber\\
 \equiv&   P_{ in,>}+ P_{ in ,< } .\nonumber
\end{align}

\noindent  {$\blacktriangleright$}  $\underline{\textbf{ Step 1. Estimates for $P_{ in,>}$:}}$   
Introduce the scaling
\be\label{E:small-7}
  \omega=\frac{\xi'_2}{|\xi'_1|\rho_1}  , 
\ee  and decompose
\begin{align}
  P_{in,>} \equiv &P_{in,>,+}+P_{in,>,-} , \label{E:small-8}\\
P_{in,>,\pm}=&   \iint_{\Sigma_{<}}  dx_1'dx_2' \mathbf v(x'_1,x'_2;t_2 ) \int  d{\xi'_1}  e^{ 2\pi i t\wp_1 }\frac{\theta_{r_\ast,\frac {r}{ \pi}}(\xi_1')}{1+\xi'^4_1}\nonumber \\
\times&\int   d{\omega} \left[1-\theta_{\rho_1,0}(\lambda'_R)\right]  \theta(\pm|\omega|\mp \frac 14) \frac{ e^{ 2\pi i [-\frac 34  {t{\xi'_1}\rho_1^2     }\omega ^2  -x_2'|{\xi'_1}|\rho_1 \omega ]  }}{ \omega - \frac{z_0}{|\xi'_1|\rho_1 }}.\nonumber
\end{align} 

For $P_{in,>,+}$, using 
\be
1\le |  \frac {z_0}{\xi'_1\rho_1}  |,\quad |   {t{\xi'_1}\rho^2 _0 }\omega^2       |\ge \frac 1Ct^{+2\epsilon_1}, \quad |t\rho_1\omega|\ge \frac{1}{C}  t^{\frac 12+\epsilon_1 }\gg t^{\frac12 }\ge|   x_2'      |  ,\label{E:small-9}
\ee     and the deformation argument of $R_{in,>,+}$ in {\bf Step 1.(A1), Lemma \ref{L:CIO-large}}, we obtain
\be \label{E:small-10}
  P_{in,>,+}
=  
e^{ -i\pi tS_0(a;\lambda') }  \mathfrak C \mathfrak T_{(1)}\theta_{r_\ast,\frac {r}{ \pi}}(\xi_1') + \epsilon_\flat o_{\epsilon_1}( t^{-1}) .  
\ee

For $P_{in,>,-}$,    applying the change of variables \eqref{E:large-26} and   \eqref{E:small-7},  
\be \label{E:small-11}
     |\eta_2+4\pi\xi'_1\lambda'_R|\le 2\pi|\xi'_2| \le \frac{\pi}{2} |\xi'_1|\rho_1  ,
\ee 
adapting the argument used for $R_{in,>,-}$ in \eqref{E:large-29-(0)} and \eqref{E:large-29-(2)}, we obtain
\be\label{E:small-12}
  P_{in,>,-}  
= \epsilon_\flat [1- \theta_{\rho_1,0}(\lambda'_R)] \mathcal O(t^{-\frac 12 + \epsilon_1} ).
\ee

Combining \eqref{E:small-10} and \eqref{E:small-12}, we obtain 
\be\label{E:small-13}
 P_{in,> }
=   
  e^{ -i\pi tS_0(a;\lambda') }  \mathfrak C \mathfrak T_{(1)}\theta_{r_\ast,\frac {r}{ \pi}}(\xi_1')  
+\epsilon_\flat [1-\theta_{t^{ -\frac 12+\epsilon_1},0}(\lambda'_R)]  \mathcal O _{\epsilon_1} ( t^{-\frac 12 +  \epsilon_1}) .  
\ee

\noindent  {$\blacktriangleright$}  $\underline{\textbf{ Step 2. Estimates for $P_{in, <}$:}}$  Via the scaling
\be\label{E:small-14}
  \omega=\frac{\xi'_2}{|\xi'_1|\rho_2} ,
\ee we decompose
\begin{align}
 P_{in,<}\equiv& P_{in,<,+}+P_{in,<,- }, \label{E:small-15}\\
P_{in,<,\pm}=&  \iint_{\Sigma_{<}}  dx_1'dx_2' \mathbf v(x'_1,x'_2;t_2 ) \int  d{\xi'_1}  e^{ 2\pi i t\wp_1 }\frac{\theta_{r_\ast,\frac {r}{ \pi}}(\xi_1')}{1+\xi'^4_1}\nonumber \\
\times&\int   d{\omega}  \theta_{\rho_1,0}(\lambda'_R)  \theta(\pm|\omega|\mp \frac 14) \frac{ e^{ 2\pi i [-\frac 34  {t{\xi'_1}\rho_2^2     }\omega ^2  -x_2'|{\xi'_1}|\rho_2 \omega ]  }}{ \omega - \frac{z_0}{|\xi'_1|\rho_2 }}.\nonumber
\end{align}

For $P_{in,<,+}$,  using
\be\label{E:small-16}
  |    \frac {\textit{Re}(z_0)}{ \xi'_1 \rho_2 }  |\ll 1 ,\quad |  t  \xi'_1\rho^2 _2  \omega^2     |\ge \frac 1Ct^{+2\epsilon_2 }, \quad | t  \rho _2  \omega|\ge   \frac 1Ct^{\frac 12+ \epsilon_2 }\gg t^{\frac12 }\ge|   x_2'     | ,
\ee and the deformation argument of $R_{in,><,+}$ in {\bf Step 1.(B1), Lemma \ref{L:CIO-large}}, we obtain
\be \label{E:small-17}
P_{in,<,+}= \epsilon_\flat o_{\epsilon_1}(t^{-1}).
\ee

Besides,  for $P_{in,<, -}$,  we will apply the change of variables \eqref{E:large-26} (\cite[Lemma 2.II]{W85, W87}),  
\be \label{E:small-18}
 |\eta_2+4\pi\xi'_1\lambda'_R|\le  2\pi|\xi'_2| \le \frac{\pi}{2} |\xi'_1| t^{-\frac12+ \epsilon_2}.
\ee adapting the same argument as used for \eqref{E:large-40},
\be\label{E:small-19}
 P_{in,<, -}=  \epsilon_\flat \theta_{t^{ -\frac 12 +\epsilon_1},0}(\lambda'_R)\mathcal O(t^{-\frac 12+2\epsilon_2 }  ). 
\ee

From \eqref{E:small-17} and  \eqref{E:small-19},  
\be\label{E:small-20}
P_{in,<}= \epsilon_\flat\theta_{t^{ -\frac 12+\epsilon_1 },0}(\lambda'_R) \mathcal O_{\epsilon_1}(t^{-\frac 12+ 2\epsilon_2  } ).
\ee

Combining \eqref{E:small-3} with \eqref{E:small-13} and \eqref{E:small-20},  we prove the lemma.

\end{proof}

Application of Lemma \ref{L:small} to $u_{2,0}$ near $\xi'_{1,1}=\pm\frac {r}{ \pi}$ requires a different treatment of $P_{in,>,-}$ because $\pm r/\pi$ are stationary points. For $n=0$, integration by parts can be applied after excluding regions whose contributions are controlled by measure-size estimates. For $n>0$, we use the deformation method and estimates for the ABF Cauchy integrals.  
\begin{lemma}\label{L:u20-small}   Suppose   $a<-\delta<0  $. 
\begin{align}
&-\dfrac 1{\pi i} \iint d\overline\lambda'\wedge d\lambda' e^{2\pi i tS_0 } \widetilde {  s}_c(\lambda'  )      (\overline\lambda'-\lambda')\widetilde{(\mathcal CT)^n  }  \label{E:u20-small-1}\\
\times& \iint dx_{1,1}'dx_{2,1}' \mathbf v ( x'_{1,1},x_{2,1}';t_2 ) \iint d{\xi'_{1,1}}d{\xi'_{2,1}}\frac{ \theta_{r_\ast,\frac {r}{ \pi}}(\xi_{1,1}')e^{ 2\pi i t \wp}}{(1+\xi'^4_{1,1})({\xi'_{2,1}} -z_{0,1}(\xi'_{1,1},\lambda'_1)) }  \nonumber\\
= &-\dfrac 1{\pi i} \iint d\overline\lambda'\wedge d\lambda' e^{ 2\pi i tS_0 } \widetilde {  s}_c(\lambda'  )      (\overline\lambda'-\lambda')
\widetilde{(\mathcal CT)^n  } 
\nonumber\\
\circ &\left(e^{- i\pi tS_0(a;\lambda'_1) }\mathfrak C \mathfrak T_{(1)}\theta_{ r_\ast, \frac {r}{ \pi}}(\xi_{1,1}' )\right)(t_1,t_2,t;\overline\lambda'_1) +\epsilon^{n+1  }_\flat\epsilon_\sharp\mathcal O_{\epsilon_1}(t^{-1+3\epsilon_2}). \nonumber
\end{align}  Here  $\lambda'_1=\lambda'_{R,1}+i\lambda'_{I,1}=\frac{\xi'_{2,2}}{2\xi'_{1,2}}-i\pi\xi'_{1,2}$ and $z_{0,1}(\xi'_{1,1},\lambda'_1)= 2\pi i\xi'^2_{1,1}+2\xi'_{1,1}\overline\lambda'_1 $.

\end{lemma}
\begin{proof}To prove the lemma, from  Lemma \ref{L:wickerhauser}, \eqref{E:lambda-L1}, and  Lemma \ref{L:small}, it suffices to study
\begin{align}
&  \iint d\overline\lambda'\wedge d\lambda' e^{2\pi i tS_0 } \widetilde {  s}_c(\lambda'  )      (\overline\lambda'-\lambda')\widetilde{(\mathcal CT)^{n}  }\circ P_{in,\gtrless,-}(t_1,t_2,t;\overline\lambda'_1).  \label{E:u20-small-2-0}
\end{align}

\noindent  {$\blacktriangleright$}  $\underline{\textbf{ Step 1.}}$  
In this step, we will   prove
\begin{align}
 P_{out,\gtrless}\equiv& [1-\theta_{2r_\ast, \frac r\pi}(\frac{\lambda'_I }\pi)] P_{in,\gtrless,-}=\epsilon_\sharp\mathcal O(t^{-1}).  \label{E:small-4}
\end{align}

Notice that $(\xi'_1,\xi'_2)\in \Omega^c_{\lambda',r_\ast/C}$. Decompose
\begin{align}
 P_{out,>}
=&    \iint_{\Sigma_{<} } dx_1'dx_2'\mathbf v( x'_1,x_2';t_2 )  \int d{\xi'_1} e^{ 2\pi i t\wp_1}\theta_{t^{-\frac 12},\frac {r}{ \pi}}(\xi_1')\label{E:small-4-0}\\
\times&[1-\theta_{\rho_1,0}(\lambda'_R)] [1-\theta_{2r_\ast, \frac r\pi}(\frac{\lambda'_I }\pi)]\int_{\xi'_2\in \Omega^c_{\lambda',r_\ast/C}} d{\xi'_2}\frac{ e^{ 2\pi i t\wp_2}}{(1+\xi'^4_1)({\xi'_2} -z_0) } \nonumber\\
+&   \iint_{\Sigma_{<} } dx_1'dx_2'\mathbf v( x'_1,x_2';t_2 )  \int d{\xi'_1} e^{ 2\pi i t\wp_1}\theta_{r_\ast,\frac {r}{ \pi}}(\xi_1')[1-\theta_{t^{-\frac 12},\frac {r}{ \pi}}(\xi_1')]\nonumber\\
\times&[1-\theta_{\rho_1,0}(\lambda'_R)][1-\theta_{2r_\ast, \frac r\pi}(\frac{\lambda'_I }\pi)]\int_{\xi'_2\in \Omega^c_{\lambda',r_\ast/C}} d{\xi'_2}\frac{ e^{ 2\pi i t\wp_2}}{(1+\xi'^4_1)({\xi'_2} -z_0) } \nonumber\\
\equiv&P_{out,>,-}+P_{out,>,+}.\nonumber  
\end{align}  

From    $|\theta_{t^{-\frac 12},\frac {r}{ \pi}}(\xi_1')|_{L^1(d\xi'_1)}\le Ct^{-\frac 12}$  and  \eqref{E:uCIO-large-5}, we obtain
\be \label{E:small-5}
P_{out,>,-}= \epsilon_\flat \mathcal O(t^{-1}).
\ee

Integration by parts with respect to $\xi'_1$, using  \eqref{E:wp-3-new}-\eqref{E:uCIO-large-11}, we derive
\be\label{E:small-6}
 P_{out,>,+} = \epsilon_\sharp \mathcal O(t^{-1}).
\ee

Proof for $P_{out,<}$ is similar. Therefore, combining with Lemma \ref{L:wickerhauser}, \eqref{E:lambda-L1}, and \eqref{E:u20-small-2-0}, it  reduces to investigating 
\begin{align}
&  \iint d\overline\lambda'\wedge d\lambda' e^{2\pi i tS_0 } \widetilde {  s}_c(\lambda'  )      (\overline\lambda'-\lambda')\widetilde{(\mathcal CT)^{n}  }\theta_{2r_\ast, \frac r\pi}(\xi'_{1,2}  )P_{in,\gtrless,-}.  \label{E:u20-small-2}
\end{align}

\noindent  {$\blacktriangleright$}  $\underline{\textbf{ Step 2: Proof of  \eqref{E:u20-small-2} when $n=0$ :}}$ Applying \eqref{E:small-19},
\be \label{E:u20-small-3}   \iint d\overline\lambda'\wedge d\lambda' e^{2\pi i tS_0 } \widetilde {  s}_c(\lambda'  )      (\overline\lambda'-\lambda')\theta_{2r_\ast, \frac r\pi}(\frac{\lambda'_I}{\pi}  ) P_{in,<,-}  =\epsilon_\flat^2\mathcal O(t^{-1+ 3\epsilon_2}). 
\ee

Similarly,  from \eqref{E:small-12} and \eqref{E:lambda-L1},
\begin{align}
& | \iint d\overline\lambda'\wedge d\lambda' e^{2\pi i tS_0 } \widetilde {  s}_c(\lambda'  )      (\overline\lambda'-\lambda')\theta_{2r_\ast, \frac r\pi}(\frac{\lambda'_I}{\pi}  ) P_{in,>,-} | \label{E:u20-small-4} \\
\le &  \epsilon_\flat^2 \mathcal O(t^{-1+ 3\epsilon_1})\nonumber\\
+& |\iint d\overline\lambda'\wedge d\lambda' e^{2\pi i tS_0 } \widetilde {  s}_c(\lambda'  )      (\overline\lambda'-\lambda')\theta_{2r_\ast,\frac r\pi}(\frac{\lambda'_I}{\pi}) [1-\theta_{\rho_1,0}(\lambda'_R) ]  [1- \theta_{\sqrt{\rho_1},\frac r\pi} (\frac{\lambda'_I }\pi) ]  \nonumber\\
 \times&\iint dx_{1 }'dx_{2 }' \mathbf v ( x'_{1 },x_{2 }';t_2 )       \iint d{\xi'_{1 }}d{\xi'_{2 }}\frac{\theta_{r_\ast,\frac r\pi}(\xi'_{1 })\theta( \frac 14 -|\frac{\xi'_{2 }}{ \xi'_{1 }\rho_1 }|)e^{ 2\pi i t \wp}}{(1+\xi'^4_1)({\xi'_{2 }} -z_0 ) }|  \nonumber\\
 +& |\iint d\overline\lambda'\wedge d\lambda' e^{2\pi i tS_0 } \widetilde {  s}_c(\lambda'  )      (\overline\lambda'-\lambda')\theta_{2r_\ast,\frac r\pi}(\frac{\lambda'_I}{\pi}) [1-\theta_{\sqrt{\rho_1},0}(\lambda'_R) ]   \theta_{\sqrt{\rho_1},\frac r\pi} (\frac{\lambda'_I }\pi)   \nonumber\\
 \times&\iint dx_{1 }'dx_{2 }' \mathbf v ( x'_{1 },x_{2 }';t_2 )       \iint d{\xi'_{1 }}d{\xi'_{2 }}\frac{\theta_{r_\ast,\frac r\pi}(\xi'_{1 })\theta( \frac 14 -|\frac{\xi'_{2 }}{ \xi'_{1 }\rho_1 }|)e^{ 2\pi i t \wp}}{(1+\xi'^4_1)({\xi'_{2 }} -z_0 ) }|  \nonumber\\
\equiv & \epsilon_\flat^2 \mathcal O(t^{-1+ 3\epsilon_1})+I_{P_{in,>,-}}+J_{P_{in,>,-}} .\nonumber 
\end{align}

 To study $I_{P_{in,>,-}}$, we consider: 
\begin{itemize}
\item  [\bf(A)]     If $|\partial_{\lambda'_I}S_0 |=|+\frac 1\pi(-a+3(\lambda'^2_R-\lambda'^2_I))|\le  t^{ -\frac 14+\frac{\epsilon_1}4 }$:  Hence 
\be\label{E:u20-small-1(add)}
\lambda'^2_R\sim t^{ -\frac 14+\frac{\epsilon_1}2 },\qquad|\partial_{\lambda'_R}S_0 |=|+\frac6\pi\lambda'_R\lambda'_I|\sim t^{ -\frac 18+\frac{\epsilon_1}4 }.
\ee Along with the change of variables \eqref{E:large-26}, 
\begin{gather}
  |\partial_{\lambda'_{ R}}\frac 1{{\xi'_{2 }} -(2\pi i{\xi'_{1 }}^2+2{\xi'_{1 }}\overline\lambda' ) }| 
\le    C| \frac {   \xi' _1  }{({\xi'_{2 }} -(2\pi i{\xi'_{1 }}^2+2{\xi'_{1 }}\overline\lambda' ))^{2} }|, \label{E:u20-small-2(add)}\\
|{{\xi'_{2 }} -(2\pi i{\xi'_{1 }}^2+2{\xi'_{1 }}\overline\lambda' ) }|
\ge \frac 1C  |\xi'_1|t^{ -\frac 18+\frac{\epsilon_1}4 } , \quad
|\eta_2+4\pi\xi'_{1 }\lambda'_R|=2\pi|\xi'_{2 }| \le \frac{\pi}{2}|\xi'_1|\rho_1  , \label{E:u20-small-3(add)}
\end{gather}integration by parts with respect to $\lambda'_R$, applying Lemma  \ref{L:sd-der},   \eqref{E:lambda-L1}, and  argument used for \eqref{E:large-29-(0)} and \eqref{E:large-29-(2)}, yields
\be\label{E:u20-small-4(add)}
I_{P_{in,>,-}}\sim   \epsilon_\flat^2 t^{{\color{black}(-1+\frac 18-\frac{\epsilon_1}{4})}+( -\frac 12+ \epsilon_1 )+(\frac18-\frac{\epsilon_1}4)} ) =   \epsilon_\flat^2 \mathcal O(t^{-\frac54 + \epsilon_1}).  
\ee 

\item [\bf(B)]   If $|\partial_{\lambda'_I}S_0 |=|+\frac 1\pi(-a+3(\lambda'^2_R-\lambda'^2_I))|\ge  t^{ -\frac 14+\frac{\epsilon_1}4 }$: In view of \eqref{E:u20-small-4}, we decompose
\begin{align}
  I^\pm_{P_{in,>,-}} \equiv & I^+_{P _{in,>,-}}  +I^-_{P _{in,>,-}},  \label{E:p-in-<-dec}\\
 I^\pm _{P _{in,>,-}}=  &  |\iint d\overline\lambda'\wedge d\lambda' e^{2\pi i tS_0 } \widetilde {  s}_c(\lambda'  )      (\overline\lambda'-\lambda')\theta_{2r_\ast,\frac r\pi}(\frac{\lambda'_I}{\pi})   \nonumber\\
 \times&[1-\theta_{\rho_1,0}(\lambda'_R) ]  [1- \theta_{\sqrt{\rho_1},\frac r\pi} (\frac{\lambda'_I }\pi) ]\theta(\pm|\pi\xi'_1-\lambda'_I |\mp\frac {\sqrt{\rho_1}}2)\nonumber\\
 \times& \iint dx_{1 }'dx_{2 }' \mathbf v ( x'_{1 },x_{2 }';t_2 )       \iint d{\xi'_{1 }}d{\xi'_{2 }}\frac{\theta_{r_\ast,\frac r\pi}(\xi'_{1 })\theta( \frac 14-|\frac{\xi'_{2 }}{ \xi'_{1 } \rho_1 }|)e^{ 2\pi i t \wp}}{(1+\xi'^4_1)({\xi'_{2 }} -z_0 ) }|.\nonumber
\end{align}

\begin{itemize}
\item [\bf(B1)] {\bf Proof for $ I_{P^+_{in,>,-}}$:} Using the change of variables \eqref{E:large-26},
\begin{gather}
  |\partial^j_{\lambda'_{ I}}\frac 1{{\xi'_{2 }} -(2\pi i{\xi'_{1 }}^2+2{\xi'_{1 }}\overline\lambda' ) }| 
\le    C| \frac {   \xi'^j_1  }{({\xi'_{2 }} -(2\pi i{\xi'_{1 }}^2+2{\xi'_{1 }}\overline\lambda' ))^{1+j} }|, \label{E:u20-small-6}\\
|{{\xi'_{2 }} -(2\pi i{\xi'_{1 }}^2+2{\xi'_{1 }}\overline\lambda' ) }|
\ge \frac \pi 2 |\xi'_1| t^{ -\frac 14+\frac{\epsilon_1}2} ,\label{E:u20-small-5(2)}\\  
|\eta_2+4\pi\xi'_{1 }\lambda'_R|=2\pi|\xi'_{2 }| \le C|\xi'_1|\rho_1  , \label{E:u20-small-5(3)}
\end{gather}   
 integration by parts with respect to $\lambda'_{I}$,  applying Lemma  \ref{L:sd-der},  \eqref{E:lambda-L1}, and adapting argument used for \eqref{E:large-29-(0)} and \eqref{E:large-29-(2)}, 
we obtain
\be\label{E:u20-small-5}
I_{P^+_{in,>,-}}\sim   \epsilon_\flat^2 t^{{\color{black}(-1+\frac 14-\frac{\epsilon_1}{4})}+( -\frac 12+ \epsilon_1 )+(\frac14-\frac{\epsilon_1}2)} ) =   \epsilon_\flat^2 \mathcal O(t^{-1 + \epsilon_1}).  
\ee

\item [\bf(B2)]  {\bf Proof for $ I_{P^-_{in,>,-}}$:} Using  the change of variables \eqref{E:large-26},  \eqref{E:u20-small-6}, \eqref{E:u20-small-5(3)}, 
\begin{gather} 
|{{\xi'_{2 }} -(2\pi i{\xi'_{1 }}^2+2{\xi'_{1 }}\overline\lambda' ) }|
\ge\frac \pi 2 |\xi'_1| t^{ -\frac 12+\epsilon_1 } ,  \label{E:u20-small-6(1)}
\end{gather}integration by parts with respect to $\lambda'_{I}$ twice,      \eqref{E:wp-3}, applying Lemma \ref{L:sd-der},  \eqref{E:lambda-L1}, the factor $\theta_{2r_\ast,\frac r\pi}(\frac{\lambda'_I}{\pi})$ (thus $|\lambda'_I|\ge r_\ast/C$),    \eqref{E:large-29-(0)}, and \eqref{E:large-29-(2)}, we obtain
\be  \label{E:u20-small-8}
I_{P^-_{in,>,-}}=  \epsilon_\flat^2 \mathcal O(t^{2{\color{black}(-1+\frac 14-\frac{\epsilon_1}4)}+(-\frac 12+\epsilon_1)-2(-\frac 12+\epsilon_1) } )
 = \epsilon_\flat^2 \mathcal O(t^{-1-\epsilon_1}) .  
\ee  
 
\end{itemize}
\end{itemize} 

To study $J_{P_{in,>,-}}$, using 
\be\label{E:u20-small-8-new}
\begin{gathered}
|\partial_{\lambda'_R}S_0|=|+\frac6\pi\lambda'_R\lambda'_I|\ge t^{-\frac 14+\frac{\epsilon_1}{2}},\\
|\partial_{\lambda'_{ R}}\frac 1{{\xi'_{2 }} -(2\pi i{\xi'_{1 }}^2+2{\xi'_{1 }}\overline\lambda' ) }| 
\le    C| \frac {   \xi' _1  }{({\xi'_{2 }} -(2\pi i{\xi'_{1 }}^2+2{\xi'_{1 }}\overline\lambda' ))^2 }|,
\end{gathered}
\ee integration by parts with respect to $\lambda'_{R}$, the change of variables \eqref{E:large-26},  \eqref{E:u20-small-5(2)} (achieved by using $[1-\theta_{t^{-\frac 14+\frac{\epsilon_1}{2}},0}(\lambda'_R)]$ instead), \eqref{E:u20-small-5(3)},  \eqref{E:large-29-(0)}, and \eqref{E:large-29-(2)}, we obtain
\be  \label{E:u20-small-8-more}
J_{P_{in,>,-}}=   \epsilon_\flat^2 \mathcal O(t^{ {\color{black}(-1+\frac 14-\frac{\epsilon_1}2)}+(-\frac 12+\epsilon_1)- (-\frac 14+\frac{\epsilon_1}4) } )
  =  \epsilon_\flat^2\mathcal O(t^{-1+\epsilon_1}) .  
\ee 

Combining \eqref{E:u20-small-3}, \eqref{E:u20-small-4},  \eqref{E:u20-small-4(add)},  \eqref{E:u20-small-5},   \eqref{E:u20-small-8}, and \eqref{E:u20-small-8-more}, 
we justify \eqref{E:u20-small-2} when $n=0$.

\noindent  {$\blacktriangleright$}  $\underline{\textbf{ Step 2: Proof of  \eqref{E:u20-small-2}   when $n>0$ :}}$ From \eqref{E:small-1}, Lemma \ref{L:wickerhauser}, and the proof of Lemma \ref{L:small} (in particular, $ \textbf{Step 2, 3}$), we first study
\begin{align}
&  \widetilde{(\mathcal C_{\lambda'} T)   }\theta_{2r_\ast, \frac {r}{ \pi}}(\xi'_{1,2}) \circ P_{in,<,-}  (t_1,t_2,t;\overline\lambda'_1) \label{E:u20-small-9}\\
=&\widetilde{(\mathcal C_{\lambda'}  T)   }\theta_{t^{ -\frac 12+\epsilon_1},0}(\frac{\xi'_{2,2}}{2\xi'_{1,2}})\theta_{2r_\ast, \frac {r}{ \pi}}(\xi'_{1,2})\iint_{\Sigma_{<} } dx_1'dx_2' \mathbf v( x'_1,x_2';t_2 ) \nonumber\\
\times& \iint  d{\xi'_{1,1}}d{\xi'_{2,1}} \theta(\frac 14-|\frac{\xi'_{2,1}}{ \xi'_{1,1}\rho_2}|)
  \frac{\theta_{r_\ast,\frac {r}{ \pi}}(\xi_{1,1}')e^{ 2\pi i t \wp}}{(1+\xi'^4_{1,1})({\xi'_{2,1}} -z_{0,1}(x, \lambda'_1)) },\nonumber\\
  & \widetilde{(\mathcal C_{\lambda'} T)   } \theta_{2r_\ast, \frac {r}{ \pi}}(\xi'_{1,2})\circ P_{in,>,-}  (t_1,t_2,t;\overline\lambda_1). \label{E:u20-small-10}\\
=&
\widetilde{(\mathcal C_{\lambda'}  T)   }[1-\theta_{t^{ -\frac 12+\epsilon_1},0}(\frac{\xi'_{2,2}}{2\xi'_{1,2}}) ]\theta_{2r_\ast, \frac {r}{ \pi}}(\xi'_{1,2})\iint_{\Sigma_{<} } dx_1'dx_2' \mathbf v( x'_1,x_2';t_2 ) \nonumber\\
\times& \iint  d{\xi'_{1,1}}d{\xi'_{2,1}} \theta(\frac {1}4-|\frac{\xi'_{2,1} }{ \xi'_{1,1}\rho_1}|)
  \frac{\theta_{r_\ast,\frac {r}{ \pi}}(\xi_{1,1}')e^{ 2\pi i t \wp}}{(1+\xi'^4_{1,1})({\xi'_{2,1}} -z_{0,1}(x,\lambda'_1)) }.\nonumber
\end{align}

Applying \eqref{E:small-19}, the change of variables \eqref{E:large-26}, 
\[   |\eta_{2,2}+4\pi\xi'_{1,2 }\lambda'_R|=2\pi|\xi'_{2,2 }| \le C|\xi'_{1,2}|\rho_1  ,\]  and adapting argument used for $P_{in,\gtrless,-}$ in Lemma \ref{L:small},  we obtain 
\be\label{E:u20-small-11}
\widetilde{(\mathcal C_{\lambda'} T)   }\theta_{2r_\ast, \frac {r}{ \pi}}(\xi'_{1,2})\circ P_{in,<,-}(t_1,t_2,t;\overline\lambda_1') =  \epsilon_\flat^2\mathcal O(t^{-\frac 12+  2\epsilon_2-\frac 12+\epsilon_1} \log t)\le\epsilon_\flat^2\mathcal O(t^{-1+  3\epsilon_2 }  ).
\ee 

On the other hand, notice that $\widetilde{(\mathcal C  T)   }[1-\theta_{t^{ -\frac 12+\epsilon_1},0}(\frac{\xi'_{2,2}}{2\xi'_{1,2}})] \theta_{2r_\ast, \frac {r}{ \pi}}( \xi'_{1,2})$ in the RHS of \eqref{E:u20-small-10} is again of the same type of the  CIO near $\xi'_{1,2}=\frac {r}{ \pi}$, and, via the chain rule,
\be 
\begin{split}
&\textit{$P_{in,>,-} $ is holomorphic in $\lambda'_{R,1}\lambda'_{I,1}$ when $\lambda'_{R,1}\lambda'_{I,1}\ne 0$ ,}\\
&|P_{in,>,-}(t_1,t_2,t;\lambda'_1)|=|P_{in,>,-}(t_1,t_2,t;\frac{\xi'_{2,2}e^{  i\alpha}}{2\xi'_{1,2}}+i\pi\xi'_{1,2})|\\
=&|P_{in,>,-}(t_1,t_2,t;\lambda'_{R,1}\cos\alpha+i( \lambda'_{R,1}\sin\alpha - \lambda'_{I,1})|\le \epsilon_\flat\mathcal O(t^{-\frac 12+ \epsilon_1}),\ |\alpha|\le\frac{\pi}4.
\end{split}
\ee Hence we can   adapt  argument in  Lemma \ref{L:small} to obtain:
\begin{multline}\label{E:u20-small-13}
 \widetilde{(\mathcal C_{\lambda'} T)   }\theta_{2r_\ast, \frac {r}{ \pi}}( \xi'_{1,2}) \circ P_{in,>,-}(t_1,t_2,t;\overline\lambda'_1) 
=   \epsilon_\flat\mathcal O(t^{-\frac 12+  2\epsilon_1} ) [1-\theta_{t^{ -\frac 12+\epsilon_1},0}(\lambda'_R) ]\\
 \times    \mathfrak C \mathfrak T_{0,(2)}  [1- \theta_{t^{ -\frac 12+\epsilon_1},0}(\frac{\xi'_{2,2}}{2\xi'_{1,2}})]\theta_{2r_\ast, \frac {r}{ \pi}}( \xi'_{1,2}) +\epsilon_\flat\epsilon_\sharp  \mathcal O_{\epsilon_1}(t^{-1+3\epsilon_1}). 
\end{multline}

Therefore, it remains to studying the first term on the RHS of \eqref{E:u20-small-13}.  To this aim,
   we will adapt  the argument as that in Lemma \ref{L:cio-0} for  the case of $|\lambda_R'|\ge r/2$ (choosing $\epsilon=\epsilon_1$ and replacing $|\lambda_R'|\ge r/2$ with $|\lambda_R'|\ge \rho_1$ instead).  Let 
 \be\label{E:u20-small-14}
 g= [1-\theta_{t^{ -\frac 12+\epsilon_1},0}(\lambda'_R) ][1-\theta_{t^{ -\frac 12+\epsilon_1},0}(\frac{\xi'_{2,2}}{2\xi'_{1,2}}) ]\theta_{2r_\ast, \frac {r}{ \pi}}( \xi'_{1,2}) ,\ee and decompose
\begin{align}
 &|\mathfrak C\mathfrak T _{ (2)} g|\le  |\mathfrak C\mathfrak T_{ (2)}  g\theta(|x_{2,2}'|-t|\lambda_R'|) |\label{E:u20-small-15}\\
 +&|\mathfrak C\mathfrak T_{ (2)} g[1-\theta( ||\xi_{2}''|-\frac{|\lambda_I'|}{2\pi} |-t^{ -\frac12+\epsilon_1})]\theta(t|\lambda_R'|-|x_{2,2}'|)|\nonumber\\
 +&|\mathfrak C\mathfrak T_{ (2)} g\theta( ||\xi_{2}''|-\frac{|\lambda_I'|}{2\pi} |-t^{-\frac12+\epsilon_1})\theta(t|\lambda_R'|-|x_{2,2}'|)|\equiv I_1+I_2+I_3.
\nonumber
\end{align}
 
 Apparently, $|I_2|\le C\epsilon_\flat t^{-\frac12+\epsilon_1}  $. 
 
On the other hand,  $|I_1|=0 $  follows from  $ |x'_{2,2}|\le t^{\frac 12}$ on $\Sigma_{<}$ and $
\psi_{t^{ -\frac 12+\epsilon_1},0}(\lambda'_R)= 0$. 

Moreover, notice that: using  $
\psi_{t^{ -\frac 12+\epsilon_1},0}(\lambda'_R)= 0$, $\theta_{2r_\ast,\frac{r}{\pi}}(\xi'_{1,2}) =\theta_{2r_\ast,\frac{r}{\pi}}(\xi''_2+\frac{\lambda'_I}{2\pi})$,  we conclude the exponent
\[
    \theta(t|\lambda'_R|-|x_{2,2}'|) \theta( ||\xi_{2}''|-\frac{|\lambda_I'|}{2\pi} |-t^{-\frac12+\epsilon_1})\  
    |(  x_{2,2}'+3  t\lambda'_R)(\xi_{2}''-\frac{\lambda'_I}{2\pi})  (\xi_{2}''+\frac{\lambda'_I}{2\pi}  )|\ge Ct^{2\epsilon_1}.   
\]Hence the amplitude function $|\mathcal F^{(2)}|\le Ce^{-t^{2\epsilon_1}}$ which implies
  $
 | I_3|=\epsilon_\flat o_{\epsilon_1}(t^{-1}) $.

Plugging estimates of $I_1,\, I_2,\, I_3$ for  \eqref{E:u20-small-15} into \eqref{E:u20-small-13}, together with Lemma \ref{L:wickerhauser}, \eqref{E:lambda-L1}, \eqref{E:u20-small-11}, we establish \eqref{E:u20-small-2}   when $n>0$.

\end{proof}

Due to the collision of the singularities of the Cauchy kernel, 
we are unable to obtain the $\rho_2\log t$ decay available for $R_{in,\gtrless,-}$ and $P_{in,\gtrless,-}$. This will lead to an obstruction to obtaining an $\mathcal O(t^{-1+\epsilon})$ estimate for the Cauchy integral.
\begin{lemma}\label{L:0-small} [$\textrm{Estimates  of $\widetilde{\mathcal CT1}$  near $\xi'_{1,1}=0$}$]  
\begin{align}
&\iint dx_1'dx_2' \mathbf v ( x'_1,x_2';t_2 )   \iint d{\xi'_1}d{\xi'_2}\frac{ \theta_{r_\ast,0}(\xi'_1)e^{ 2\pi i t \wp}}{(1+\xi'^4_1)({\xi'_2} -z_0) }\label{E:+0-small-1}\\
&\hskip.8in=    e^{- i\pi tS_0(a;\lambda') } \mathfrak C \mathfrak T_{(1)}\theta_{r_\ast,0}(\xi'_1)     + \epsilon_\flat \mathcal O_{\epsilon_1}(t^{-\frac 14+\epsilon_2} )+\epsilon_\sharp\mathcal O  (t^{-1}).\nonumber
\end{align}

\end{lemma}

\begin{proof}  
Following the same deformation approach used in {\bf Step 2 of Lemma \ref{L:CIO-large}}, we derive 
\be\label{E:+0-small-2}
-\dfrac 1{2\pi i}   \iint_{\Sigma_{>} }dx_1'dx_2' \mathbf v( x'_1,x_2';t_2 ) \iint d{\xi'_1}d{\xi'_2} 
 \frac{\theta_{r_\ast,0}(\xi_1')   e^{ 2\pi i t\wp}}{(1+\xi'^4_1)({\xi'_2} -z_0) } =
\epsilon_\sharp \mathcal O(t^{-1}).
\ee 
Hence it reduces to studying estimates on the regime $\Sigma_{<}$. Decompose 
\begin{align}
&-\dfrac 1{2\pi i}   \iint_{\Sigma_{<} }dx_1'dx_2' \mathbf v( x'_1,x_2';t_2 ) \iint d{\xi'_1}d{\xi'_2} 
 \frac{\theta_{r_\ast,0}(\xi_1')   e^{ 2\pi i t\wp}}{(1+\xi'^4_1)({\xi'_2} -z_0) }\label{E:+0-small-3}\\
= &       \dfrac { -1}{2\pi i}  \iint_{\Sigma_{<} } dx_1'dx_2'\mathbf v( x'_1,x_2';t_2 )  \iint d{\xi'_1}d{\xi'_2}\frac{\theta_{r_\ast, 0}(\xi_1')[1-\theta_{|\xi'_1  |^{1/2}\rho_1,0}(\lambda'_R)] e^{ 2\pi i t\wp}}{(1+\xi'^4_1)({\xi'_2} -z_0) } \nonumber\\
+&   \dfrac  {-1}{2\pi i}  \iint_{\Sigma_{<} } dx_1'dx_2' \mathbf v( x'_1,x_2';t_2 ) \iint  d{\xi'_1}d{\xi'_2}\frac{\theta_{r_\ast, 0}(\xi_1')\theta_{|\xi'_1  |^{1/2}\rho_1,0}(\lambda'_R) e^{ 2\pi i t\wp}}{(1+\xi'^4_1)({\xi'_2} -z_0) }\nonumber\\
\equiv&  Q_{in,> }+ Q_{in,< }  .\nonumber
\end{align}

\noindent  {$\blacktriangleright$}  $\underline{\textbf{ Step 1. Estimates for $Q_{in,>}$:}}$ Introduce
the scaling
\be 
  \omega=\frac{\xi'_2}{|\xi'_1  |^{1/2}\rho_1},\label{E:+0-small-7}
\ee we decompose
\begin{align}
 Q_{in,> }\ \ =&Q_{in,>, +}+Q_{in,>, -} ,\label{E:+0-small-8}\\
Q_{in,>, \pm}=& \dfrac {-1}{2\pi i}  \iint_{\Sigma_{<} } dx_1'dx_2'\mathbf v( x'_1,x_2';t_2 ) \int d{\xi'_1}   e^{ 2\pi i t\wp_1}\frac{\theta_{r_\ast,0}(\xi_1')}{1+\xi'^4_1}  \nonumber\\
\times& \int  d{\omega}[1-\theta_{|\xi'_1  |^{1/2}\rho_1,0}(\lambda'_R)]\theta(\pm|\omega|\mp\frac 14) \frac{e^{ 2\pi i [-\frac 34t \frac{|\xi'_1|}{\xi'_1} \rho_1^2   \omega ^2 -x_2'|\xi'_1|^{1/2}\rho_1  \omega ] }}{\omega-\frac {z_0}{|\xi'_1  |^{1/2}\rho_1} }.\nonumber 
\end{align}

For $Q_{in,>,+}$, thanks to 
\be\label{E:+0-small-9}
\begin{gathered}
1\le |    \frac {z_0}{|\xi'_1|^{1/2}\rho_1}  |,\quad |t\frac{|\xi'_1|}{\xi'_1}\rho_1^2\omega^2|\ge \frac 1C t^{+2\epsilon_1},\quad
 |   t   \rho_1   \omega       |\ge  \frac 1C {t^{1-\frac 12+\epsilon_1 }}  \gg t^{ \frac12 }\ge |   x_2'     | ,
\end{gathered}
\ee and a similar  argument as that for $R_{in,>,+}$ in {\bf Step 1.(A1), Lemma \ref{L:CIO-large}},  
\be  \label{E:+0-small-10}
 Q_{in,>, + }
=  
  e^{ -i\pi tS_0(a;\lambda') } \mathfrak C \mathfrak T_{(1)}\theta_{r_\ast,0}(\xi_1')    + \epsilon_\flat o_{\epsilon_1}( t^{-1}) .   
\ee

  For $Q_{in,>, -}$, using 
  the change of variables \eqref{E:large-26},  \eqref{E:large-29-(1)}, \eqref{E:large-29-(1)-(add)}, 
\be\label{E:+0-small-11}
|\eta_{2}+4\pi\xi'_{1 }\lambda'_R|=2\pi|\xi'_{2 }| \le \frac\pi2|\xi'_1|^{\frac 12}\rho_1  ,
\ee and setting $\eta_1=\sqrt{6\pi\rho_1} u$, $b_0=\sqrt{6\pi\rho_1}\beta$,  we obtain: 
\begin{itemize}
\item if $|\xi'_1\lambda'_R|\le \rho_1  $, then 
\begin{align}
 |Q_{in,>, -}|
\le& C\epsilon_\flat\int_{|\eta_1+\lambda'_I|\le r_\ast }d\eta_1    \int_{|\eta_2 |\le 6\pi\rho_1 }  \frac{d\eta _2}{\sqrt{(\eta_1^2-b^2_0)^2+\eta_2 ^2}}\label{E:+0-small-12}
\\
\le &  C\epsilon_\flat   \int_{|\eta_1+\lambda'_I|\le r_\ast }\texttt{arcsinh}\frac{6\pi\rho_1 }{ |{\eta_1^2-b_0^2}|}d\eta_1  \nonumber\\
\le & C\epsilon_\flat  \sqrt{ \rho_1}\int_0^{r_\ast/\sqrt{6\pi\rho_1}} \texttt{arcsinh}\frac{ 1 }{ |u^2-\beta^2 |}du  \nonumber\\
=&\epsilon_\flat \mathcal O(t^{-\frac 14+ \epsilon_1 } ) .\nonumber\end{align}

\item  if $|\xi' _1\lambda'_R|\ge \rho_1  $,  then 
\begin{align}
|Q_{in,>, -}|
\le & C\epsilon_\flat\int_{|\eta_1+\lambda'_I|\le r_\ast }d\eta_1      \frac{\rho_1}{\sqrt{(\eta_1^2-b^2_0)^2+(6\pi\rho_1) ^2}}\label{E:+0-small-12-(1)}\\
\le & C\epsilon_\flat\sqrt{\rho_1}\int_{0}^{r_\ast/\sqrt{6\pi\rho_1}} du      \frac{1}{\sqrt{(u^2-\beta^2 )^2+1}}\nonumber\\
= & \epsilon_\flat \mathcal O(t^{-\frac 14+ \frac{\epsilon_1}2 } ) .\nonumber
\end{align}
\end{itemize}

Combining \eqref{E:+0-small-10}, \eqref{E:+0-small-12},  and \eqref{E:+0-small-12-(1)}, we obtain
\be \label{E:+0-small-13}
 Q_{in,>}
=   
  e^{- i\pi tS_0(a;\lambda') } \mathfrak C \mathfrak T_{(1)}\theta_{r_\ast,0}(\xi_1')   
 + \epsilon_\flat \mathcal O(t^{-\frac 14+\frac{\epsilon_1}2 } )+\epsilon_\flat o_{\epsilon_1}(t^{-1})  .    
\ee

\noindent  {$\blacktriangleright$}  $\underline{\textbf{ Step 2. Estimates for $Q_{in,<}$:  }}$  Via the scaling
\be\label{E:+0-small-14}
  \omega=\frac{\xi'_2}{|\xi'_1|^{1/2}\rho_2},
\ee we decompose
\begin{align}
 Q_{in,<}\ \ =& Q_{in,<,+}+Q_{in,<,- }, \label{E:+0-small-15}\\
Q_{in,<,\pm}=&\frac{-1}{2\pi i}  \iint_{\Sigma_{<}}  dx_1'dx_2' \mathbf v(x'_1,x'_2;t_2 ) \int  d{\xi'_1}  e^{ 2\pi i t\wp_1 }\frac{\theta_{r_\ast,0}(\xi_1')}{1+\xi'^4_1}\nonumber\\
\times&\int   d{\omega} \theta_{|\xi'_1|^{1/2}\rho_1,0}(\lambda'_R) \theta(\pm|\omega|\mp\frac14) \frac{e^{ 2\pi i [-\frac 34 t \frac{|\xi'|}{\xi'_1}\rho^2_2    {{\omega}^2}  -x_2'|\xi'_1|^{1/2}\rho_2{\omega}]  }}{\omega - \frac{z_0}{|\xi'_1|^{1/2}\rho_2}  }.\nonumber
\end{align}

Using 
\be\label{E:+0-small-16}
\begin{gathered}
  |    \frac {\textit{Re}(z_0)}{ |\xi'_1 |^{1/2}\rho_2}  |\ll 1,\quad   |t\frac{|\xi'_1|}{\xi'_1}\rho_2^2\omega^2|\ge t^{+2\epsilon_2},\quad 
 |   t   \rho _2  \omega       |\ge   {t^{ \frac 12+\epsilon_2  }}  \gg t^{ \frac12 }\ge |   x_2'     | ,
\end{gathered}
\ee and   applying the same argument as that for $R_{in,><,+}$ in {\bf Step 1.(B1), Lemma \ref{L:CIO-large}},
\be \label{E:+0-small-17}
 Q_{in,<, + }
= \epsilon_\flat \theta_{t^{ -\frac 12 +\epsilon_1 },0}(\lambda'_R)  o_{\epsilon_1}( t^{-1}) .   
\ee

For $Q_{in,<, -}$,  
applying  the change of variables \eqref{E:large-26},
\be \label{E:+0-small-18}
|\eta_{2}+4\pi\xi'_{1 }\lambda'_R|=2\pi|\xi'_{2 }| \le C|\xi'_1|^{\frac 12}\rho_2  .
\ee  
 and argument for \eqref{E:+0-small-12},  we have  
\be\label{E:+0-small-19}
 Q_{in,<, -} 
=  \epsilon_\flat \theta_{t^{ -\frac 12 +\epsilon_1 },0}(\lambda'_R)  \mathcal O(t^{-\frac 14+   \epsilon_2}).\ee

Combining $ \textbf{Step 1}$ and $ \textbf{Step 2}$,  \eqref{E:+0-small-1} is established.

\end{proof}

Although $0$ is not a stationary point of $\wp_1$,   estimating   $u_{2,0}$ near $\xi'_{1,1}=0$ is more delicate because the singularities of the Cauchy kernel $\frac 1{\xi'_2-z_0}$ may collide. This creates an obstruction to obtaining an $\mathcal O(t^{-1})$ estimate  (see \eqref{E:q-<0}, \eqref{E:u20-0-small-3}, and \eqref{E:u20-small-14(Q<)} below for  $Q_{in,<,-}$). 
Before analyzing $u_{2,0}$ near $\xi'_{1,1}=0$, we establish three technical lemmas concerning $Q{in,\lessgtr,-}$.
\begin{lemma}\label{L:u20-0-pre-<}
 For $Q_{in,<,-}$, defined by \eqref{E:+0-small-14} and \eqref{E:+0-small-15},  if $\kappa<1/3$ then
\begin{align}
&[1-\theta_{t^{-\kappa},0}(\lambda'_I)] \iint_{\Sigma_{<}}  dx_1'dx_2' \mathbf v(x'_1,x'_2;t_2 ) \iint  d{\xi'_1}  d{\xi'_2}  \frac{e^{ 2\pi i t\wp } }{(1+\xi'^4_1)(\xi'_2-z_0)}\label{E:q-<0} \\
\times& \theta(\frac 18|\lambda'_I|-|\pi\xi'_1-\lambda'_I|)]\theta_{r_\ast,0}(\xi_1')\theta_{|\xi'_1|^{1/2}\rho_1,0}(\lambda'_R) \theta(\frac14-|\omega|)  \nonumber\\
= &\epsilon _\flat\theta_{\rho_1,0}(\lambda'_R)[1-\theta_{t^{-\kappa},0}(\lambda'_I)] \mathcal O(t^{-\frac 12 +\frac\kappa 2+ 2\epsilon_2}).\nonumber
\end{align}

\end{lemma}
\begin{proof} We can not apply integration by parts to gain extra decay despite $0$ is not the stationary point of $\wp_1$.  

Without loss of generality, we assume 
$ \lambda'_I=b_0\ge 0 $. Notice that, in the regime, $|\xi'_1|\sim|\lambda'_I|\ge t^{-\kappa}$. Therefore, via the change of variables \eqref{E:large-26}, applying \eqref{E:large-29-(1)}, \eqref{E:large-29-(1)-(add)}, \eqref{E:+0-small-19}, setting $y=\xi'-2b_0$, and   $\kappa<1/3$, we obtain
\begin{align}
& \textit{LHS of \eqref{E:q-<0}} \label{E:u20-0-pre-10}\\
\le &  C\epsilon _\flat\theta_{\rho_1,0}(\lambda'_R)[1-\theta_{t^{-\kappa},0}(\lambda'_I)]  \int_{Cr_\ast>|\eta_1+\lambda'_I|> t^{-\kappa }/C}d\eta_1\int_{|\eta_2| \le C\rho_2|\xi'_1|^{\frac12}}\frac{d\eta_2}{\sqrt{(\eta_1^2-b_0^2)^2+\eta_2^2}}\nonumber\\
\le & C\epsilon _\flat\theta_{\rho_1,0}(\lambda'_R)[1-\theta_{t^{-\kappa},0}(\lambda'_I)]\int_{ Cr_\ast>|\xi'_1|>t^{-\kappa }/C}d\xi'_1\int_{|\eta_2| \le C\rho_2|\xi'_1|^{\frac12}}\frac{d\eta_2}{\sqrt{(\xi'_1-2b_0)^2\xi'^2_1+\eta_2^2}}\nonumber\\
\le &    C\epsilon _\flat\theta_{\rho_1,0}(\lambda'_R)[1-\theta_{t^{-\kappa},0}(\lambda'_I)]\int_{ Cr_\ast>|\xi'_1|>t^{-\kappa }/C, |\xi'_1-2b_0|\ge b_0}d\xi'_1\texttt{arcsinh} \frac{\rho_2}{\xi'^{1/2}_1(\xi'_1-2b_0)} \nonumber\\
+& C\epsilon _\flat\theta_{\rho_1,0}(\lambda'_R)[1-\theta_{t^{-\kappa},0}(\lambda'_I)]\int_{ Cr_\ast>|\xi'_1|>t^{-\kappa }/C,\,|\xi'_1-2b_0|\le b_0}d\xi'_1\texttt{arcsinh} \frac{\rho_2}{\xi'^{1/2}_1(\xi'_1-2b_0)} \nonumber\\
\le & C\epsilon _\flat\theta_{\rho_1,0}(\lambda'_R)[1-\theta_{t^{-\kappa},0}(\lambda'_I)]\int_{ Cr_\ast>|\xi'_1|>t^{-\kappa }/C}d\xi'_1 (\rho_2\xi'^{-3/2}_1)\nonumber\\
+ & C\epsilon _\flat\theta_{\rho_1,0}(\lambda'_R)[1-\theta_{t^{-\kappa},0}(\lambda'_I)]\int_{ b_0>|y| }dy \texttt{arcsinh} \frac{\rho_2}{ b_0^{1/2} y}\nonumber\\
= &\epsilon _\flat\theta_{\rho_1,0}(\lambda'_R)[1-\theta_{t^{-\kappa},0}(\lambda'_I)] \mathcal O(t^{-\frac 12 +\frac\kappa 2+2 \epsilon_2}).\nonumber
\end{align}  

\end{proof}

To study $Q_{in,>,-}$, we will need:
\begin{lemma}\label{L:tech}
\begin{align}
\int_{ r_\ast>|\eta_1+b_0|=|\xi'_1|> t^{-\kappa } }d\eta_1\int_{|\eta_2| \le C\rho_1|\xi'_1|^{\frac12}}\frac{d\eta_2}{|\xi'_1|\sqrt{(\eta_1^2-b_0^2)^2+\eta_2^2}}  =&o_{\kappa}(t^{+\epsilon_1}),\label{E:tech-1}\\
\int_{ r_\ast>|\eta_1+b_0|=|\xi'_1|> t^{-\kappa } }d\eta_1\int_{|\eta_2| \le C\rho_1|\xi'_1|^{\frac12}}\frac{d\eta_2}{|\xi'_1|^{1/2}\sqrt{(\eta_1^2-b_0^2)^2+\eta_2^2}}= &\mathcal O(t^{-\frac 16+\frac{\epsilon_1}3}),\label{E:tech-2}\\
\int_{ r_\ast>|\eta_1+b_0|=|\xi'_1|> t^{-\kappa } }d\eta_1 \frac{\rho_1|\xi'_1|^{\frac12}}{|\xi'_1|\sqrt{(\eta_1^2-b_0^2)^2+(\rho_1|\xi'_1|^{\frac12})^2}}= &o_\kappa( t^{+\epsilon_1}),\label{E:tech-3}\\
\int_{ r_\ast>|\eta_1+b_0|=|\xi'_1|> t^{-\kappa } }d\eta_1 \frac{\rho_1|\xi'_1|^{\frac12}}{|\xi'_1|^{1/2}\sqrt{(\eta_1^2-b_0^2)^2+(\rho_1|\xi'_1|^{\frac12})^2}}= &\mathcal O(t^{-\frac 16+\frac{\epsilon_1}3} ).\label{E:tech-4}
\end{align}

\end{lemma}

\begin{proof}
 Applying \eqref{E:large-29-(1)}, \eqref{E:large-29-(1)-(add)}, setting $\xi'_1 =\rho_1^{2/3}u$, $b_0=\rho_1^{2/3}\beta$,
\begin{align}
&\textit{LHS of \eqref{E:tech-1}}\label{E:tech-5}\\
\le& C\int_{ r_\ast>| \xi'_1| > t^{-\kappa } }d\xi'_1\int_{|\eta_2| \le C\rho_1|\xi'_1|^{\frac12}}\frac{d\eta_2}{ |\xi'_1| \sqrt{(\xi'_1-2b_0)^2\xi'^2_1 +\eta_2^2}}\nonumber\\
\le& C \int_{|u|\ge t^{-\kappa}/\rho_1^{2/3}} du\frac 1{|u| }\texttt{arcsinh} \frac{ 1}{|u|^{1/2} |u-2\beta|} 
\le C(1+\kappa^2(\log t)^2));\nonumber\\
&\textit{LHS of \eqref{E:tech-2}}\label{E:tech-6}\\
\le& C\int_{ r_\ast>| \xi'_1| > t^{-\kappa } }d\xi'_1\int_{|\eta_2| \le C\rho_1|\xi'_1|^{\frac12}}\frac{d\eta_2}{ |\xi'_1|^{1/2} \sqrt{(\xi'_1-2b_0)^2\xi'^2_1 +\eta_2^2}}\nonumber\\
\le& C\rho_1^{1/3}\int du\frac 1{|u|^{1/2}}\texttt{arcsinh} \frac{ 1}{|u|^{1/2} |u-2\beta|} 
\le C t^{-\frac 16+\frac{\epsilon_1}3}  ;\nonumber\\
&\textit{LHS of \eqref{E:tech-3}}\label{E:tech-7}\\
\le& C\int_{ r_\ast>| \xi'_1| > t^{-\kappa } }d\xi'_1\frac{\rho_1 }{|\xi'_1|\sqrt{\xi' (\xi'_1-2b_0 )^2+ \rho_1   ^2}}\nonumber\\
\le& C \int_{ r_\ast/\rho_1^{2/3} >u > t^{-\kappa }/\rho_1^{2/3} } du\frac{ 1 }{|u|\sqrt{u (u-2\beta )^2+  1   }} 
\le C \log t  ;\nonumber\\
&\textit{LHS of \eqref{E:tech-4}}\label{E:tech-8}\\
\le& C\int_{ r_\ast>| \xi'_1| > t^{-\kappa } }d\xi'_1\frac{\rho_1 }{|\xi'_1|^{1/2}\sqrt{\xi' (\xi'_1-2b_0 )^2+ \rho_1   ^2}}\nonumber\\
\le& C \rho^{1/3}_1\int_{ r_\ast/\rho_1^{2/3} >u > t^{-\kappa }/\rho_1^{2/3} } du\frac{ 1 }{|u|^{1/2}\sqrt{u (u-2\beta )^2+  1   }} 
\le C t^{-\frac 16+\frac{\epsilon_1}3} .\nonumber
\end{align}

\end{proof}

\begin{lemma}\label{L:u20-0-pre->} For $Q_{in,>,-}$, defined by \eqref{E:+0-small-7} and \eqref{E:+0-small-8},
\begin{align}
   Q_{in,>,-}= & \epsilon_\flat\mathcal O(t^{-\frac 23+\epsilon_1}), \label{E:q->0} \\
 [1-\theta_{\rho_1,0}(\lambda'_R)] [1-\theta_{\sqrt{\rho_1},0}(\lambda'_I)] 
  \partial_{\lambda'_I} Q_{in,>,-}(t_1,t_2,t;\lambda')=&\epsilon _\flat  \mathcal O(t^{+\epsilon_1}),\label{E:q->0-more}\\
 [1-\theta_{\rho_1,0}(\lambda'_R)]  \partial_{\lambda'_I}Q_{in,>,-}(t_1,t_2,t;\lambda') = & \epsilon_\flat\mathcal O(t^{\frac 14}).\label{E:q->0-add}
\end{align}

\end{lemma}

\begin{proof} 

\noindent  {$\blacktriangleright$}  $\underline{\textbf{ Step 1. Proof for \eqref{E:q->0}:  }}$ Let $\kappa$ be a positive constant. 
Notice that, via the same argument used for \eqref{E:+0-small-12} and \eqref{E:+0-small-12-(1)},
\begin{align}
 &| \iint_{\Sigma_{<}}  dx_1'dx_2' \mathbf v(x'_1,x'_2;t_2 ) \iint  d{\xi'_1}  \frac{e^{ 2\pi i t\wp }}{(1+\xi'^4_1)(\xi'_2-z_0)}\label{E:u20-0-pre-0}\\
\times&\theta_{t^{-\kappa},0}(\xi'_1)\theta_{r_\ast,0}(\xi_1')[1-\theta_{|\xi'_1  |^{1/2}\rho_1,0}(\lambda'_R)] \theta(\frac 14-|\omega|)  |= \epsilon_\flat\mathcal O(t^{-\frac 14-\frac \kappa{4}+\epsilon_1}).\nonumber
\end{align}

Besides, observe
\begin{align}
\bullet\quad&[1-\theta_{t^{-\kappa},0}(\xi'_1)] [1-\theta_{|\xi'_1  |^{1/2}\rho_1,0}(\lambda'_R)] \theta(\frac 14-|\omega|) |\partial_{\xi'_1}\frac 1{{\xi'_2} -(2\pi i{\xi'_1}^2+2{\xi'_1}\overline\lambda') }|
\label{E:u20-0-pre-1}\\
 =&[1-\theta_{t^{-\kappa},0}(\xi'_1)] [1-\theta_{|\xi'_1  |^{1/2}\rho_1,0}(\lambda'_R)] \theta(\frac 14-|\omega|) \nonumber\\
 \times&| \frac { i(2\pi  {\xi'_1}^2-2\xi_1'\lambda'_I)}{\xi'_1({\xi'_2} -(2\pi i{\xi'_1}^2+2{\xi'_1}\overline\lambda'))^2 }+\frac {2  \pi i\xi'_1 }{ ({\xi'_2} -(2\pi i{\xi'_1}^2+2{\xi'_1}\overline\lambda'))^2 } +\frac {   \lambda'_R }{ ({\xi'_2} -(2\pi i{\xi'_1}^2+2{\xi'_1}\overline\lambda'))^2 }|  \nonumber\\
\le & C[1-\theta_{t^{-\kappa},0}(\xi'_1)][1-\theta_{|\xi'_1  |^{1/2}\rho_1,0}(\lambda'_R)] \theta(\frac 14-|\omega|) \nonumber\\
\times&(| \frac { 1}{\xi'_1({\xi'_2} -(2\pi i{\xi'_1}^2+2{\xi'_1}\overline\lambda')) }|+C| \frac { 1}{\lambda'_R({\xi'_2} -(2\pi i{\xi'_1}^2+2{\xi'_1}\overline\lambda')) }|)
\nonumber\\
\le & C[1-\theta_{t^{-\kappa},0}(\xi'_1)][1-\theta_{|\xi'_1  |^{1/2}\rho_1,0}(\lambda'_R)] \theta(\frac 14-|\omega|) \nonumber\\
\times&(| \frac { 1}{\xi'_1({\xi'_2} -(2\pi i{\xi'_1}^2+2{\xi'_1}\overline\lambda')) }|+C| \frac { \rho_1^{-1}}{|\xi'_1|^{1/2}({\xi'_2} -(2\pi i{\xi'_1}^2+2{\xi'_1}\overline\lambda')) }|),
\nonumber\\
\bullet\quad &[1-\theta_{t^{-\kappa},0}(\xi'_1)][1-\theta_{|\xi'_1  |^{1/2}\rho_1,0}(\lambda'_R)] \theta(\frac 14-|\omega|)|\frac{\partial_{\xi'_1}e^{2\pi i [-t  \frac 34\frac{{\xi'_2}^2}{\xi'_1} -x_2'{\xi'_2}  ]}}{\xi'_2-(2\pi i{\xi'_1}^2+2{\xi'_1}\overline\lambda')} | \label{E:u20-0-pre-2}\\
\le& [1-\theta_{t^{-\kappa},0}(\xi'_1)][1-\theta_{|\xi'_1  |^{1/2}\rho_1,0}(\lambda'_R)] \theta(\frac 14-|\omega|)\frac{|  2\pi i (t\frac 34\frac{{\xi_2'}^2}{{\xi'_1}^2} )e^{2\pi i [-t  \frac 34\frac{{\xi'_2}^2}{\xi'_1} -x_2'{\xi'_2}  ]}|}{\xi'_2-(2\pi i{\xi'_1}^2+2{\xi'_1}\overline\lambda')} \nonumber\\
\le&C [1-\theta_{t^{-\kappa},0}(\xi'_1)][1-\theta_{|\xi'_1  |^{1/2}\rho_1,0}(\lambda'_R)] \theta(\frac 14-|\omega|)\frac{   t^{+2\epsilon_1}  }{|\xi'_1(\xi'_2-(2\pi i{\xi'_1}^2+2{\xi'_1}\overline\lambda'))|} .\nonumber
\end{align}

As a result, integration by parts with respect to $\xi'_1$ (taking advantage of $0$ is not the stationary point of $\wp_1$) and applying Lemma \ref{L:tech}, we obtain
\begin{align}
 &| \iint_{\Sigma_{<}}  dx_1'dx_2' \mathbf v(x'_1,x'_2;t_2 ) \int  d{\xi'_1}  e^{ 2\pi i t\wp_1 }\frac{\theta_{r_\ast,0}(\xi_1')}{1+\xi'^4_1}[1-\theta_{t^{-\kappa},0}(\xi'_1)]\label{E:u20-0-pre-3}\\
\times&\int  d{\xi'_2}[1-\theta_{|\xi'_1  |^{1/2}\rho_1,0}(\lambda'_R)] \theta(\frac 14-|\omega|) \frac{e^{ 2\pi i t\wp_2 }}{\xi'_2- z_0 }|\nonumber\\
\le &\epsilon_\flat\mathcal O_\kappa(t^{-1+\epsilon_1}) +\epsilon_\flat\mathcal O(t^{-1+(\frac 12-\epsilon_1)+(-\frac 16+\frac {\epsilon_1}3)})+\epsilon_\flat\mathcal O_\kappa(t^{-1+ \epsilon_1+2\epsilon_1  )}).\nonumber
\end{align}

Combining \eqref{E:u20-0-pre-0} and \eqref{E:u20-0-pre-3}, and taking $\kappa=5/3$, we establish \eqref{E:q->0}.

\noindent  {$\blacktriangleright$}  $\underline{\textbf{ Step 2. Proof for   \eqref{E:q->0-more}:  }}$ Decompose
\begin{align}
&[1-\theta_{\rho_1,0}(\lambda'_R)] [1-\theta_{\sqrt{\rho_1},0}(\lambda'_I)] 
\iint_{\Sigma_{<}}  dx_1'dx_2' \mathbf v(x'_1,x'_2;t_2 ) \iint  d{\xi'_1}  d{\xi'_2}  \label{E:u20-0-pre-4} \\
\times&  \theta_{r_\ast,0}(\xi_1') \theta(\frac 14-|\omega|)\partial_{\lambda'_I} \frac{e^{ 2\pi i t\wp }}{(1+\xi'^4_1)(\xi'_2-z_0)}\nonumber\\
=&[1-\theta_{\rho_1,0}(\lambda'_R)] [1-\theta_{\sqrt{\rho_1},0}(\lambda'_I)] 
\iint_{\Sigma_{<}}  dx_1'dx_2' \mathbf v(x'_1,x'_2;t_2 ) \iint  d{\xi'_1}  d{\xi'_2}  \nonumber \\
\times&\theta(|\pi\xi'_1-\lambda'_I|-\frac 18|\lambda'_I|)  \theta_{r_\ast,0}(\xi_1')  \theta(\frac 14-|\omega|)\partial_{\lambda'_I} \frac{e^{ 2\pi i t\wp }}{(1+\xi'^4_1)(\xi'_2-z_0)}\nonumber\\
+&[1-\theta_{\rho_1,0}(\lambda'_R)] [[1-\theta_{\sqrt{\rho_1},0}(\lambda'_I)] 
\iint_{\Sigma_{<}}  dx_1'dx_2' \mathbf v(x'_1,x'_2;t_2 ) \iint  d{\xi'_1}  d{\xi'_2}  \nonumber \\
\times&\theta(\frac 18|\lambda'_I|-|\pi\xi'_1-\lambda'_I|)  \theta_{r_\ast,0}(\xi_1')  \theta(\frac 14-|\omega|)\partial_{\lambda'_I} \frac{e^{ 2\pi i t\wp }}{(1+\xi'^4_1)(\xi'_2-z_0)}\nonumber\\
\equiv&Q_{in,>,-}^{(1)}+Q_{in,>,-}^{(2)}.\nonumber
\end{align}

For $Q_{in,>,-}^{(1)}$, notice that 
\begin{align}
 &[1-\theta_{\sqrt{\rho_1},0}(\lambda'_I)]  \theta(|\pi\xi'_1-\lambda'_I|-\frac 18|\lambda'_I|)[1-\theta_{\rho_1 ,0}(\lambda'_R)] \theta(\frac 14-|\omega|) |\partial_{\lambda'_I}\frac 1{{\xi'_2} -(2\pi i{\xi'_1}^2+2{\xi'_1}\overline\lambda') }|
\label{E:u20-0-pre-5}\\
\le & [1-\theta_{\sqrt{\rho_1},0}(\lambda'_I)] \theta(|\pi\xi'_1-\lambda'_I|-\frac 18|\lambda'_I|)[1-\theta_{\rho_1 ,0}(\lambda'_R)] \theta(\frac 14-|\omega|) |  \frac {2 \xi'_1}{({\xi'_2} -(2\pi i{\xi'_1}^2+2{\xi'_1}\overline\lambda'))^2 }|\nonumber\\
\le& C[1-\theta_{ \rho_1,0}(\lambda'_R)] \theta(\frac 14-|\omega|)   \frac {\sqrt{\rho_1}}{| {\xi'_2} -(2\pi i{\xi'_1}^2+2{\xi'_1}\overline\lambda')| } . \nonumber 
\end{align}Along with \eqref{E:+0-small-12} and \eqref{E:+0-small-12-(1)}, yields
\be\label{E:u20-0-pre-6}
Q_{in,>,-}^{(1)}=\epsilon_\flat\mathcal O(t^{-\frac 14+\epsilon_1+\frac 14-\frac{\epsilon_1}{2}}).
\ee

For $Q_{in,>,-}^{(2)}$, due to $\theta(\frac 18|\lambda'_I|-|\pi\xi'_1-\lambda'_I|)$ and $[1-\theta_{t^{-\kappa},0}(\xi'_1)]$, we have
\be \label{E:u20-0-pre-7}
|\xi'_1|\sim|\lambda'_I|\ge \sqrt{\rho_1},
\ee which implies
\begin{align}
 &[1-\theta_{\sqrt{\rho_1},0}(\xi'_1)]\theta(\frac 18|\lambda'_I|-|\pi\xi'_1-\lambda'_I|) [1-\theta_{\rho_1 ,0}(\lambda'_R)] \theta(\frac 14-|\omega|) |\partial_{\xi'_1}\partial_{\lambda'_I}\frac 1{{\xi'_2} -z_0 }|
\label{E:u20-0-pre-8}\\
 =&[1-\theta_{\sqrt{\rho_1},0}(\xi'_1)]\theta(\frac 18|\lambda'_I|-|\pi\xi'_1-\lambda'_I|) [1-\theta_{ \rho_1,0}(\lambda'_R)] \theta(\frac 14-|\omega|) \nonumber\\
 \times&[| \frac {\xi'_1\times i(2\pi  {\xi'_1}^2-2\xi_1'\lambda'_I)}{\xi'_1({\xi'_2} -(2\pi i{\xi'_1}^2+2{\xi'_1}\overline\lambda'))^3 }+\frac {\xi'_1\times 2  \pi i\xi'_1 }{ ({\xi'_2} -(2\pi i{\xi'_1}^2+2{\xi'_1}\overline\lambda'))^3 } +\frac { \xi'_1  \lambda'_R }{ ({\xi'_2} -(2\pi i{\xi'_1}^2+2{\xi'_1}\overline\lambda'))^3 }|  \nonumber\\
 +&| \frac {1}{ ({\xi'_2} -(2\pi i{\xi'_1}^2+2{\xi'_1}\overline\lambda'))^2  }|]\nonumber\\
\le & C[1-\theta_{\sqrt{\rho_1},0}(\xi'_1)]\theta(\frac 18|\lambda'_I|-|\pi\xi'_1-\lambda'_I|)[1-\theta_{ \rho_1,0}(\lambda'_R)] \theta(\frac 14-|\omega|) \nonumber\\
\times&(| \frac { 1}{ \xi'_1\lambda'_R({\xi'_2} -(2\pi i{\xi'_1}^2+2{\xi'_1}\overline\lambda'))  }|+C| \frac { 1}{\lambda'^2_R({\xi'_2} -(2\pi i{\xi'_1}^2+2{\xi'_1}\overline\lambda')) }|)
\nonumber\\
\le &  C [1-\theta_{ \rho_1,0}(\lambda'_R)] \theta(\frac 14-|\omega|)  | \frac { t^{ \frac 34-\frac 32\epsilon_1} + t^{1-2\epsilon_1}}{  {\xi'_2} -(2\pi i{\xi'_1}^2+2{\xi'_1}\overline\lambda')  }|. 
\nonumber
\end{align}

Taking advantage of $0$ is not a stationary point of $\wp_1$, take integration by part with respect to $\xi'_{1}$, and \eqref{E:u20-0-pre-8}, we have
\be\label{E:u20-0-pre-9}
Q_{in,>,-}^{(2)}=\epsilon_\flat\mathcal O(t^{-\frac 12-\frac 12\epsilon_1}+t^{-\frac 14-\epsilon_1}).
\ee

Combining \eqref{E:u20-0-pre-4}, \eqref{E:u20-0-pre-6},   \eqref{E:u20-0-pre-9},  we establish \eqref{E:q->0-more}.

\noindent  {$\blacktriangleright$}  $\underline{\textbf{ Step 3. Proof for \eqref{E:q->0-add}:  }}$ 
In this case, due to composing with $Q_{in,<,-}$, we cannot take integration by part with respect to $\xi'_{1,2}$ for $Q_{in,>,-}$ although $0$ is not a stationary point of $\wp_1$.

Instead, using \eqref{E:+0-small-12}, \eqref{E:+0-small-12-(1)}, \eqref{E:+0-small-19}, and
\begin{align}
 & [1-\theta_{\rho_1 ,0}(\lambda'_R)] \theta(\frac 14-|\omega|) |\partial_{\lambda'_I}\frac 1{{\xi'_2} -(2\pi i{\xi'_1}^2+2{\xi'_1}\overline\lambda') }|
\nonumber\\
\le & [1-\theta_{\rho_1 ,0}(\lambda'_R)] \theta(\frac 14-|\omega|) |  \frac {2 \xi'_1}{({\xi'_2} -(2\pi i{\xi'_1}^2+2{\xi'_1}\overline\lambda'))^2 }|\nonumber\\
\le& C[1-\theta_{ \rho_1,0}(\lambda'_R)] \theta(\frac 14-|\omega|)   \frac {t^{\frac 12-\epsilon_1}}{| {\xi'_2} -(2\pi i{\xi'_1}^2+2{\xi'_1}\overline\lambda')| } . \nonumber 
\end{align}

\end{proof}

Using Lemmas \ref{L:small}--\ref{L:u20-0-pre->}, we refine the argument of Lemma \ref{L:u20-small} to establish the following estimates for $u_{2,0}$ near $\xi'_{1,1}=0$.
\begin{lemma}\label{L:u20-0-small}   Suppose that $|a|>\delta>0$. As $t\to\infty$,  
\begin{align}
&-\dfrac 1{\pi i} \iint d\overline\lambda'\wedge d\lambda' e^{2\pi i tS_0 } \widetilde {  s}_c(\lambda'  )      (\overline\lambda'-\lambda')\widetilde{(\mathcal CT)^n  }  \label{E:u20+0-small-1}\\
\times&\iint dx_{1,1}'dx_{2,1}' \mathbf v ( x'_{1,1},x_{2,1}';t_2 ) \iint d{\xi'_{1,1}}d{\xi'_{2,1}}\frac{ \theta_{r_\ast,0}(\xi'_{1,1})e^{ 2\pi i t \wp}}{(1+\xi'^4_{1,1})({\xi'_{2,1}} -z_{0,1}(\xi'_{1,1},\lambda'_1)) }  \nonumber\\
= & \iint d\overline\lambda'\wedge d\lambda' e^{ 2\pi i tS_0 } \widetilde {  s}_c(\lambda'  )      (\overline\lambda'-\lambda')
\widetilde{(\mathcal CT)^n  } \circ \left(e^{ -i\pi tS_0(a;\lambda'_1) }\mathfrak C \mathfrak T_{(1)}\theta_{r_\ast,0}(\xi_{1,1}')    \right)(t_1,t_2,t;\overline\lambda'_1) \nonumber\\
+&\epsilon^{n+1 }_\flat\epsilon_\sharp \mathcal O_{\epsilon_1}(t^{-\frac{19}{20}+3\epsilon_2} ).\nonumber
\end{align} Here   $\lambda'_1=\frac{\xi'_{2,2}}{2\xi'_{1,2}}-i\pi\xi'_{1,2}=\lambda'_{R,1}+i\lambda'_{I,1}$ and $z_{0,1}= 2\pi i\xi'^2_{1,1}+2\xi'_{1,1}\overline\lambda'_1 $. 

\end{lemma} 
 
\begin{proof}   

To prove the lemma, from  Lemmas \ref{L:wickerhauser}, \ref{L:0-small}, and \eqref{E:lambda-L1},   it suffices to study
\begin{align}
&-\dfrac 1{\pi i} \iint d\overline\lambda'\wedge d\lambda' e^{2\pi i tS_0 } \widetilde {  s}_c(\lambda'  )      (\overline\lambda'-\lambda')\widetilde{(\mathcal CT)^{n}  }\circ Q_{in,\gtrless,-}(t_1,t_2,t;\overline\lambda'_1).  \label{E:u20-0-small-2}
\end{align} 

 Moreover, replacing \eqref{E:uCIO-large-10} by
\begin{align}
   &\partial_{\xi'_1}\frac 1{{\xi'_2} -(2\pi i{\xi'_1}^2+2{\xi'_1}\overline\lambda') }   
=   \frac {  4\pi i {\xi'_1} -2i \lambda'_I }{ ({\xi'_2} -(2\pi i{\xi'_1}^2+2{\xi'_1}\overline\lambda'))^2 }+\frac {   2 \lambda'_R }{ ({\xi'_2} -(2\pi i{\xi'_1}^2+2{\xi'_1}\overline\lambda'))^2 }, \label{E:large-10(add)}
\end{align} following the approach of {\bf Step 1 }of Lemma \ref{L:u20-large}, we prove
\be\label{E:+0-small-4}
Q_{out,\gtrless}\equiv [1- \theta_{2r_\ast, 0}(\lambda'_I)]Q_{in,\gtrless,-}=\epsilon_\sharp \mathcal O(t^{-1}).
\ee 

 Therefore, the study of \eqref{E:u20-0-small-2} reduces to investigating 
\begin{align}
&  \iint d\overline\lambda'\wedge d\lambda' e^{2\pi i tS_0 } \widetilde {  s}_c(\lambda'  )      (\overline\lambda'-\lambda')\widetilde{(\mathcal CT)^{n}  }\theta_{2r_\ast, 0}(\xi'_{1,2}  )\circ Q_{in,\gtrless,-}(t_1,t_2,t;\overline\lambda'_1).  \label{E:+0-small-5}
\end{align}

\noindent  {$\blacktriangleright$}  $\underline{\textbf{ Step 1. Estimates for \eqref{E:+0-small-5} for $Q_{in,<,-}$ when $n=0$ :}}$ 
   Let $\kappa<1/3$ be a positive constant. Applying \eqref{E:lambda-L1},  \eqref{E:+0-small-19},
\begin{align} 
&| -\dfrac 1{\pi i} \iint d\overline\lambda'\wedge d\lambda' e^{2\pi i tS_0 } \widetilde {  s}_c(\lambda'  )      (\overline\lambda'-\lambda')\theta_{2r_\ast, 0}(\xi'_{1,2}  )Q_{in,<,-}|\label{E:u20-0-small-3}\\
\le & | -\dfrac 1{\pi i} \iint d\overline\lambda'\wedge d\lambda' e^{2\pi i tS_0 } \widetilde {  s}_c(\lambda'  )      (\overline\lambda'-\lambda')\theta_{2r_\ast, 0}(\xi'_{1,2}  ) \theta_{t^{-\kappa},0}(\lambda'_I) Q_{in,<,-}|\nonumber\\
+& | -\dfrac 1{\pi i} \iint d\overline\lambda'\wedge d\lambda' e^{2\pi i tS_0 } \widetilde {  s}_c(\lambda'  )      (\overline\lambda'-\lambda')\theta_{2r_\ast, 0}(\xi'_{1,2}  )[1-\theta_{t^{-\kappa},0}(\lambda'_I)]Q_{in,<,-}|\nonumber \\
\equiv&\epsilon_\flat^2 \mathcal O(t^{-\frac 12+\epsilon_1-\frac 14+ \epsilon_2-2 \kappa})+I_{Q_{in,<,-}} = \epsilon_\flat^2 \mathcal O(t^{-\frac 34-2\kappa+ \frac32\epsilon_2  })+I_{Q_{in,<,-}}.\nonumber
\end{align}
 
To analyse $I_{Q_{in,<,-}}$, we consider:
\begin{itemize}

\item [($\textrm{\bf A}_1 $)] First of all,  if 
\be\label{E:u20-small-1(Q<)}
| \lambda'_R |\ge  \min\{1,r_\ast\},
\ee then
\be\label{E:u20-small-2(Q<)}
 |\partial_{\lambda'_R}S_0 |=|+\frac6\pi\lambda'_R\lambda'_I|\sim |\lambda'_I|.
\ee

As a result, using the change of variables \eqref{E:large-26},  \eqref{E:+0-small-15} (the definition of $Q_{in, <,-}$), $\kappa<1/3$, 
\begin{gather}
  |\partial_{\lambda'_{ R}}\frac 1{{\xi'_{2 }} -(2\pi i{\xi'_{1 }}^2+2{\xi'_{1 }}\overline\lambda' ) }| 
\le    C| \frac {   \xi' _1  }{({\xi'_{2 }} -(2\pi i{\xi'_{1 }}^2+2{\xi'_{1 }}\overline\lambda' ))^{2} }|, \label{E:u20-small-3(Q<)}\\
|{{\xi'_{2 }} -(2\pi i{\xi'_{1 }}^2+2{\xi'_{1 }}\overline\lambda' ) }|
\ge \frac 1C  |\xi'_1|  , \quad
|\eta_2+4\pi\xi'_{1 }\lambda'_R|=2\pi|\xi'_{2 }| \le C|\xi'_1|^{1/2}\rho_1  , \label{E:u20-small-4(Q<)}
\end{gather}integration by parts with respect to $\lambda'_R$, Lemma \ref{L:sd-der},  the factor $\overline\lambda'-\lambda'$, and  \eqref{E:+0-small-19}, we obtain
\be\label{E:u20-small-5(Q<)}
I_{Q_{in,<,-}}=  \epsilon_\flat^2\mathcal O(t^{ ( -\frac 12+ \epsilon_1 )-1 +( -\frac 14+ \epsilon_1 ) } ) =   \epsilon_\flat^2\mathcal O(t^{-\frac 74+2\epsilon_1 }).  
\ee

\item [($\textrm{\bf A}_2 $)]  From ($\textrm{\bf A}_1 $), it reduces to considering
\be\label{E:u20-small-6(Q<)}
| \lambda'_R |\le  \min\{1,r_\ast\},
\ee Together with $\theta_{2r_\ast, 0}(\xi'_{1,2}  )=\theta_{2r_\ast, 0}(\frac{\lambda'_I}{\pi}  )$, implies   
\be\label{E:u20-small-7(Q<)}
 |\partial_{\lambda'_I}S_0 |=|+\frac 1\pi(-a+3(\lambda'^2_R-\lambda'^2_I))|\ge  1/C.
\ee

Moreover, using \eqref{E:+0-small-15}, we decompose
\begin{align}
 I_{Q_{in,<,-}} \equiv& I_{Q^+_{in,<,-}}  +I_{Q^-_{in,<,-}} ,  \label{E:u20-small-8(Q<)}\\
I_{Q^\pm_{in,<,-}}= &  -\dfrac 1{\pi i} \iint d\overline\lambda'\wedge d\lambda' e^{2\pi i tS_0 } \widetilde {  s}_c(\lambda'  )      (\overline\lambda'-\lambda')\theta_{2r_\ast, 0}(\xi'_{1,2}  )[1-\theta_{t^{-\kappa},0}(\lambda'_I)]\nonumber\\
\times&  \frac{-1}{2\pi i}  \iint_{\Sigma_{<}}  dx_1'dx_2' \mathbf v(x'_1,x'_2;t_2 ) \iint  d{\xi'_1} d{\xi'_2}\frac{e^{ 2\pi i t\wp  } }{(1+\xi'^4_1)(\xi'_2 -z_0)}   \nonumber\\
\times&\theta_{r_\ast,0}(\xi_1')\theta(\pm|\pi\xi'_{1 }-\lambda'_I|\mp \frac{|\lambda'_I|}{8}) \theta_{|\xi'_1|^{1/2}\rho_1,0}(\lambda'_R) \theta(1-|\omega|) 
.\nonumber
\end{align}

\begin{itemize}
\item[($\textrm{\bf A}_{21} $)]  {\bf Estimates for $I_{Q^+_{in,<,-}}$:}  Integration by parts with respect to $\lambda'_{I}$, using the change of variables \eqref{E:large-26},   \eqref{E:+0-small-19}, 
\begin{gather}
  |\partial_{\lambda'_{ I}}\frac 1{{\xi'_{2 }} -(2\pi i{\xi'_{1 }}^2+2{\xi'_{1 }}\overline\lambda' ) }| 
\le   C| \frac {\xi'_{1 }}{   ({\xi'_{2 }} -(2\pi i{\xi'_{1 }}^2+2{\xi'_{1 }}\overline\lambda' ))^2   }|, \label{E:u20-small-9(Q<)} \\
|{\xi'_{2 }} -(2\pi i{\xi'_{1 }}^2+2{\xi'_{1 }}\overline\lambda' ) |\ge \frac 1C|\xi'_1|t^{ -\kappa },\label{E:u20-small-10(Q<)}
\end{gather}
Lemma \ref{L:sd-der},     and the factor $\theta_{\rho_1,0}(\lambda'_R)$, we conclude
 \be \label{E:u20-small-11(Q<)}
I_{Q^+_{in,<,-}}  
=  \epsilon_\flat^2  \mathcal O(t^{-\frac 12+\epsilon_1}t^{   -1 +(-\frac 14+  \epsilon_2 ) +\kappa} ) = \epsilon_\flat^2  \mathcal O(t^{-\frac {7}{4}+ \kappa+\frac32\epsilon_2 }) .  
\ee

\item[($\textrm{\bf A}_{22} $)]   {\bf Estimates for $I_{Q^-_{in,<,-}}$:}    Applying Lemma \ref{L:sd-der}, \eqref{E:q-<0}, we obtain
\be \label{E:u20-small-12(Q<)}
I_{Q^-_{in,<,-}}  
= \epsilon_\flat^2  \mathcal O(t^{-1+\frac\kappa2+ 3\epsilon_2 }) .  
\ee
 
\end{itemize}

\end{itemize}

Combining  \eqref{E:lambda-L1}, \eqref{E:u20-0-small-3}, \eqref{E:u20-small-5(Q<)}, \eqref{E:u20-small-8(Q<)}, \eqref{E:u20-small-11(Q<)}, and \eqref{E:u20-small-12(Q<)}, taking
\be\label{E:u20-small-13(Q<)} 
\kappa=\frac 1{10},
\ee we prove the case of $n=0$ for $Q_{in, <,-}$. Namely, 
\be\label{E:u20-small-14(Q<)}
 -\dfrac 1{\pi i} \iint d\overline\lambda'\wedge d\lambda' e^{2\pi i tS_0 } \widetilde {  s}_c(\lambda'  )      (\overline\lambda'-\lambda')\theta_{2r_\ast, 0}(\xi'_{1,2}  )Q_{in,<,-}= \epsilon^2_\flat\mathcal O(t^{-\frac{19}{20}+3\epsilon_2}).
\ee

\noindent  {$\blacktriangleright$}  $\underline{\textbf{ Step 2. Estimates for \eqref{E:+0-small-5} for $Q_{in,>,-}$ when $n=0$ :}}$ 
 Using \eqref{E:q->0}, the factor $\overline\lambda'-\lambda'$,
\begin{align}
& -\dfrac 1{\pi i} \iint d\overline\lambda'\wedge d\lambda' e^{2\pi i tS_0 } \widetilde {  s}_c(\lambda'  )      (\overline\lambda'-\lambda')\theta_{2r_\ast, 0}(\lambda'_I  ) Q_{in,>,-}   \label{E:u20-0-small-4} \\
= &-\dfrac 1{\pi i} \iint d\overline\lambda'\wedge d\lambda' e^{2\pi i tS_0 } \widetilde {  s}_c(\lambda'  )      (\overline\lambda'-\lambda') \theta_{\rho_1 ,0}(\lambda'_R) Q_{in,>,-}    \nonumber\\
+& \dfrac {-1}{\pi i} \iint d\overline\lambda'\wedge d\lambda' e^{2\pi i tS_0 } \widetilde {  s}_c(\lambda'  )      (\overline\lambda'-\lambda')[1-\theta_{ \rho_1,0}(\lambda'_R) ]   \theta_{\sqrt{\rho_1},0} (\lambda'_I)    Q_{in,>,-}   \nonumber\\
+& \dfrac {-1}{\pi i} \iint d\overline\lambda'\wedge d\lambda' e^{2\pi i tS_0 } \widetilde {  s}_c(\lambda'  )      (\overline\lambda'-\lambda') \theta_{2r_\ast, 0}(\lambda'_I  )[1-\theta_{ \rho_1,0}(\lambda'_R) ] [1-\theta_{\sqrt{\rho_1},0} (\lambda'_I)]  Q_{in,>,-}   \nonumber\\ 
\equiv&\epsilon_\flat^2 \mathcal O(t^{-\frac12+   \epsilon_1-\frac23+\epsilon_1 }) + \epsilon_\flat^2 \mathcal O(t^{-\frac12+   \epsilon_1-\frac23+\epsilon_1 }) +I_{Q_{in,>,-}}.\nonumber 
\end{align}

To study $I_{Q_{in,>,-}}$, we consider:
\begin{itemize}
\item [($\textrm{\bf B}_1 $)] Adapting the argument used for ($\textrm{\bf A}_1 $) for $Q_{in,<,-}$,   if 
\be\label{E:u20-small-1(Q>)}
| \lambda'_R |\ge  \min\{1,r_\ast\},
\ee  
then
\be\label{E:u20-small-2(Q<)->}
 |\partial_{\lambda'_R}S_0 |=|+\frac6\pi\lambda'_R\lambda'_I|\sim |\lambda'_I|.
\ee

As a result, using the change of variables \eqref{E:large-26},  \eqref{E:+0-small-8} (the definition of $Q_{in, >,-}$),   
\begin{gather}
  |\partial_{\lambda'_{ R}}\frac 1{{\xi'_{2 }} -(2\pi i{\xi'_{1 }}^2+2{\xi'_{1 }}\overline\lambda' ) }| 
\le    C| \frac {   \xi' _1  }{({\xi'_{2 }} -(2\pi i{\xi'_{1 }}^2+2{\xi'_{1 }}\overline\lambda' ))^{2} }|, \label{E:u20-small-3(Q<)->}\\
|{{\xi'_{2 }} -(2\pi i{\xi'_{1 }}^2+2{\xi'_{1 }}\overline\lambda' ) }|
\ge \frac 1C  |\xi'_1|  , \quad
|\eta_2+4\pi\xi'_{1 }\lambda'_R|=2\pi|\xi'_{2 }| \le C|\xi'_1|^{1/2}\rho_1  , \label{E:u20-small-4(Q<)->}
\end{gather}integration by parts with respect to $\lambda'_R$, Lemma \ref{L:sd-der},  the factor $\overline\lambda'-\lambda'$, \eqref{E:+0-small-12}, and  \eqref{E:+0-small-12-(1)}, we obtain
\be\label{E:u20-small-2(Q>)}
I_{Q_{in,>,-}} =  \epsilon_\flat^2\mathcal O(t^{-\frac 54+\epsilon_1 }).  
\ee
 
\item [($\textrm{\bf B}_2 $)]  From ($\textrm{\bf B}_1 $), it reduces to considering
\be\label{E:u20-small-3(Q>)}
| \lambda'_R |\le  \min\{1,r_\ast\},
\ee which implies     
\be\label{E:u20-small-4(Q>)}
|\partial_{\lambda'_I}S_0 |=|+\frac 1\pi(-a+3(\lambda'^2_R-\lambda'^2_I))|\ge  \frac 1C.
\ee  

Consequently,  integration by part with respect to $\lambda'_I$, applying \eqref{E:q->0-more}, we obtain
 \be \label{E:u20-0-small-10}
I_{Q _{in,>,-}}  
=  \epsilon_\flat^2  \mathcal O(t^{  -1+    \epsilon_1}) .  
\ee

\end{itemize}

Combining  \eqref{E:lambda-L1}, \eqref{E:u20-0-small-4}, \eqref{E:u20-small-2(Q>)},   and \eqref{E:u20-0-small-10}, we  establish    the case of $n=0$ for $Q_{in,>,-}$:
\be\label{E:u20-small-9(Q>)} 
 -\dfrac 1{\pi i} \iint d\overline\lambda'\wedge d\lambda' e^{2\pi i tS_0 } \widetilde {  s}_c(\lambda'  )      (\overline\lambda'-\lambda')\theta_{2r_\ast, 0}(\lambda'_I  ) Q_{in,>,-}  = \epsilon_\flat^2\mathcal O(t^{ -1 +\epsilon_1}).
\ee

\noindent  {$\blacktriangleright$}  $\underline{\textbf{ Step 3. Estimates of \eqref{E:+0-small-5}  for $Q_{in,>,-}$ when  $n>0$ : }}$ 
We will first discuss the case when  $n=1$. Notice that   $\widetilde{(\mathcal C  T)   }\theta_{2r_\ast, 0}( \xi'_{1,2})\circ Q_{in,>,-}(t_1,t_2,t;\overline\lambda'_1)$  is a CIO near $\xi'_{1,2}=0$, and
\be \label{E:u20-0-small-14}
\begin{split}
&\textit{$Q_{in,>,-} $ is holomorphic in $\lambda'_{R,1}\lambda'_{I,1}$ when $ \lambda'_{I,1}\lambda'_{R,1}\ne 0$,}\\
&|Q_{in,>,-}(x,\lambda'_1)|=|Q_{in,>,-}(x,\frac{\xi'_{2,2}e^{\alpha}}{2\xi'_{1,2}}-i\pi\xi'_{1,2})|\\
=&|Q_{in,>,-}(x,\lambda'_{R,1}\cos\alpha+i( \lambda'_{R,1}\sin\alpha - \lambda'_{I,1})|\le \epsilon_\flat\mathcal O(t^{-\frac 23+\epsilon_1}),\ |\alpha|\le  \frac{\pi}{4}.
\end{split}
\ee  Hence we can   adapt  argument in  $ \textbf{Step 2}$  in Lemma \ref{L:CIO-large} and $ \textbf{Step 1}$-$ \textbf{Step 2}$  in Lemma \ref{L:0-small}  to obtain: 
\begin{align}
& |\widetilde{(\mathcal C_{\lambda'} T)   }\theta_{2r_\ast, 0}(\xi'_{1,2})\circ Q_{in,>,-}(t_1,t_2,t;\overline\lambda'_1)|
\label{E:u20-0-small-15}\\
\le & |\mathfrak C \mathfrak T_{(2)}  [1- \theta_{|\xi'_{1,2}|^{1/2}\rho_1,0}(\lambda'_{R})]\theta_{2r_\ast, 0}( \xi'_{1,2})   Q_{in,>,-,(1)} |\nonumber\\
+&|Q_{in,<,-,(2)} \theta_{2r_\ast, 0}(\xi'_{1,2})  Q_{in,>,-,(1)} |\nonumber\\
+&|Q_{in,>,-,(2)}\theta_{2r_\ast, 0}(\xi'_{1,2})  Q_{in,>,-,(1)} | +\epsilon_\flat\epsilon_\sharp o(t^{-1}) \nonumber\\
\equiv& I^{(11)}_{Q_{in,>,-}}+I^{(12)}_{Q_{in,>,-}}+I^{(13)}_{Q_{in,>,-}}+\epsilon_\flat\epsilon_\sharp \mathcal O_{\epsilon_1}(t^{-1}) .\nonumber
\end{align}Here we have use the notation $Q_{in,\gtrless,-,(j)}\equiv Q_{in,\gtrless,-}(t_1,t_2,t; \lambda'_j)$ for simplicity.
 
Applying  \eqref{E:q->0}, \eqref{E:q->0-add}, and argument used in  {\bf Step 1, 2}, we can derive
\begin{multline}\label{E:u20-0-small-17-1}
\iint d\overline\lambda'\wedge d\lambda' e^{2\pi i tS_0 } \widetilde {  s}_c(\lambda'  )      (\overline\lambda'-\lambda')I^{(12)}_{Q_{in,>,-}},\\
\iint d\overline\lambda'\wedge d\lambda' e^{2\pi i tS_0 } \widetilde {  s}_c(\lambda'  )      (\overline\lambda'-\lambda')I^{(13)}_{Q_{in,>,-}}= \epsilon_\flat^3o(t^{-1}).\hskip1in
\end{multline}

On the other hand, 
 \begin{itemize}
\item[($\textrm{\bf C}_1 $)] If 
\be \label{E:u20-0-small-17}
| \lambda'_R |\ge  \min\{1,r_\ast\},
\ee adapting argument in the proof of Lemma \ref{L:cio-0} (with its parameter $\epsilon$ replaced by $\epsilon_1$),  
\begin{align}
 I^{(11)}_{Q_{in,>,-}}
\le & |\mathfrak C\mathfrak T _{ (2)} [1-\theta( ||\xi_2''|-\frac{|\lambda_I'|}{2\pi} |-t^{ -\frac12+\epsilon_1})]  [1- \theta_{|\xi'_{1,2}|^{1/2}\rho_1,0}(\lambda'_{R})]\theta_{2r_\ast, 0}( \xi'_{1,2})\nonumber\\
\times&   Q_{in,>,-,(1)}(x,\lambda'_{R }+2\pi i\xi''_2)|\nonumber\\
 +&|\mathfrak C\mathfrak T_{ (2)}  \theta( ||\xi_2''|-\frac{|\lambda_I'|}{2\pi} |-t^{-\frac12+\epsilon_1})  [1- \theta_{|\xi'_{1,2}|^{1/2}\rho_1,0}(\lambda'_{R})]\theta_{2r_\ast, 0}( \xi'_{1,2})  \nonumber\\
\times& Q_{in,>,-}(x,\lambda'_{R }+2\pi i\xi''_2)| \equiv  I_2+I_3.\nonumber 
\end{align}

Applying \eqref{E:q->0}, \eqref{E:u20-0-small-14}, 
\be\label{E:u20-0-small-19}
\begin{split}
I_2=  \epsilon_\flat^2  o(t^{-1 } ).
\end{split}
\ee  

Using \eqref{E:u20-0-small-17}, the exponent computation 
\[ 
    \theta(t|\lambda'_R|-|x_{2,2}'|) \theta( ||\xi_{2}''|-\frac{|\lambda_I'|}{2\pi} |-t^{-\frac12+\epsilon_1})\  
    |(  x_{2,2}'+3  t\lambda'_R)(\xi_{2}''-\frac{\lambda'_I}{2\pi})  (\xi_{2}''+\frac{\lambda'_I}{2\pi}  )|\ge Ct^{2\epsilon_1}.   
\]Hence the amplitude function $|\mathcal F^{(2)}|\le Ce^{-t^{2\epsilon_1}}$ which implies
\be\label{E:u20-0-small-19-new}
I_3= \epsilon_\flat^2o (t^{-1} ) .
\ee

As a result,
\be\label{E:u20-0-small-20}
I^{(11)}_{Q_{in,>,-}}= \epsilon_\flat^2o (t^{-1 } ).
\ee

\item  [($\textrm{\bf C}_2 $)]  If 
\be \label{E:u20-0-small-21}
| \lambda'_R |\le  \min\{1,r_\ast\},
\ee   then
\be\label{E:u20-0-small-22}
|\partial_{\xi''_2}\mathfrak S | \ge  \frac 1C.
\ee
 
Integration by parts with respect to $\xi''_2$, applying Lemma \ref{L:dCIO}, \eqref{E:q->0-more}, and \eqref{E:u20-0-small-21},  
\be\label{E:u20-0-small-24}
 I^{(11)}_{Q_{in,>,-}}
= \epsilon_\flat\epsilon_\sharp \mathcal O (t^{-1+\epsilon_1} ).
\ee

\end{itemize}

Combining \eqref{E:u20-0-small-15}, \eqref{E:u20-0-small-17-1}, \eqref{E:u20-0-small-20}, and \eqref{E:u20-0-small-24}, we have established 
\begin{align}
&  \iint d\overline\lambda'\wedge d\lambda' e^{2\pi i tS_0 } \widetilde {  s}_c(\lambda'  )      (\overline\lambda'-\lambda')\widetilde{(\mathcal CT)   }\theta_{2r_\ast, 0}(\xi'_{1,2}  )\circ Q_{in,>,-} =\epsilon_\flat^2\epsilon_\sharp \mathcal O (t^{-1+\epsilon_1} ).  \label{E:+0-small-25}
\end{align}
  
For $n\ge 2$, applying  \eqref{E:q->0}, \eqref{E:+0-small-12}, \eqref{E:+0-small-12-(1)},   \eqref{E:+0-small-19}, and Lemma \ref{L:wickerhauser}, we obtain
\begin{align}
&  \iint d\overline\lambda'\wedge d\lambda' e^{2\pi i tS_0 } \widetilde {  s}_c(\lambda'  )      (\overline\lambda'-\lambda')\widetilde{(\mathcal CT)  ^n }\theta_{2r_\ast, 0}(\xi'_{1,2}  )\circ Q_{in,>,-} =\epsilon_\flat^{n+1}\epsilon_\sharp \mathcal O_{\epsilon_1} (t^{-1} ).  \label{E:+0-small-26}
\end{align} 

\noindent  {$\blacktriangleright$}  $ \underline{\textbf{ Step 4. Estimates of \eqref{E:+0-small-5}  for $Q_{in,<,-}$ when  $n>0$ : }}$  If $n=1$, applying \eqref{E:+0-small-19} and Lemma \ref{L:0-small},  
\begin{align}
&|  \widetilde{(\mathcal CT)   }\theta_{2r_\ast, 0}(\xi'_{1,2}  )\circ Q_{in,<,-}(t_1,t_2,t;\overline\lambda'_1)|\label{E:u20-0-small-1-(Q<n)}\\
\le & |Q_{in,<,+,(2)} \theta_{2r_\ast, 0}(\xi'_{1,2})  Q_{in,<,-,(1)} |\nonumber\\
+& |Q_{in,>,+,(2)} \theta_{2r_\ast, 0}(\xi'_{1,2})  Q_{in,<,-,(1)} |\nonumber\\
+&|Q_{in,<,-,(2)} \theta_{2r_\ast, 0}(\xi'_{1,2})  Q_{in,<,-,(1)}|\nonumber\\
+&|Q_{in,>,-,(2)}\theta_{2r_\ast, 0}(\xi'_{1,2})  Q_{in,<,-,(1)} | +\epsilon_\flat\epsilon_\sharp o (t^{-1}) \nonumber\\
\equiv& I^{(11)}_{Q_{in,<,-}}+I^{(12)}_{Q_{in,<,-}}+I^{(13)}_{Q_{in,<>,-}}+I^{(14)}_{Q_{in,<,-}}+\epsilon_\flat\epsilon_\sharp o (t^{-1}) .\nonumber
\end{align}Here we have use the notation $Q_{in,\gtrless,\pm,(j)}\equiv Q_{in,\gtrless,\pm}(t_1,t_2,t; \lambda'_j)$ for simplicity.
 
Applying  \eqref{E:q-<0}, 
 and argument used in  {\bf Step 1, 2}, we can derive
\begin{multline}\label{E:u20-0-small-13}
\iint d\overline\lambda'\wedge d\lambda' e^{2\pi i tS_0 } \widetilde {  s}_c(\lambda'  )      (\overline\lambda'-\lambda')I^{(13)}_{Q_{in,<,-}},\\
\iint d\overline\lambda'\wedge d\lambda' e^{2\pi i tS_0 } \widetilde {  s}_c(\lambda'  )      (\overline\lambda'-\lambda')I^{(14)}_{Q_{in,>,-}}= \epsilon_\flat^3o(t^{-1}).\hskip1in
\end{multline}

Moreover, for $I^{(12)}_{Q_{in,<,-}}$, from $Q_{in,>,+,(2)}$,  
\be\label{E:u20-0-small-26}
 |\xi'_{2,2}|=|2\xi'_{1,2}\lambda'_{R,1}|>\frac 14|\xi'_{1,2}|^{1/2}\rho_1;  
\ee from $Q_{in,<,-,(1)}$,    
\be\label{E:u20-0-small-27}
|\lambda'_{R,1}|<\rho_1.
\ee 

Combining \eqref{E:u20-0-small-26}, \eqref{E:u20-0-small-27}, the factor $ \theta_{2r_\ast, 0}(\xi'_{1,2}) $, and \eqref{E:wp-4}, we derive 
\be\label{E:u20-0-small-28}
I^{(12)}_{Q_{in,<,-}}\equiv 0.
\ee

In an entirely similar way, we can justify
\be\label{E:u20-0-small-29}
I^{(11)}_{Q_{in,<,-}}\equiv 0.
\ee

Combining \eqref{E:u20-0-small-1-(Q<n)}, \eqref{E:u20-0-small-13}, \eqref{E:u20-0-small-28},  and \eqref{E:u20-0-small-29}, we have established 
\begin{align}
&  \iint d\overline\lambda'\wedge d\lambda' e^{2\pi i tS_0 } \widetilde {  s}_c(\lambda'  )      (\overline\lambda'-\lambda')\widetilde{(\mathcal CT)   }\theta_{2r_\ast, 0}(\xi'_{1,2}  )\circ Q_{in,<,-} =\epsilon_\flat^2\epsilon_\sharp o (t^{-1} ).  \label{E:+0-small-30}
\end{align}
  
For $n=2, 3$, bootstrapping the above argument, we obtain
\begin{align}
&  \iint d\overline\lambda'\wedge d\lambda' e^{2\pi i tS_0 } \widetilde {  s}_c(\lambda'  )      (\overline\lambda'-\lambda')\widetilde{(\mathcal CT)  ^n }\theta_{2r_\ast, 0}(\xi'_{1,2}  )\circ Q_{in,>,-} =\epsilon_\flat^{n+1}\epsilon_\sharp o (t^{-1} ),  \label{E:+0-small-31}
\end{align}
Together with Lemma \ref{L:wickerhauser}, we establish \eqref{E:+0-small-31} for $n\ge 4$.

\end{proof}

Finally, we conclude the Subsection \ref{SS:CI-0} by the reduced representation formula of $u_{2,0}$:
\begin{theorem}[A reduced representation formula of $u_{2,0}$]\label{T:first-reduction} Suppose that $|a|>\delta>0$ and \eqref{intro-ID} holds.
Then, for every sufficiently small $\epsilon>0$,
as $t\to\infty$,
\be 
    u_{2,0}(x) 
 =   -\frac{1}{\pi i}   \sum_{n=1}^\infty   \iint d\overline\lambda'\wedge d\lambda'  \widetilde {  s}_c(\lambda'  )e^{\beta_{n+1}2\pi i tS_0 }      (\overline\lambda'-\lambda')    \left[ \mathfrak C \mathfrak T 1 \right]^{0,(n)} +\epsilon _\flat\epsilon_\sharp o_{\epsilon } ( t^{-\frac{19}{20}+ \epsilon }). \label{E:first-reduction-1}\ee
 
\end{theorem}

\begin{proof} Applying Lemma \ref{L:cio-0} (with its parameter $\epsilon$ replaced by $\epsilon_1$)   and \eqref{E:+0-small-19}, we have 
\begin{align}
&\iint d\overline\lambda'\wedge d\lambda'  \widetilde {  s}_c(\lambda'  )e^{2\pi i tS_0 }      (\overline\lambda'-\lambda') Q_{in,<,-} e^{-i\pi tS_0}\left[ \mathfrak C \mathfrak T 1 \right]^{0,(1)}= \epsilon_\flat^3  \mathcal O(t^{-\frac 12 +\epsilon_1-\frac 14+   \epsilon_2-\frac 12+   \epsilon_1}). \label{E:first-reduction-2} 
\end{align}

Applying Lemma \ref{L:cio-0} (with its parameter $\epsilon$ replaced by $\epsilon_1$), \ref{L:lambda-r}, \eqref{E:+0-small-12},   and   \eqref{E:q->0-more},
\begin{align}
& \iint d\overline\lambda'\wedge d\lambda'  \widetilde {  s}_c(\lambda'  )e^{ 2\pi i tS_0 }      (\overline\lambda'-\lambda') Q_{in,>,-} e^{-i\pi tS_0}\left[ \mathfrak C \mathfrak T 1 \right]^{0,(1)}\label{E:first-reduction-3}\\ 
= &   \iint d\overline\lambda'\wedge d\lambda'  \widetilde {  s}_c(\lambda'  )e^{ 2\pi i tS_0 }      (\overline\lambda'-\lambda')[1-\theta_{2r_\ast,0}(\frac{\lambda'_I}{\pi})] Q_{in,>,-} e^{-i\pi tS_0}\left[ \mathfrak C \mathfrak T 1 \right]^{0,(1)}\nonumber\\
+&\iint d\overline\lambda'\wedge d\lambda'  \widetilde {  s}_c(\lambda'  )e^{ 2\pi i tS_0 }      (\overline\lambda'-\lambda')\theta_{2r_\ast,0}(\frac{\lambda'_I}{\pi})\theta_{\rho_1,0}(\lambda'_R)Q_{in,>,-} e^{-i\pi tS_0}\left[ \mathfrak C \mathfrak T 1 \right]^{0,(1)}\nonumber\\
+& \iint d\overline\lambda'\wedge d\lambda'  \widetilde {  s}_c(\lambda'  )e^{ 2\pi i tS_0 }      (\overline\lambda'-\lambda')\theta_{2r_\ast,0}(\frac{\lambda'_I}{\pi})[1-\theta_{\rho_1,0}(\lambda'_R)]\theta_{\sqrt{\rho_1},0}(\lambda'_I)Q_{in,>,-} \nonumber\\
\times&e^{-i\pi tS_0}\left[ \mathfrak C \mathfrak T 1 \right]^{0,(1)}\nonumber\\
+& \iint d\overline\lambda'\wedge d\lambda'  \widetilde {  s}_c(\lambda'  )e^{ 2\pi i tS_0 }      (\overline\lambda'-\lambda')\theta_{2r_\ast,0}(\frac{\lambda'_I}{\pi})[1-\theta_{\rho_1,0}(\lambda'_R)][1-\theta_{\sqrt{\rho_1},0}(\lambda'_I)]Q_{in,>,-}\nonumber\\
\times& e^{-i\pi tS_0}\left[ \mathfrak C \mathfrak T 1 \right]^{0,(1)} \nonumber\\
=&  \epsilon_\flat^3 \mathcal O(t^{-\frac 12+   \epsilon_1  -\frac 12+   \epsilon_1}+t^{-\frac 12+   \epsilon_1 -\frac 14+   \epsilon_1-\frac 12+   \epsilon_1}+t^{  -\frac 12+    \epsilon_1  -\frac 14+   \epsilon_1  -\frac 12+   \epsilon_1  }+t^{  -1  +   \epsilon_1  }).  \nonumber
\end{align} Here we adapt the argument used to estimate $I_{Q_{in,>,-}}$ in {\bf Step 2  of   Lemma \ref{L:u20-0-small}} to estimate the term localized by $\theta_{2r_\ast,0}(\frac{\lambda'_I}{\pi})[1-\theta_{\rho_1,0}(\lambda'_R)][1-\theta_{\sqrt{\rho_1},0}(\lambda'_I)]$.

Consequently,   applying  Lemmas \ref{L:cio-0}, \ref{L:u20-large}, \ref{L:u20-0-small},   \ref{L:u20-0-small}, \eqref{E:first-reduction-2}, and \eqref{E:first-reduction-3}, we conclude
\begin{multline}\label{E:first-reduction-4}
   -\dfrac 1{\pi i} \iint d\overline\lambda'\wedge d\lambda' e^{2\pi i tS_0 } \widetilde {  s}_c(\lambda'  )      (\overline\lambda'-\lambda')\widetilde{(\mathcal CT)^2  }1\\ =  -\frac{1}{\pi i}       \iint d\overline\lambda'\wedge d\lambda'  \widetilde {  s}_c(\lambda'  )e^{\beta_{3}2\pi i tS_0 }      (\overline\lambda'-\lambda')    \left[ \mathfrak C \mathfrak T 1 \right]^{0,(2)}+2\epsilon_\flat^2\epsilon_\sharp   \mathcal O_{\epsilon_1}( t^{-1+ 3\epsilon_2}+t^{-\frac{19}{20}+3\epsilon_2} ).  
\end{multline}

Bootstrapping the above argument, 
 we obtain
\begin{align}
     u_{2,0}(x) 
 = &    \frac{-1}{\pi i}   \sum_{n=1}^\infty  \left( \iint d\overline\lambda'\wedge d\lambda'  \widetilde {  s}_c(\lambda'  )e^{\beta_{n+1}2\pi i tS_0 }      (\overline\lambda'-\lambda') \left[ \mathfrak C \mathfrak T 1 \right]^{0,(n)}  \right.\label{E:first-reduction-5}\\  
  +&\left.n\epsilon^{n+1} _\flat\epsilon_\sharp \mathcal O_\epsilon(  t^{-1+ 3\epsilon_2}+t^{-\frac{19}{20}+3\epsilon_2}) \right)    \nonumber\\
=&\frac{-1}{\pi i}   \sum_{n=1}^\infty   \iint d\overline\lambda'\wedge d\lambda'  \widetilde {  s}_c(\lambda'  )e^{\beta_{n+1}2\pi i tS_0 }      (\overline\lambda'-\lambda') \left[ \mathfrak C \mathfrak T 1 \right]^{0,(n)}\nonumber\\
 +&\epsilon  _\flat\epsilon_\sharp \mathcal O_\epsilon(  t^{-1+ 3\epsilon_2}+t^{-\frac{19}{20}+3\epsilon_2}) .\nonumber  
\end{align}  
Setting $4\epsilon_2=\epsilon $, we complete the proof.
\end{proof}

\subsection{Long time asymptotics of $u_{2,0}(x) $ when    $ a >+\delta>0$ }\label{SS:u+2}

Throughout this subsection, we assume that $ a >+\delta>0$, and  \eqref{intro-ID} holds.   We define $
\psi_{r,w_0}$ and $ \chi$    as in  \eqref{E:localize-fct} and   \eqref{E:local}, respectively, take  $\epsilon_\flat$  and  $\epsilon_\sharp$  as in \eqref{E:large-15},  and set  $  b=  ( -r^2+{\lambda'}_R^2)^{1/2}/{2\pi}$.

Using Lemma \ref{L:dCIO}, we adapt the argument of Proposition \ref{P:u-12}, which combines $L^1$ estimates with the Riemann--Lebesgue lemma, to establish the following asymptotic estimates away from the stationary points.
\begin{lemma}\label{L:+3ndReduction-outside}
Suppose that $ a >+\delta>0$. As $t\to\infty$,   
\begin{align}
&    \sum_{n=1}^\infty |\iint d\overline\lambda'\wedge d\lambda'  \ e^{\beta_{n+1}2\pi i tS_0 }     \widetilde {  s}_c(\lambda'  ) (\overline\lambda'-\lambda'){\color{black}[1-\chi(\lambda')]}\left[ \mathfrak C \mathfrak T 1 \right]^{0,(n)}| =\epsilon _\flat \epsilon_\sharp  o(t^{-1}). \label{E:first-reduction-1-+new-1-outside} 
\end{align}   
\end{lemma}
\begin{proof} The proof of the lemma demonstrates that the term $(\overline{\lambda'} - \lambda')$ is essential for eliminating the Kiselev conditions, such as the integrability of $(1+|\lambda'|)  s_c$  or boundedness of $\partial^k_{\lambda'_I}s_c$ for $k\le 2$.

\begin{itemize}
\item [$\blacktriangleright$]{\it Case of $a-3{\lambda'_R}^2>0$}: In this case,   $|{\lambda_R'}|\le r $, 
and  the analysis can be reduced to cases:  
\begin{itemize}
\item [(1)]  $\psi_{r,r}(\lambda'_R)\ne 0$ and $ \psi_{  r,0}(\lambda'_I)=0$; 
\item [(2)] $ \psi_{r,r}(\lambda'_R)=0$. 
\end{itemize}
Notice that  
\be\label{E:lambda-plus}
|\partial_{\lambda'_I}S_0(a;\lambda_R';\lambda'_I)|= \frac 1\pi|3 {\lambda'_I}^2+(a-3{\lambda'}_R^2  )|\ge r/C
\ee for both cases. Therefore, we obtain \eqref{E:first-reduction-1-+new-1-outside} by applying integration by parts  with respect to $\lambda'_I$ to 
\[
 {\color{black}\iint d\overline\lambda'\wedge d\lambda'  \ e^{\beta_{n+1}2\pi i tS_0 }     \widetilde {  s}_c(\lambda'  ) (\overline\lambda'-\lambda')[1-\chi(\lambda')]\theta( a-3{\lambda'_R}^2 )}   
  \left[ \mathfrak C \mathfrak T 1 \right]^{0,(n)},\]
\begin{itemize}
\item using  Lemmas \ref{L:sd-der}, \ref{L:growth}, \eqref{E:intro-s-c-ana-c-0}, \eqref{E:lambda-plus}, and $|{\lambda_R'}|\le r $   to obtain
\be\label{E:integration-by-parts-o-42}
|\partial_{\lambda'_I}  \left(\widetilde{  s}_c(\lambda'  )(\overline\lambda'-\lambda') \frac{1}{|\partial_{\lambda'_I}   S_0| }\right)|_{L^1(d\lambda'_R d\lambda'_I)}<C {\color{black} \epsilon_\flat} ;
\ee
\item adopting \eqref{E:intro-s-c-ana-c-0}, \eqref{E:n-lambda-i}, \eqref{E:lambda-plus}, Lemma \ref{L:growth}, and $|{\lambda_R'}|\le r $ to get
\be\label{E:integration-by-parts-o-42-new}
|\widetilde{  s}_c(\lambda'  )(\overline\lambda'-\lambda') \frac{1}{|\partial_{\lambda'_I}   S_0| }\partial_{\lambda'_I}   \left[ \mathfrak C \mathfrak T 1 \right]^{0,(n)}  |_{L^1(d\lambda'_R d\lambda'_I)}<C {\color{black}\epsilon_\flat^{n }\epsilon_\sharp}.
\ee
\end{itemize}  

\item [$\blacktriangleright$]{\it Proof of $a-3{\lambda'_R}^2<0$}: In this case, $|{\lambda_R'}|\ge r $. The analysis can be reduced to cases:  
\begin{itemize}
\item [(1')]  $\psi_{r,r}(\lambda'_R)\ne 0$ and $ \psi_{  r,0}(\lambda'_I)=0$; 
\item [(2')] $ \psi_{r,r}(\lambda'_R)=0$. 
\end{itemize}

Proof of Case of (1') can be proceeded as that of Case (1). 
For Case of (2'), \eqref{E:first-reduction-1-+new-1-outside} is justified by  
adapting argument of Proposition \ref{P:u-12}, that is, integration by parts  with respect to $\lambda'$ to
\[
 \iint d\overline\lambda'\wedge d\lambda'  \  e^{\beta_{n+1}2\pi i tS_0 } \widetilde {  s}_c(\lambda'  ) (\overline\lambda'-\lambda')[1-\chi(\lambda')]\theta( -(a-3{\lambda'_R}^2))   
  \left[ \mathfrak C \mathfrak T 1 \right]^{0,(n)}  ,
\]  
\begin{itemize}
\item Using  $ \psi_{r,r}(\lambda'_R)=0$ and  $|{\lambda_R'}|\ge r $, we derive 
\be\label{E:integration-by-parts-o-42->0}
 \nabla\cdot   {\color{black}  \theta( -(a-3{\lambda'_R}^2)) }  =0.
\ee
\item Applying \eqref{E:intro-s-c-ana-c-0},  Lemmas \ref{L:sd-der}, \ref{L:growth}, and taking advantage of the factor $(\overline\lambda'-\lambda')$, we have
\be\label{E:integration-by-parts-o-42->}
|\nabla\cdot \left(\widetilde{  s}_c(\lambda'  )(\overline\lambda'-\lambda')(1-\chi)\frac{\nabla S_0}{|\nabla S_0|^2}\right)|_{L^1(d\lambda'_R d\lambda'_I)}<C{\color{black} \epsilon_\flat}.
\ee
\item Applying \eqref{E:intro-s-c-ana-c-0},   the factor $(\overline\lambda'-\lambda')$,  Lemmas \ref{L:growth}, and \ref{L:lambda-r} for $|{\lambda_R'}|\ge r $,
\be\label{E:integration-by-parts-o-42->1}
|\widetilde{  s}_c(\lambda'  )(\overline\lambda'-\lambda')(1-\chi)\frac{\nabla S_0}{|\nabla S_0|^2}\nabla\cdot \left[  \mathfrak C \mathfrak T 1 \right]^{0,(n)} |_{L^1(d\lambda'_R d\lambda'_I)}<C {\color{black}\epsilon_\flat^{n }\epsilon_\sharp}.
\ee
\end{itemize}   
\end{itemize}

\end{proof}

With the aid of Lemma \ref{L:lambda-r}, we now adapt the argument from Proposition \ref{P:u-11} to derive asymptotic estimates near the stationary points. However, due to the lack of effective control on higher derivatives of the ABF Cauchy integrals, we cannot use Airy function properties to obtain an $\mathcal O(t^{-\frac 43})$ estimate as in Proposition \ref{P:u-11}, and instead only derive an $o(t^{-1})$ bound.   
\begin{lemma}\label{L:+2ndReduction}
Suppose that $ a >+\delta>0$. As $t\to\infty$,   
\be   \label{E:2ndreduction-1-+new} 
 \sum_{n=1}^\infty  \iint d\overline\lambda'\wedge d\lambda'   \ e^{\beta_{n+1}2\pi i tS_0 } \widetilde {  s}_c(\lambda'  ) (\overline\lambda'-\lambda'){\color{black} \chi(\lambda')}   \left[ \mathfrak C \mathfrak T 1 \right]^{0,(n)}  =   \epsilon _\flat \epsilon_\sharp  o( t^{-1})  . \hskip.3in
\ee
\end{lemma}
\begin{proof} Applying integration by parts, using the factors $\overline\lambda'-\lambda'$, $\chi$,   Lemmas \ref{L:sd-der},  \ref{L:lambda-r},  and adapt argument as  in Proposition \ref{P:u-12} (Riemann-Lebesgue), 
\begin{multline*}
    \textit{LHS of \eqref{E:2ndreduction-1-+new}} \\=   \mathcal O( t^{-1}) \sum_{n=1}^\infty\int d \lambda'_I  \int d\lambda'_R\  e^{\beta_{n+1}2\pi i tS_0 } \partial_{\lambda_R'}\left(\frac{1}{\lambda'_R}  \widetilde {  s}_c(\lambda'  )   \chi(\lambda')  \left[ \mathfrak C \mathfrak T 1 \right]^{0,(n)}\right)  = \epsilon _\flat \epsilon_\sharp o( t^{-1}) . 
\end{multline*}

\end{proof}

\begin{theorem}\label{T:+u-2}   Suppose that $a>+\delta>0$ and \eqref{intro-ID} holds.
Then, for every sufficiently small $\epsilon>0$,
as $t\to\infty$,
\be \label{E:+u-2}
   u_{2,0 }= \epsilon _\flat\epsilon_\sharp o_{\epsilon } ( t^{-\frac{19}{20}+ \epsilon }).
\ee
\end{theorem}
\begin{proof}
Follows from Theorem  \ref{T:first-reduction}, Lemmas \ref{L:+3ndReduction-outside}, and \ref{L:+2ndReduction}.
\end{proof}

\subsection{Long time asymptotics of $u_{2,0}(x) $ when    $ a <-\delta<0$}\label{SS:u-2}

Throughout this subsection, we assume that $ a <-\delta<0$, and  \eqref{intro-ID} holds.   We define $
\psi_{r,w_0}$    as in  \eqref{E:localize-fct}, take  $\epsilon_\flat$ and  $\epsilon_\sharp$  as in \eqref{E:large-15},  and set  $  b=  ( r^2+{\lambda'}_R^2)^{1/2}/{2\pi}$.

An analogue of Lemma \ref{L:+3ndReduction-outside} is stated as follows:
\begin{lemma}\label{L:3-reduction-u20} Suppose that $ a<-\delta<0$.     As $t\to\infty$,  
\begin{multline}\label{E:-reduction-3-outside-new}
     \sum_{n=1}^\infty  \iint d\overline\lambda'\wedge d\lambda'  e^{\beta_{n+1}2\pi i tS_0 }\widetilde {  s}_c(\lambda'  )      (\overline\lambda'-\lambda') [1-\psi_{r,0}(\lambda'_R)\psi_{r,r}(\lambda'_I)]  \\
 \times  \theta(|\lambda_R'|-t^{\color{black}-\frac12-2\epsilon}) 
\left[ \mathfrak C \mathfrak T 1 \right]^{0,(n)}= \epsilon _\flat \epsilon_\sharp o(t^{-1}) . \hskip.3in
\end{multline} 
\end{lemma} 
\begin{proof}   
Without loss of generality, the analysis of \eqref{E:-reduction-3-outside-new} can be reduced to cases:  
\begin{itemize}
\item [(1'')]  $\psi_{r,0}(\lambda'_R)\ne 0$ and $ \psi_{  r,r}(\lambda'_I)=0$; 
\item [(2'')] $ \psi_{r,0}(\lambda'_R)=0$. 
\end{itemize}

For Case (1''),  $|\partial_{\lambda'_I}[{ \lambda'}^3_I+(a-3{\lambda'}_R^2)\lambda'_I]|>r/C'_0$ for some positive constant $C_0'$. Hence the proof of \eqref{E:-reduction-3-outside-new} can be established by applying integration by parts with respect to $\lambda_I'$ and $|\lambda_R'|<r$.

For Case (2''), the proof of \eqref{E:-reduction-3-outside-new} can be obtained by applying integration by parts with respect to $\lambda '$, Lemma \ref{L:lambda-r}, the factor  $(\overline\lambda'-\lambda')$,   and an adaptation of the argument in the proof of Proposition \ref{P:u-12}. 
\end{proof}  

For $a<-\delta<0$, the absence of a positive lower bound for $|\lambda_R'|$ prevents us from applying the argument used in Lemma \ref{L:+2ndReduction} near the stationary points. Moreover, without effective higher derivative estimates, we mainly rely on integration by parts, and the proof becomes more delicate.  The obstruction arises when derivatives act on the ABF Cauchy integrals: regardless of how small the integration region is, their derivatives admit at best an $\mathcal{O}(1)$ bound (see \eqref{E:-AddReduction-8-+}). Consequently, the strongest decay we obtain is $o(t^{-11/12+\epsilon})$.

\begin{theorem}\label{T:-u-2}     Suppose that $a<-\delta<0$ and \eqref{intro-ID} holds.
Then, for every sufficiently small $\epsilon>0$,
as $t\to\infty$,
\be\label{E:-u-2}
   u_{2,0 }= \epsilon _\flat \epsilon_\sharp o_\epsilon( {t}^{-\frac{11}{12}+\epsilon}  ) .
\ee 
\end{theorem} \begin{proof}   We divide the asymptotic analysis into two regimes:
\begin{align}
 &  \sum_{n=1}^\infty |\iint d\overline\lambda'\wedge d\lambda'  e^{\beta_{n+1}2\pi i tS_0 }  \widetilde {  s}_c(\lambda'  )     (\overline\lambda'-\lambda') {\color{black} \psi_{r,0}(\lambda'_R)\psi_{r,r}(\lambda'_I)}\theta(|\lambda_R'|-t^{-\frac 12-2\epsilon_\ast})  \label{E:-2ndreduction-1-+greece}\\
&\quad\times\mathfrak C \mathfrak T_{(n)} 
{\color{black} [1-\theta(||\xi_n''|- \frac{|\lambda'_I|}{2\pi}|-t^{-1+2z})]} \left[ \mathfrak C \mathfrak T 1 \right]^{0,(n-1)}| ; \nonumber  
\\
&\sum_{n=1}^\infty  |\iint d\overline\lambda'\wedge d\lambda' e^{\beta_{n+1}2\pi i tS_0 }      \widetilde {  s}_c(\lambda'  ) (\overline\lambda'-\lambda') {\color{black} \psi_{r,0}(\lambda'_R)\psi_{r,r}(\lambda'_I)} \theta(|\lambda_R'|-t^{-\frac 12-2\epsilon_\ast})     \label{E:-2ndReduction-(0)-more-greece}\\
\times&\mathfrak C\mathfrak T_{(n)}\theta(||\xi_n''|- \frac{|\lambda'_I|}{2\pi}|-t^{-1+2z}) \left[\mathfrak C\mathfrak T1\right]^{0,(n-1)} |  .  \nonumber
\end{align} where the parameters $z$ and $ \epsilon_\ast  $ will be chosen later to optimize the asymptotic estimates.

\begin{itemize}
\item [$\blacktriangleright$] 
Decompose  \eqref{E:-2ndreduction-1-+greece} into:
\begin{align}
\le&  \sum_{n=1}^\infty   |\iint d\overline\lambda'\wedge d\lambda'  \ e^{\beta_{n+1}2\pi i tS_0 }     \widetilde {  s}_c(\lambda'  ) (\overline\lambda'-\lambda'){\color{black}\psi_{r,0}(\lambda'_R)\psi_{r,r}(\lambda'_I)}\theta(|\lambda_R'|-t^{-\frac 12-2\epsilon_\ast}) \label{E:-2ndreduction-2-+new}\\
 \times &  \psi_{t^{-z},b}(\lambda_I')  \mathfrak C \mathfrak T_{(n)} {\color{black} [1-\theta(||\xi_n''|- \frac{|\lambda'_I|}{2\pi}|-t^{-1+2z})]}\left[ \mathfrak C \mathfrak T 1 \right]^{0,(n-1)}|\nonumber   \\
 + & \sum_{n=1}^\infty  |\iint d\overline\lambda'\wedge d\lambda'  \ e^{\beta_{n+1}2\pi i tS_0 }     \widetilde {  s}_c(\lambda'  ) (\overline\lambda'-\lambda'){\color{black}\psi_{r,0}(\lambda'_R)\psi_{r,r}(\lambda'_I)}\theta(|\lambda_R'|-t^{-\frac 12-2\epsilon_\ast}) \nonumber   \\
\times&  [1-\psi_{t^{-z},b}(\lambda_I')  ]\mathfrak C \mathfrak T_{(n)} {\color{black} [1-\theta(||\xi_n''|- \frac{|\lambda'_I|}{2\pi}|-t^{-1+2z})]\left[ \mathfrak C \mathfrak T 1 \right]^{0,(n-1)}}| 
\equiv  I +I' .\nonumber    
\end{align} 

From the size of the integration region,
\be\label{E:-AddReduction-9}
I  \le  C\epsilon_\flat ^2
   t^{-(1-z)}.
\ee 

Integration by parts with respect to $\lambda_I'$, using \eqref{E:lambda-L1}, Lemmas \ref{L:sd-der}, \ref{L:growth}, \eqref{E:lambda-plus}, 
   $ b=( r^2+{\lambda'}_R^2)^{1/2}/{2\pi}$, and  \eqref{E:lambda-i}, we  obtain 
\be\label{E:-AddReduction-8-+}
 I' \le  C  \epsilon _\flat \epsilon_\sharp  t^{-(1-z)} .
\ee 

\item [$\blacktriangleright$] As for \eqref{E:-2ndReduction-(0)-more-greece}, let $0<\epsilon_z\ll z$, consider the decomposition  
\begin{align}
&   {\color{black} \psi_{r,r}(\lambda'_I)}\theta(|{\lambda'}_R|-t^{-\frac 12-2\epsilon_\ast} ) \theta(||\xi_n''|- \frac{|\lambda'_I|}{2\pi}|-t^{-1+2z})  \label{E:-3ndReduction-2}\\ 
=&   {\color{black} \psi_{r,r}(\lambda'_I)} \theta(|\lambda_R'|-t^{-2z+\epsilon_z}) \theta(||\xi_n''|- \frac{|\lambda'_I|}{2\pi}|-t^{-1+2z})  \nonumber\\
+&   {\color{black} \psi_{r,r}(\lambda'_I)}\theta(|\lambda_R'|-t^{-\frac 12-2\epsilon_\ast}) \theta(t^{-2z+\epsilon_z}-|\lambda_R'|)  \theta(||\xi_n''|- \frac{|\lambda'_I|}{2\pi}|-t^{-1+2z} ).\nonumber
\end{align}

From  
\begin{multline}\label{E:-2ndReduction-6-(new)}
   \psi_{r,r}(\lambda'_I) \theta(|\lambda_R'|-t^{-2z+\epsilon_z}) \theta(||\xi_n''|- \frac{|\lambda'_I|}{2\pi}|-t^{-1+2z})  \theta(t|\lambda_R'|-|x_{2,n }'|)  \\
 \times    |(  x_{2,n }'+3  t\lambda'_R)(\xi_n''-\frac{\lambda'_I}{2\pi})  (\xi_n''+\frac{\lambda'_I}{2\pi}  )| \ge t^{ 1- 2z+\epsilon_z-1+2z}/C . 
 \end{multline}
Therefore,  the $L^1(d\xi''_n)$-norms of $\mathcal F^{(n)}$ over the corresponding domains for the first term on the right-hand side of \eqref{E:-3ndReduction-2} is less than $  C_{\epsilon_z}(t^{-2})$. 
 Along with \eqref{E:intro-s-c-ana-c-0} and the factor $(\overline\lambda'-\lambda')$,  the analysis reduces to studying the contribution over the domain corresponding to the second term, which is bounded by:
\begin{align}
&  |\iint d\overline\lambda'\wedge d\lambda'\ e^{\beta_{n+1}2\pi i tS_0 }   \widetilde {  s}_c(\lambda'  )     (\overline\lambda'-\lambda') {\color{black}\psi_{r,0}(\lambda'_R)\psi_{r,r}(\lambda'_I)}\theta(|\lambda_R'|-t^{-\frac 12-2\epsilon_\ast})  \label{E:-2ndReduction-4}\\
\times&\theta(t^{-2z+\epsilon_z}-|\lambda_R'|)  \psi_{t^{-y+\epsilon_y},b}(\lambda_I')\mathfrak C\mathfrak T_{(n)}   \theta(||\xi_n''|- \frac{|\lambda'_I|}{2\pi}|-t^{-1+2z})\left[\mathfrak C\mathfrak T1\right]^{0,(n-1)}| \nonumber\\
+&  |\iint d\overline\lambda'\wedge d\lambda' \ e^{\beta_{n+1}2\pi i tS_0 } \widetilde {  s}_c(\lambda'  )(\overline\lambda'-\lambda') {\color{black}\psi_{r,0}(\lambda'_R)\psi_{r,r}(\lambda'_I)}\theta(|\lambda_R'|-t^{-\frac 12-2\epsilon_\ast})  \nonumber\\
\times& \theta(t^{-2z+\epsilon_z}-|\lambda_R'|)(1-\psi_{t^{-y+\epsilon_y},b}(\lambda_I')) \mathfrak C\mathfrak T_{(n)}\psi_{t^{-y-\epsilon_y},b}(\xi_n'')   \nonumber\\
\times&\theta(||\xi_n''|- \frac{|\lambda'_I|}{2\pi}|-t^{-1+2z})  \left[\mathfrak C\mathfrak T1\right]^{0,(n-1)} |\nonumber\\
+&  |\iint d\overline\lambda'\wedge d\lambda'\ e^{\beta_{n+1}2\pi i tS_0 }\widetilde {  s}_c(\lambda'  )(\overline\lambda'-\lambda') {\color{black}\psi_{r,0}(\lambda'_R)\psi_{r,r}(\lambda'_I)}\theta(|\lambda_R'|-t^{-\frac 12-2\epsilon_\ast}) \nonumber\\
\times& \theta(t^{-2z+\epsilon_z}-|\lambda_R'|) (1-\psi_{t^{-y+\epsilon_y},b}(\lambda_I'))\mathfrak C\mathfrak T_{(n)}(1-\psi_{t^{-y-\epsilon_y},b}(\xi_n''))\nonumber\\
\times& \theta(||\xi_n''|- \frac{|\lambda'_I|}{2\pi}|-t^{-1+2z})\left[\mathfrak C\mathfrak T1\right]^{0,(n-1)}|  
\equiv  I_1+I_2+I_3.\nonumber
\end{align}Here $ y$ will be determined to optimize the asymptotic estimates and $0<\epsilon_y\ll y$.

Applying Lemma \ref{L:cio-0} with its parameter $\epsilon$ replaced by $\epsilon_\ast$, and using  
\[
|\psi_{t^{-y+\epsilon_y},b}(\lambda_I')|_{L^1(d\lambda_I')}\le Ct^{-y+\epsilon_y},\quad |\theta(t^{-2z+\epsilon_z}-|\lambda_R'|)|_{L^1(d\lambda_R')}\le Ct^{-2z+\epsilon_z},\] and \eqref{E:intro-s-c-ana-c-0}, we obtain
\be\label{E:-2ndReduction-5}
|I_1|\le C_{\epsilon_\ast} {\color{black}\epsilon_\flat^{n }\epsilon_\sharp} t^{-2z+\epsilon_z-y+\epsilon_y- \frac 12+\epsilon_\ast } .
\ee

Moreover,  using the two stationary points $ \pm b= \pm (  r^2+{\lambda'}_R^2)^{1/2}/{2\pi}$ of $\mathfrak S$ (due to $\color{black}a<-\delta<0$), we have
\be\label{E:b}
(1-\psi_{t^{-y+\epsilon_y},b}(\lambda_I'))\psi_{t^{-y-\epsilon_y},b}(\xi_n'')|(\xi_n'' -  \frac{\lambda'_I}{2\pi})(\xi_n''+\frac{\lambda'_I}{2\pi})|\ge \frac 1Ct^{-y},\ee
 which implies
\begin{align}
 &  (1-\psi_{t^{-y+\epsilon_y},b}(\lambda_I')) \psi_{t^{-y-\epsilon_y},b}(\xi_n'')  \theta(t|\lambda_R'|-|x_{2,n }'|)\theta(|\lambda_R'|-t^{-\frac 12-2\epsilon_\ast})\label{E:-2ndReduction-6} \\
 \times&   |(  x_{2,n }'+3  t\lambda'_R)(\xi_n''-\frac{\lambda'_I}{2\pi})  (\xi_n''+\frac{\lambda'_I}{2\pi}  )| \ge t^{ 1-\frac 12-2\epsilon_\ast-y}/C .\nonumber
 \end{align}
Provided
\be\label{E:y-asump}
 \frac 12-2\epsilon_\ast-y>0,
\ee together with \eqref{E:intro-s-c-ana-c-0}, the factor $(\overline\lambda'-\lambda')$, and \eqref{E:real-ima-new}, yields 
\be\label{E:-2ndReduction-7}
|I_2|\le \epsilon_\flat^{n+1} o_{\epsilon_\ast,y}(t^{-1}).
\ee

Finally, for $I_3$,  we apply integration by parts with respect to $\xi_n''$,    \eqref{E:n-lambda-i},  $  b=  (  r^2+{\lambda'}_R^2)^{1/2}/{2\pi}$, and  $|\theta(t^{-2z+\epsilon_z}-|\lambda_R'|)|_{L^1(d\lambda_R')}\le C t^{-2z+\epsilon_z}$ to conclude  
\be\label{E:-2ndReduction-8}
|I_3|\le C{\color{black}\epsilon_\flat^{n }\epsilon_\sharp}t^{-1+y+\epsilon_y-2z+\epsilon_z}  .
\ee

\end{itemize}
In view of \eqref{E:-AddReduction-9}, \eqref{E:-AddReduction-8-+}, \eqref{E:-2ndReduction-5}, and \eqref{E:-2ndReduction-8}, to optimize the estimates, we obtain
\be\label{E:maximise}
y=\frac 14+\frac{\epsilon_\ast} 2	,\quad z=\frac 1{12}+\frac{\epsilon_\ast}{6}+\frac{\epsilon_z}3+\frac{\epsilon_y}3. 
\ee
Hence, choosing $\frac{\epsilon_\ast}{6}+\frac{\epsilon_z}3+\frac{\epsilon_y}3=\frac\epsilon 2$, 
\be\label{E:maximise-I-j}
|I_j|= \epsilon_\flat^{n+1} o_{\epsilon }( {t}^{-\frac{11}{12}+ \epsilon   }  ) \ \textit{for }j=1,2,3 ,
\ee together with Theorem \ref{T:first-reduction} and Lemmas   \ref{L:cio-0}, \ref{L:3-reduction-u20},  yield \eqref{E:-u-2}.

\end{proof}

\section{Long time asymptotics of $u_{2,1}(x)$}\label{S:m-1}

We adapt the approach from Section \ref{S:m} to derive the asymptotic behavior of $u_{2,1}$.  To facilitate integration by parts and ensure integrability without imposing additional conditions on $\partial^j_{\lambda'_I}\widetilde s_c$ and $\lambda'\widetilde s_c$ near $\lambda'_I=0$ (cf \cite{Ki06}), particular care is required, and the argument becomes more involved.

Throughout this section,      $a$, $r$, $t_i$, $t$ are as defined in    \eqref{E:real-ima}, \eqref{E:a}-\eqref{E:stationary--}, and   $\rho_1,\rho_2,\,\epsilon_1,\,\epsilon_2,\,\epsilon_\flat$, and $\epsilon_\sharp$ are defined according to \eqref{E:large-15}.

\subsection{ABF Cauchy integrals}\label{SS:u21-technical} 

Formally,
\begin{align} 
  & \partial_{x_1}\widetilde{\mathcal CT1 }(t_1,t_2,t;\overline\lambda') 
=   \frac{-1}{2\pi i}\iint\frac{ 2\pi i\xi'_1 e^{2\pi i tS_0 } \widetilde{  s} (\zeta ' ) }{(1+\xi'^4_1)(\zeta'-\overline\lambda')}  d\overline\zeta'\wedge d\zeta'  \label{E:u21-CT-ABF}\\
&\qquad\qquad=    \frac{-1}{2\pi i}\iint dx_1'dx_2'\mathbf v  (x'_1,x'_2; t_2 )  \iint d{\xi'_1}   d{\xi'_2} \frac{ 2\pi i\xi'_1e^{ 2\pi it\wp}}{(1+\xi'^4_1)({\xi'_2} -z_0) }   . \nonumber 
\end{align}

In view of
\begin{gather}
\left.\overline{\zeta }'_h\right|_{\xi'_{2,h}=2\pi i\xi'^2_{1,h}+2\xi'_{1,h}\overline\lambda'_h,\ \xi'_{1,h}=\xi''_h+\frac{\lambda'_{I,h}}{2\pi}}=\lambda'_{R,h}+2\pi i\xi''_h=\overline\lambda'_{h-1},\nonumber\\
\xi'_{1,h}-\frac{\lambda'_{I,h}}{2\pi}=\xi''_h,\quad 2\pi\xi''_h=-\lambda'_{I,h-1},\quad \xi' _{1,n+1}=-\frac{\lambda_I'}{ \pi},\label{Aux-1}\\
(\xi_h''-\frac{\lambda'_{I,h}}{2\pi})(\xi_h''+\frac{\lambda'_{I,h}}{2\pi})=(\xi_h''+\xi''_{h+1})(\xi_h''-\xi''_{h+1}),\quad
\overline\xi'_h-\xi'_h=2\pi i\xi'_{1,h}=2\pi i(\xi''_h-\xi''_{h+1}),\nonumber
\end{gather}
analogous to Definition \ref{D:CIO-1} and \ref{D:n-CIO}, we introduce

\begin{definition}[the ABF Cauchy integrals]\label{D:u21-CIO-1} Define $[\mathfrak C \mathfrak T1]^{1,(n)}$ and $\mathfrak C \mathfrak T_{1,(n)}$ by: 
\begin{align}
  & e^{- i\pi tS_0(a;\lambda') } [\mathfrak C \mathfrak T1]^{1,(1)}
=  e^{ -i\pi tS_0(a;\lambda') }\iint dx_1'dx_2' \mathbf v  (x'_1,x'_2; t_2 ) e^{i\lambda'_I(x_1'+2\lambda'_Rx_2')}\label{E:d-CT1-u12}\\
\times& \int d\xi_1''   e^{2\pi i t \mathfrak G^\sharp} 2\pi i\xi'_1 \Theta(t;\lambda';x_1',x_2';\xi_1'')\mathcal F (t;\lambda';x_1',x_2';\xi_1'')    
\equiv    e^{ -i\pi tS_0(a;\lambda') }   \mathfrak C \mathfrak T_{(1)}  2\pi i\xi'_1,\nonumber
\end{align} and 
\begin{align} 
 & e^{- \beta_ni\pi tS_0(a;\lambda') }\left[ \mathfrak C \mathfrak T 1 \right]^{1,(n)}  
 = e^{- \beta_ni\pi tS_0(a;\lambda') } \sum_{h=1}^n \mathfrak C \mathfrak T_{0, (n)}\cdots \mathfrak C \mathfrak T_{0, (h+1)} \label{E:dct-n-aux-u12} \\
&\qquad\qquad    \times \{\,\mathfrak C \mathfrak T_{0, (h)}  (2\pi i ) (\xi_h''-\xi_{ h+1}'' )   \, \} \left[ \mathfrak C \mathfrak T 1 \right]^{0, (h-1)} (t_1,t_2,t;   {\lambda'_R}+2\pi i\xi_h''). \nonumber
\end{align}  
\end{definition}

\begin{lemma}\label{L:dCIO-u21} 
\begin{gather}
|[\mathfrak C \mathfrak T1]^{1,(n)}|\le C_\epsilon{\color{black}\epsilon_\flat^{n-1 }\epsilon_\sharp}(t^{-\frac12+\epsilon}), \quad\textit{ as $t\to\infty$},\label{E:CIO-+new-u21} \\
\textit{$[\mathfrak C \mathfrak T1]^{1,(n)}$ is holomorphic in $\lambda'_R\lambda'_I$ when $\lambda'_I\ne 0$,}\label{E:u21-n-holo}
\end{gather}
and
\begin{align}
|\partial_{\lambda'_I}\left[\mathfrak C \mathfrak T1\right]^{1, (n)}| \le  C(1+|\lambda_R'|) \epsilon_\flat^{n-1}\epsilon_\sharp,   \label{E:u21-n-lambda-i}  \\
\theta(|\lambda_R'|-1/C_0) |
 \partial_{\lambda'_R} [\mathfrak C \mathfrak T1]^{1,(n)}] 
\le    C (1+ |\lambda_I'| ){\color{black}\epsilon_\flat^{n-1}\epsilon_\sharp}. \label{E:u21-n-lambda-r}
\end{align}

\end{lemma}

\begin{proof} 
The proof proceeds by the same argument as in Lemma \ref{L:CIO} and  \ref{L:cio-0}. 

\end{proof}

\subsection{The Cauchy integrals and a reduced representation formula for $u_{2,1}(x)$}\label{SS:CI-1} 

Without the factor $(\overline \lambda'-\lambda')$, we have to impose an additional condition, \[1/\delta>|a|,\] to assure the differentiability and integrability of $(1+|\lambda'|)^j\widetilde s_c$   for $j= 0,\,1$. In particular, \eqref{E:lambda-L1} is no longer valid which is crucial to derive the reduced representation formula and to estimate the ABF Cauchy integrals (cf. \eqref{E:-AddReduction-8-+}). To address this issue, let us define
\be\label{E:confine-u20}
 \Lambda(\lambda')=\theta(10r-|\lambda'_R|)+\theta(|\lambda'_R|-10r )\theta( |\lambda'_I|- \frac{r}{10} ). 
\ee  
 
\begin{lemma}\label{L:LAMBDA} Suppose $1/\delta>|a|>\delta>0 $. Hence
\be\label{E:confine--u20}
 |\Lambda(\lambda')\widetilde {  s}_c(\lambda'  )|_{L^1(d\lambda'_R d\lambda'_I)}\le C\epsilon_\flat. 
\ee

\end{lemma} 
\begin{proof} Follows from $1/\delta>|a|>\delta>0$, \eqref{E:intro-s-c-ana-c-0}, and \eqref{E:real-ima-new} directly. 

\end{proof}

We can follow the same strategy of Theorem \ref{T:first-reduction} to derive a reduced representation formula of $u_{2,1}$. In particular, we will  investigate (cf. \eqref{E:large-18}, \eqref{E:large-32}, \eqref{E:small-8}, \eqref{E:small-15}, \eqref{E:+0-small-8}, \eqref{E:+0-small-15})
\begin{align}
&\mathfrak R _{in,\gtrless,-}= \frac{ -1}{2\pi i}  \iint_{\Sigma_{<}}  dx_1'dx_2' \mathbf v (x'_1,x'_2;t_2 ) \iint  d{\xi'_1}d{\xi'_2}\frac{2\pi i\xi'_1 e^{ 2\pi i t\wp } }{(1+\xi'^4_1)(\xi'_2-z_0)} \label{E:lambda-R-flat}\\
&\ \ \times  \Xi^c(\xi'_1)  \times \left\{\begin{array}{l}[1-\theta_{\rho_1,0}( \lambda'_R )]\\
\theta_{\rho_1,0}( \lambda'_R )
\end{array}\right\} \times
\theta(\frac14 - |\omega|)  ,\qquad \omega=\frac{\xi'_2}{|\xi'_1|\rho_1},\nonumber\\
&\mathfrak P _{in,\gtrless,-}=  \frac{ -1}{2\pi i}\iint_{\Sigma_{<}}  dx_1'dx_2' \mathbf v (x'_1,x'_2;t_2 ) \iint  d{\xi'_1}d{\xi'_2}  \frac{2\pi i\xi'_1e^{ 2\pi i t\wp  } }{(1+\xi'^4_1)(\xi'_2-z_0)} \label{E:lambda-p}\\
&\ \ \times \theta_{r_\ast,\frac {r}{ \pi}}(\xi_1') \times \left\{\begin{array}{l}[1-\theta_{\rho_1,0}( \lambda'_R )]\\
\theta_{\rho_1,0}( \lambda'_R )
\end{array}\right\} \times \theta( \frac 14-|\omega|)   ,\qquad \qquad \quad\omega=\frac{\xi'_2}{|\xi'_1|\rho_1},\nonumber \\
&\mathfrak  Q _{in,\gtrless,-} = \dfrac {-1}{2\pi i}  \iint_{\Sigma_{<} } dx_1'dx_2'\mathbf v ( x'_1,x_2';t_2 ) \iint d{\xi'_1}d{\xi'_2}   \frac{2\pi i\xi'_1e^{ 2\pi i t\wp }}{(1+\xi'^4_1)(\xi'_2-z_0)}\label{E:lambda-q>} \\
&\ \ \times  \theta_{r_\ast,0}(\xi_1')\times
\left\{\begin{array}{l}
[1-\theta_{|\xi'_1  |^{1/2}\rho_1,0}(\lambda'_R)] \\
\theta_{|\xi'_1  |^{1/2}\rho_1,0}(\lambda'_R)
\end{array}\right\}\times\theta(\frac 14-|\omega|)  
 ,\quad\quad\omega=\frac{\xi'_2}{|\xi'_1|^{1/2}\rho_1},\nonumber
\end{align}and
\be\label{E:lambda-q>sharp-0}
\mathfrak  Q _{in,\gtrless,-} = \mathfrak  Q _{in,\gtrless,-}^\sharp+\mathfrak  Q _{in,\gtrless,-}^\flat=-(\overline\lambda'-\lambda')  Q _{in,\gtrless,-} +\mathfrak  Q _{in,\gtrless,-}^\flat,
\ee with
\begin{align}
\mathfrak  Q _{in,\gtrless,-}^\flat=&\dfrac {-1}{2\pi i}  \iint_{\Sigma_{<} } dx_1'dx_2'\mathbf v ( x'_1,x_2';t_2 ) \iint d{\xi'_1}d{\xi'_2}   \frac{2 i(\pi\xi'_1-\lambda'_I)e^{ 2\pi i t\wp }}{(1+\xi'^4_1)(\xi'_2-z_0)}, \label{E:lambda-q>sharp}\\
\times&  \theta_{r_\ast,0}(\xi_1')\times
\left\{\begin{array}{l}
[1-\theta_{|\xi'_1  |^{1/2}\rho_1,0}(\lambda'_R)] \\
\theta_{|\xi'_1  |^{1/2}\rho_1,0}(\lambda'_R)
\end{array}\right\}\times \theta(\frac 14-|\omega|),\qquad  \omega=\frac{\xi'_2}{|\xi'_1|^{1/2}\rho_1} ,\nonumber\\
\mathfrak  Q _{in,\gtrless,-}^\sharp=&\dfrac {-1}{2\pi i}  \iint_{\Sigma_{<} } dx_1'dx_2'\mathbf v ( x'_1,x_2';t_2 ) \iint d{\xi'_1}d{\xi'_2}   \frac{2  i\lambda_I' e^{ 2\pi i t\wp }}{(1+\xi'^4_1)(\xi'_2-z_0)}, \label{E:lambda-q>flat}\\
\times&  \theta_{r_\ast,0}(\xi_1')\times
\left\{\begin{array}{l}
[1-\theta_{|\xi'_1  |^{1/2}\rho_1,0}(\lambda'_R)] \\
\theta_{|\xi'_1  |^{1/2}\rho_1,0}(\lambda'_R)
\end{array}\right\}\times\theta(\frac 14-|\omega|)\nonumber\\
=& -(\overline\lambda'-\lambda')  Q _{in,\gtrless,-},\hskip2.3in  \omega=\frac{\xi'_2}{|\xi'_1|^{1/2}\rho_1} .\nonumber 
\end{align}

\begin{lemma}\label{L:LAMBDA-CIO}  Suppose   $1/\delta>|a|>\delta>0$.   For $ j=0,1$, $K=\mathfrak R _{in,>,-}$, $\mathfrak P _{in,>,-}$, or $\mathfrak Q _{in,>,-}$,  defined by \eqref{E:lambda-R-flat}-\eqref{E:lambda-q>},  
\begin{align}
|\theta( |\lambda'_R|- 10r ) \theta( \frac{r}{10} - |\lambda'_I| )  \partial_{\lambda_I'}^j  \left[ \mathfrak C \mathfrak T 1 \right]^{1,(n)}|\le &C{\color{black}\epsilon_\flat^{n-1}\epsilon_\sharp}\frac{1  }{[t(1+ |\lambda' _R|^2 )]^{1-j}} , \label{E:lambda-1} \\
|\theta( |\lambda'_R|- 10r ) \theta( \frac{r}{10} - |\lambda'_I| )  \partial_{\lambda_I'}^j  K|\le &C\epsilon_\flat \frac{1 }{( 1+ |\lambda' _R|   )^{1+j}} \mathcal O(t^{-\frac 12+ \epsilon_1}). \label{E:lambda-2}
\end{align}

Consequently, for $n\ge 1$,
\be \label{E:lambda-3}
\iint d\overline\lambda'\wedge d\lambda'  \widetilde {  s}_c(\lambda'  )e^{\beta_{n+1}2\pi i tS_0 }    \theta( |\lambda'_R|- 10r ) \theta( \frac{r}{10} - |\lambda'_I| )   \widetilde{(\mathcal CT)^{n}  } 1= \epsilon_\flat^n\epsilon_\sharp   \mathcal O(t^{- 1}).
\ee

\end{lemma}

\begin{proof}  First of all, proof for \eqref{E:lambda-2}    follows  from the definition  \eqref{E:lambda-R-flat}, \eqref{E:lambda-p} and  \eqref{E:lambda-q>}.

\noindent  {$\blacktriangleright$}  $ \underline{\textbf{ Step 1. Proof for \eqref{E:lambda-1}:}}$  
To prove \eqref{E:lambda-1}, it is equivalent to proving
\begin{align}
|\theta( |\lambda'_R|- 10r ) \theta( \frac{r}{10} - |\lambda'_I| )  \partial_{\lambda_I'}^j \mathfrak C \mathfrak T_{1,(n)} \left[ \mathfrak C \mathfrak T 1 \right]^{0,(n-1)}|\le &C{\color{black}\epsilon_\flat^{n-1}\epsilon_\sharp}\frac{1 }{[t(1+ |\lambda' _R|^2 )]^{1-j}},
\label{E:chi-integrability-1} \\
|\theta( |\lambda'_R|- 10r ) \theta( \frac{r}{10} - |\lambda'_I| )  \partial_{\lambda_I'}^j\mathfrak C \mathfrak T_{(n)} \left[ \mathfrak C \mathfrak T 1 \right]^{1,(n-1)}|\le &C{\color{black}\epsilon_\flat^{n-1}\epsilon_\sharp}\frac{1 }{[t(1+ |\lambda' _R|^2 )]^{1-j}}.\label{E:chi-integrability-2} 
\end{align}

Proofs for \eqref{E:chi-integrability-1} and \eqref{E:chi-integrability-2} are identical. We only provide the proof for \eqref{E:chi-integrability-2} when $j=1$ for simplicity. To this aim, decompose
\begin{multline}\label{E:chi-1}
 | \partial_{\lambda_I'} \mathfrak C \mathfrak T_{(n)} \left[ \mathfrak C \mathfrak T 1 \right]^{1,(n-1)}|
\le   | \partial_{\lambda_I'} \mathfrak C \mathfrak T_{(n)}  \theta( ||\xi_n''|-\frac{|\lambda_I'|}{2\pi} |-r)\left[ \mathfrak C \mathfrak T 1 \right]^{1,(n-1)}|\\
  + | \partial_{\lambda_I'} \mathfrak C \mathfrak T_{(n)} [1-\theta(||\xi_n''|-\frac{|\lambda_I'|}{2\pi} |-r)]\left[ \mathfrak C \mathfrak T 1 \right]^{1,(n-1)}|  
\equiv I_1+I_2 .
\end{multline} 

From
$\theta( |\lambda'_R|- 10r )\theta(t|\lambda_R'|-|x_{2,n}'|)\theta( ||\xi_n''|-\frac{|\lambda_I'|}{2\pi} |-r) |(x_{2,n}'+3t\lambda_R')(\xi_n''-\frac{\lambda_I'}{2\pi})(\xi_n''+\frac{\lambda_I'}{2\pi})|>\frac{t}C$, therefore, the amplitude function $\mathcal F^{(n)}\le Ct^{-\frac tC}$ and then
\be\label{E:chi-2}
|I_1|\le C{\color{black}\epsilon_\flat^{n-1}\epsilon_\sharp} o((t\lambda' _R)^{-1}). 
\ee

For $I_2$,   we have
\be\label{E:first-reduction-5-u21-G}
|\theta( |\lambda'_R|- 10r )\theta( \frac{r}{10} - |\lambda'_I| )[1-\theta(||\xi_n''|-\frac{|\lambda_I'|}{2\pi} |- r)]\partial_{\xi''_n}\mathfrak S |\ge   {   {\lambda'}_R ^2 } /C.
\ee  Therefore, applying integration by parts with respect to $\xi''_n$, 
\begin{align}
& I_2 
\le  \frac{C}{t}|\iint dx_{1,n}'dx_{2,n}'\mathbf v ( x'_{1,n},x_{2,n}';t_2 )(x_{1,n}'+2\lambda_R'x_{2,n}')e^{-\beta_ni\lambda_I'(x_{1,n}'+2\lambda_R'x_{2,n}')}\label{E:chi-3}\\
\times&\int d\xi_n''e^{\beta_n2\pi it \mathfrak S^\sharp}\partial_{\xi''_n}\{\ \frac{[1-\theta(||\xi_n''|-\frac{|\lambda_I'|}{2\pi} |-r)]\theta(-(x_{2,n}'+3t\lambda_R')(\xi_n''-\frac{\lambda_I'}{2\pi})(\xi_n''+\frac{\lambda_I'}{2\pi}))}{\partial_{\xi''_n}\mathfrak S}\nonumber\\
\times&\frac{e^{-2\pi i(x_{1,n}'+2\lambda'_Rx'_{2,n})\xi''_n}}{1+\xi'^4_{1,n}}
e^{\beta_n4\pi^2(x_{2,n}'+3t\lambda_R')(\xi_n''-\frac{\lambda_I'}{2\pi})(\xi_n''+\frac{\lambda_I'}{2\pi})}\left[ \mathfrak C \mathfrak T 1 \right]^{1,(n-1)}|\}\nonumber\\
+&\frac{C}{t}|\iint dx_{1,n}'dx_{2,n}'\mathbf v ( x'_{1,n},x_{2,n}';t_2 )e^{-\beta_ni\lambda_I'(x_{1,n}'+2\lambda_R'x_{2,n}')} \nonumber\\
\times&\partial_{\lambda'_I}\int d\xi_n''e^{\beta_n2\pi it \mathfrak S^\sharp}\left[\partial_{\xi''_n}\frac{[1-\theta(||\xi_n''|-\frac{|\lambda_I'|}{2\pi} |-r)]\theta(-(x_{2,n}'+3t\lambda_R')(\xi_n''-\frac{\lambda_I'}{2\pi})(\xi_n''+\frac{\lambda_I'}{2\pi})) }{\partial_{\xi''_n}\mathfrak S}\right]\nonumber\\
\times& \frac{e^{-2\pi i(x_{1,n}'+2\lambda'_Rx'_{2,n})\xi''_n}}{1+\xi'^4_{1,n}}e^{\beta_n4\pi^2(x_{2,n}'+3t\lambda_R')(\xi_n''-\frac{\lambda_I'}{2\pi})(\xi_n''+\frac{\lambda_I'}{2\pi})}\left[ \mathfrak C \mathfrak T 1 \right]^{1,(n-1)}|\nonumber\\
+&\frac{C}{t}|\iint dx_{1,n}'dx_{2,n}'\mathbf v ( x'_{1,n},x_{2,n}';t_2 ) e^{-\beta_ni\lambda_I'(x_{1,n}'+2\lambda_R'x_{2,n}')}\nonumber\\
\times&\partial_{\lambda'_I}\int d\xi_n''e^{\beta_n2\pi it \mathfrak S^\sharp}\frac{[1-\theta(||\xi_n''|-\frac{|\lambda_I'|}{2\pi} |-r)]\theta(-(x_{2,n}'+3t\lambda_R')(\xi_n''-\frac{\lambda_I'}{2\pi})(\xi_n''+\frac{\lambda_I'}{2\pi})) }{\partial_{\xi''_n}\mathfrak S}\nonumber\\
\times&\partial_{\xi''_n}\frac{e^{-2\pi i(x_{1,n}'+2\lambda'_Rx'_{2,n})\xi''_n}}{1+\xi'^4_{1,n}}e^{\beta_n4\pi^2(x_{2,n}'+3t\lambda_R')(\xi_n''-\frac{\lambda_I'}{2\pi})(\xi_n''+\frac{\lambda_I'}{2\pi})}\left[ \mathfrak C \mathfrak T 1 \right]^{1,(n-1)}|\nonumber\\
\equiv& I_{2,1}+I_{2,2}+I_{2,3}.\nonumber
\end{align}

Using \eqref{E:first-reduction-5-u21-G}, the factors $\theta( |\lambda'_R|- 10r ) $, $\theta( \frac{r}{10} - |\lambda'_I| )$, $[1-\theta(||\xi_n''|-\frac{|\lambda_I'|}{2\pi} |-r)]$, and, if necessary, exchanging the order of  $\partial_{\lambda'_I}$ and $\partial_{\xi''_n}$,
\be\label{E:chi-4}
|\theta( |\lambda'_R|- 10r ) \theta( \frac{r}{10} - |\lambda'_I| ) I_{2,j}|\le {\color{black}\epsilon_\flat^{n-1}\epsilon_\sharp}\frac 1t\frac{{\lambda'}_R^2t }{1+{\lambda'}^2_R},\quad j=1,2,3.
\ee

Combining \eqref{E:chi-1}, \eqref{E:chi-2}, \eqref{E:chi-3}, and \eqref{E:chi-4}, we prove \eqref{E:chi-integrability-2}.

\noindent  {$\blacktriangleright$}  $ \underline{\textbf{ Step 2. Proof for \eqref{E:lambda-3}:}}$ Thanks to 
\be\label{E:chi-5}
 \theta( |\lambda'_R|- 10r ) \theta( \frac{r}{10} - |\lambda'_I| )  |\partial_{\lambda'_I}S_0 | \ge  1/C.
\ee

Consequently, \eqref{E:lambda-3} follows from applying integration by parts with respect to $\lambda'_I$,   Lemma \ref{L:sd-der}, \eqref{E:lambda-1}, and \eqref{E:lambda-2}.
\end{proof}  

 We also need the following technical lemma. 
\begin{lemma}\label{L:u21-0-pre}
\begin{align}
  \mathfrak Q^\flat_{in,<,-}= & \epsilon_\flat\theta_{\rho_1,0}(\lambda'_R)\mathcal O(t^{-\frac 12+2\epsilon_2}),\label{E:u21-q-<0} \\
  \mathfrak Q^\flat_{in,>,-}=& \epsilon_\flat\mathcal O(t^{-\frac 12+2\epsilon_1}),\label{E:u21-q->0}\\
 [1-\theta_{\rho_1,0}(\lambda'_R)] \mathfrak Q^\flat_{in,>,-}=& \epsilon _\flat  \mathcal O(t^{- 1+2\epsilon_2}) , \label{E:u21-q->0-more}\\
  [1-\theta_{\rho_1,0}(\lambda'_R)]  \partial_{\lambda'_I}\mathfrak Q^\flat_{in,>,-}   | ,\ [1-\theta_{\rho_1,0}(\lambda'_R)] |\partial_{\lambda'_R}\mathfrak Q^\flat_{in,>,-}   = &  \epsilon_\flat\mathcal O(t^{+\epsilon_2}) . \label{E:u21-q->0-extra}
\end{align}

\end{lemma} 
\begin{proof} Via the change of variables \eqref{E:large-26}, we write $2(\pi\xi'_1-\lambda'_I)=\eta_1-\lambda'_I$, and assume $\lambda'_I=b_0>0$ without loss of generality throughout the proof.

\noindent  {$\blacktriangleright$}  $ \underline{\textbf{ Step 1. Proof for \eqref{E:u21-q-<0} and \eqref{E:u21-q->0}:}}$ First of all, if $|\lambda'_I|\ge 2r_\ast$, we have
\[
 \mathfrak Q^\flat_{in,<,-}=  \epsilon_\flat\theta_{\rho_1,0}(\lambda'_R)\mathcal O(t^{-\frac 12+ \epsilon_2}),\qquad  \mathfrak Q^\flat_{in,>,-}=  \epsilon_\flat\mathcal O(t^{-\frac 12+ \epsilon_1}).
\]Therefore, we may assume that $|\lambda'_I|\leq 2r_\ast$ in the proof of this step.

For $\mathfrak Q^\flat_{in,<,-}$,  setting $x=\eta_1-b_0=\sqrt{\rho_2} y$, $b_0=\sqrt{\rho_2}\beta$, and using \eqref{E:large-29-(1)}, \eqref{E:large-29-(1)-(add)}, and
\be |\xi'_{2 }| \le C|\xi'_1|^{1/2}\rho_2  ,\label{E:u21-q->0-1}
\ee we derive
\begin{align}
|\mathfrak Q^\flat_{in,<,-} |\le&C |\iint_{\Sigma_{<}}  dx_1'dx_2' \mathbf v (x'_1,x'_2;t_2 )\iint  d{\xi'_1} d{\xi'_2} \frac{2(\pi\xi'_1-\lambda'_I)  e^{ 2\pi i t\wp  }}{(1+\xi'^4_1)(\xi'_2-(2\pi\xi'^2_1+2\xi'_1\overline\lambda'))}\label{E:u21-q->0-2}\\
\times&  \theta_{r_\ast,0}(\xi_1') \theta_{|\xi'_1|^{1/2}\rho_1,0}(\lambda'_R)  \theta(\frac 14-|\omega|)   | \nonumber \\
\le &C\epsilon_\flat \theta_{ \rho_1,0}(\lambda'_R) | \int_{|\eta_1+\lambda'_I|\le r_\ast }d\eta_1    \int_{|\eta_2 |\le C\rho_2 }  \frac{(\eta_1-b_0)d\eta _2}{\sqrt{(\eta_1^2-b^2_0)^2+\eta_2 ^2}}|\nonumber
\\
\le &  C \epsilon_\flat  \theta_{ \rho_1,0}(\lambda'_R) \int_{|x|\le r_\ast }|x|\texttt{arcsinh}\frac{ \rho_2 }{ |{ x(x+2b_0)}|}dx \nonumber\\
\le & C\epsilon_\flat \theta_{ \rho_1,0}(\lambda'_R){ \rho_2}\int_{|y|<r_\ast/\sqrt{\rho_2}}|y|\texttt{arcsinh}\frac{ 1 }{ |y||y+2\beta|}dy \nonumber\\
=&\epsilon_\flat\theta_{ \rho_1,0}(\lambda'_R)\mathcal O(t^{-\frac 12+ 2\epsilon_2 } ) .\nonumber
\end{align}

Besides adapting \eqref{E:u21-q->0-2} for $|\xi'_1\lambda'_R|\le\rho_1$, setting 
$x=\eta_1-b_0=\sqrt{\rho_1} y$, $b_0=\sqrt{\rho_1}\beta$,
using
\begin{align}
  &|  \int_{|\eta_1+\lambda'_I|\le r_\ast }d\eta_1      \frac{(\eta_1-b_0)\rho_1}{\sqrt{(\eta_1^2-b^2_0)^2+(4\pi\rho_1) ^2}}|\label{E:u21-q->0-3}\\
\le & C {\rho_1}\int_{|x|<r_\ast}  dx      \frac{|x|}{\sqrt{x^2(x+2b_0)^2  +(4\pi\rho_1)^2}}\nonumber\\
\le & C {\rho_1}\int_{|y|<r_\ast/\sqrt{\rho_1}}  dy      \frac{|y|}{\sqrt{y^2(y+2\beta)^2  +(4\pi )^2}}\nonumber\\
= &  \mathcal O(t^{-\frac 12+ 2\epsilon_1 } ) .\nonumber
\end{align}
additionally for $|\xi'_1\lambda'_R|\ge\rho_1$,    we can derive
\be\label{E:u21-q->0-4}
\mathfrak Q^\flat_{in,>,-} = \epsilon_\flat\mathcal O(t^{-\frac 12+ 2\epsilon_1 } ).
\ee

\noindent  {$\blacktriangleright$}  $ \underline{\textbf{ Step 2. Proof for \eqref{E:u21-q->0-more} on the regime $|\xi'_1|>\rho_1$:}}$ To prove it, we replace \eqref{E:u20-0-pre-1} and \eqref{E:u20-0-pre-2}   by
\begin{align}
\bullet\quad&[1-\theta_{\rho_1,0}(\xi'_1)] [1-\theta_{\rho_1,0}(\lambda'_R)] \theta(\frac 14-|\omega|) |\partial_{\xi'_1}\frac 1{{\xi'_2} -(2\pi i{\xi'_1}^2+2{\xi'_1}\overline\lambda') }|
\label{E:u21-q->0-6}\\
\le & C[1-\theta_{\rho_1,0}(\xi'_1)] [1-\theta_{\rho_1,0}(\lambda'_R)] \theta(\frac 14-|\omega|) \nonumber\\
\times&(| \frac { 1}{\xi'_1({\xi'_2} -(2\pi i{\xi'_1}^2+2{\xi'_1}\overline\lambda')) }|+C| \frac { 1}{\lambda'_R({\xi'_2} -(2\pi i{\xi'_1}^2+2{\xi'_1}\overline\lambda')) }|)
\nonumber\\
\le & C[1-\theta_{\rho_1,0}(\xi'_1)] [1-\theta_{\rho_1,0}(\lambda'_R)] \theta(\frac 14-|\omega|)  | \frac { t^{\frac 12-\epsilon_1}  }{  {\xi'_2} -(2\pi i{\xi'_1}^2+2{\xi'_1}\overline\lambda')  }| ,
\nonumber\\
\bullet\quad &[1-\theta_{\rho_1,0}(\xi'_1)] [1-\theta_{\rho_1,0}(\lambda'_R)] \theta(\frac 14-|\omega|)\partial_{\xi'_1}e^{2\pi i [-t  \frac 34\frac{{\xi'_2}^2}{\xi'_1} -x_2'{\xi'_2}  ]} | \label{E:u21-q->0-7}\\
\le& [1-\theta_{\rho_1,0}(\xi'_1)] [1-\theta_{\rho_1,0}(\lambda'_R)] \theta(\frac 14-|\omega|)|  2\pi i (t\frac 34\frac{{\xi_2'}^2}{{\xi'_1}^2} )e^{2\pi i [-t  \frac 34\frac{{\xi'_2}^2}{\xi'_1} -x_2'{\xi'_2}  ]}| \nonumber\\
\le& Ct^{1-1+2\epsilon_1+\frac 12-\epsilon_1} .\nonumber
\end{align}

As a result, integration by parts with respect to $\xi'_1$ (taking advantage of $0$ is not the stationary point of $\wp_1$) and applying \eqref{E:u21-q->0}, we obtain
\begin{align}
 &| [1-\theta_{ \rho_1,0}(\lambda'_R)]\iint_{\Sigma_{<}}  dx_1'dx_2' \mathbf v (x'_1,x'_2;t_2 ) \iint  d{\xi'_1}d{\xi'_2} \frac{ e^{ 2\pi i t\wp  }}{(1+\xi'^4_1)(\xi'_2-z_0)}\label{E:u21-q->0-8}\\
\times&\theta_{r_\ast,0}(\xi_1')[1-\theta_{\rho_1,0}(\xi'_1)] \theta(\frac 14-|\omega|)  |\le\epsilon_\flat\mathcal O(t^{-1+(-\frac 12+\epsilon_2 )+\frac 12+\epsilon_1}) = \epsilon_\flat\mathcal O(t^{-1+2 \epsilon_2  }).\nonumber
\end{align}

\noindent  {$\blacktriangleright$}  $ \underline{\textbf{ Step 3. Proof for \eqref{E:u21-q->0-more} on the regime $|\xi'_1|<\rho_1$:}}$ 
\begin{itemize}
\item if $|\xi'\lambda'_R|\le \rho_1$, using $|\xi'_1|<\rho_1$,
\begin{align} 
&|[1-\theta_{2r_\ast,0}(\frac{\lambda'_I}{\pi} )]\iint_{\Sigma_{<}}  dx_1'dx_2' \mathbf v (x'_1,x'_2;t_2 ) \iint  d{\xi'_1}d{\xi'_2} \frac{2(\pi\xi'_1-\lambda'_I) e^{ 2\pi i t\wp  }}{(1+\xi'^4_1)(\xi'_2-z_0)}\label{E:u21-q->0-5}\\
\times& \theta_{\rho_1,0}(\xi'_1)[1-\theta_{|\xi'_1|^{1/2}\rho_1,0}(\lambda'_R)]  \theta(\frac 14-|\omega|)  |\nonumber\\
= &\epsilon_\flat\mathcal O(t^{-1+2\epsilon_1}),\nonumber
\end{align} and,  setting $x=\eta_1-b_0= {\rho_1} y$, $b_0= {\rho_1}\beta$, 
\begin{align} 
&|\theta_{2r_\ast,0}(\frac{\lambda'_I}{\pi} )\iint_{\Sigma_{<}}  dx_1'dx_2' \mathbf v (x'_1,x'_2;t_2 ) \iint  d{\xi'_1}d{\xi'_2} \frac{2(\pi\xi'_1-\lambda'_I) e^{ 2\pi i t\wp  }}{(1+\xi'^4_1)(\xi'_2-z_0)}\label{E:u21-q->0-5-new}\\
\times& \theta_{\rho_1,0}(\xi'_1)[1-\theta_{|\xi'_1|^{1/2}\rho_1,0}(\lambda'_R)]  \theta(\frac 14-|\omega|)  |\nonumber\\
\le &C\epsilon_\flat\theta_{2r_\ast,0}(\frac{\lambda'_I}{\pi} )\int_{|\eta_1-b_0 |\le C\rho_1 }d\eta_1    \int_{|\eta_2 |\le C\rho_1 }  \frac{(\eta_1-b_0)d\eta_2}{\sqrt{(\eta_1^2-b^2_0)^2+\eta_2 ^2}}    \nonumber\\
\le &  C \epsilon_\flat    \int_{|x|\le C\rho_1}|x|\texttt{arcsinh}\frac{ \rho_1 }{ |{|x||x+2b_0|}|}dx \nonumber\\
\le & C\epsilon_\flat    \rho_1 ^2\int_{|y|<C }|y|\texttt{arcsinh}\frac{1/ \rho_1 }{ |y||y+2\beta|}dy \nonumber\\
= &\epsilon_\flat\mathcal O(t^{-1+2\epsilon_2}),\nonumber
\end{align}

\item if $|\xi'\lambda'_R|\ge \rho_1$, 
\begin{align} 
&| \iint_{\Sigma_{<}}  dx_1'dx_2' \mathbf v (x'_1,x'_2;t_2 ) \iint  d{\xi'_1}d{\xi'_2} \frac{ e^{ 2\pi i t\wp  }}{(1+\xi'^4_1)(\xi'_2-z_0)}\label{E:u21-q->0-6-new}\\
\times&\theta_{t^{-\frac 12},0}(\xi'_1) [1-\theta_{|\xi'_1|^{1/2}\rho_1,0}(\lambda'_R)] \theta(\frac 14-|\omega|)  |\nonumber\\
\le &C\epsilon_\flat\int_{|\eta_1+\lambda'_I|\le \rho_1 }d\eta_1      \frac{(\eta_1-b_0)\rho_1}{\sqrt{(\eta_1^2-b^2_0)^2+(4\pi\rho_1) ^2}}    = \epsilon_\flat\mathcal O(t^{-1+ \epsilon_2}).\nonumber
\end{align} 
\end{itemize}

Combining \eqref{E:u21-q->0-8}, \eqref{E:u21-q->0-5}, and \eqref{E:u21-q->0-6},   we establish \eqref{E:u21-q->0-more}. 

\noindent  {$\blacktriangleright$}  $ \underline{\textbf{ Step 4. Proof for \eqref{E:u21-q->0-extra}:}}$  Applying \eqref{E:+0-small-12}, \eqref{E:+0-small-12-(1)}, \eqref{E:u21-q-<0}, \eqref{E:u21-q->0},
\begin{align}
 & [1-\theta_{\rho_1 ,0}(\lambda'_R)] \theta(\frac 14-|\omega|) |\partial_{\lambda'_I}\frac {\pi\xi'-\lambda'_I}{{\xi'_2} -(2\pi i{\xi'_1}^2+2{\xi'_1}\overline\lambda') }|
\nonumber\\
\le&[1-\theta_{\rho_1 ,0}(\lambda'_R)] \theta(\frac 14-|\omega|) | \frac {1}{{\xi'_2} -(2\pi i{\xi'_1}^2+2{\xi'_1}\overline\lambda') }|\nonumber\\
+& C[1-\theta_{ \rho_1,0}(\lambda'_R)] \theta(\frac 14-|\omega|)   \frac {t^{\frac 12-\epsilon_1}(\pi\xi'-\lambda'_I)}{| {\xi'_2} -(2\pi i{\xi'_1}^2+2{\xi'_1}\overline\lambda')| } , \nonumber 
\end{align}we obtain
\begin{multline}\label{E:u20-0-pre-10-new}
[1-\theta_{\rho_1,0}(\lambda'_R)  |\partial_{\lambda'_I}\mathfrak Q^\flat_{in,>,-}(t_1,t_2,t;\lambda') |\\
\le\epsilon_\flat\mathcal O(t^{-\frac 14+\epsilon_1}+t^{+\frac 12-\epsilon_1-\frac 12+2\epsilon_1})=\epsilon_\flat\mathcal O(t^{+\epsilon_2}).\qquad\qquad
\end{multline}

Proof for the case of the $\lambda'_R$-derivative is similar.
\end{proof}

Using Lemmas \ref{L:dCIO-u21}-\ref{L:u21-0-pre}, and Theorem \ref{T:first-reduction},  we refine the argument of Lemma \ref{L:u20-small} to derive the reduced representation formula of $u_{2,1}$.
\begin{theorem}[A reduced representation formula of $u_{2,1}$]\label{T:u21-first-reduction} Suppose that $1/\delta>|a|>\delta>0$ and \eqref{intro-ID} holds.
Then, for every sufficiently small $\epsilon>0$,
as $t\to\infty$,
\be \label{E:u21-first-reduction-1} 
    u_{2,1}(x)=   -\frac{1}{\pi i}  \sum_{n=1}^\infty  |\iint d\overline\lambda'\wedge d\lambda'  \widetilde {  s}_c(\lambda'  )\Lambda(\lambda')e^{\beta_{n+1}2\pi i tS_0 }          \left[ \mathfrak C \mathfrak T 1 \right]^{1,(n)}  +\epsilon _\flat\epsilon_\sharp o_{\epsilon } ( t^{-\frac{19}{20}+ \epsilon  }). 
\ee

\end{theorem}

\begin{proof} Applying Lemmas \ref{L:LAMBDA} and \ref{L:LAMBDA-CIO}, we obtain
\begin{align}
 &-\dfrac 1{\pi i} \iint d\overline\lambda'\wedge d\lambda' e^{2\pi i tS_0 } \widetilde {  s}_c(\lambda'  )  \widetilde{(\mathcal CT)^{n-h}  }   \widetilde{ \mathcal CT } (2\pi i\xi'_{1,h})\widetilde{(\mathcal CT)^{h-1}  } 1 \label{E:u21-first-reduction-2}\\
= &-\dfrac 1{\pi i} \iint d\overline\lambda'\wedge d\lambda' e^{ 2\pi i tS_0 } \widetilde {  s}_c(\lambda'  )   \Lambda(\lambda')  
\widetilde{(\mathcal CT)^{n-h}  } \widetilde{ \mathcal CT } (2\pi i\xi'_{1,h})\widetilde{(\mathcal CT)^{h-1}  } 1+\epsilon^{n   }_\flat\epsilon_\sharp\mathcal O  (t^{-1 }). \nonumber
\end{align}  

\noindent  {$\blacktriangleright$}  $\underline{\textbf{ Step 1.   $n=h$:}}$
In view of  
\begin{align}
&\textit{$\overline\lambda'-\lambda'$ is not used in the proof    of Lemmas \ref{L:CIO-large}-\ref{L:0-small}, $\frac{2\pi i\xi'}{1+|\xi'_1|^4}$ is bounded, }\label{E:u21-first-reduction-3}
\end{align} the formulas  \eqref{E:lambda-q>}-\eqref{E:lambda-q>flat}, 
  and Theorem \ref{T:first-reduction}, 
 we can prove 
\begin{align}
&-\dfrac 1{\pi i} \iint d\overline\lambda'\wedge d\lambda' e^{ 2\pi i tS_0 } \widetilde {  s}_c(\lambda'  )   \Lambda(\lambda')  
 \widetilde{ \mathcal CT } (2\pi i\xi'_{1,h})\widetilde{(\mathcal CT)^{h-1}  } 1\label{E:u21-first-reduction-5}\\
=&-\dfrac 1{\pi i} \iint d\overline\lambda'\wedge d\lambda' e^{ \beta_{h+1}2\pi i tS_0 } \widetilde {  s}_c(\lambda'  )   \Lambda(\lambda')  
 \mathfrak C\mathfrak T_{(h)}2\pi i\xi'_{1,n} \left[\mathfrak C\mathfrak T\right]^{0,(h-1)}    \nonumber\\
+&  \dfrac {-1}{\pi i}\iint d\overline\lambda'\wedge d\lambda' e^{2\pi i tS_0 } \widetilde {  s}_c(\lambda'  ) \Lambda(\lambda')  \mathfrak Q ^\flat_{in,\gtrless,-}e^{-\beta_{h-1 }  \pi i tS_0 }\left[\mathfrak C\mathfrak T\right]^{0,(h-1)}    
+\epsilon _\flat^{h} \epsilon_\sharp\mathcal O_{\epsilon_1}( t^{-\frac{19}{20}+ \epsilon }). \nonumber
\end{align}

Applying \eqref{E:u21-q-<0},  we justify
\be \label{E:u21-first-reduction-6}
\iint d\overline\lambda'\wedge d\lambda' e^{2\pi i tS_0 } \widetilde {  s}_c(\lambda'  ) \Lambda(\lambda')  \mathfrak Q ^\flat_{in,<,-}e^{-\beta_{h-1 }  \pi i tS_0 }\left[\mathfrak C\mathfrak T\right]^{0,(h-1)}  
=  \epsilon _\flat^{h+1}\mathcal O(t^{-1  +3\epsilon_2} ). 
\ee

Besides, 
\begin{align}
&\iint d\overline\lambda'\wedge d\lambda' e^{2\pi i tS_0 } \widetilde {  s}_c(\lambda'  ) \Lambda(\lambda')  \mathfrak Q ^\flat_{in,>,-}e^{-\beta_{h-1 }  \pi i tS_0 }\left[\mathfrak C\mathfrak T\right]^{0,(h-1)} \label{E:u21-first-reduction-7}\\
=&\iint d\overline\lambda'\wedge d\lambda' e^{2\pi i tS_0 } \widetilde {  s}_c(\lambda'  ) \Lambda(\lambda') \theta_{\rho_1,0}(\lambda'_R) \mathfrak Q ^\flat_{in,>,-}e^{-\beta_{h-1 }  \pi i tS_0 }\left[\mathfrak C\mathfrak T\right]^{0,(h-1)}\nonumber\\
+& \iint d\overline\lambda'\wedge d\lambda' e^{2\pi i tS_0 } \widetilde {  s}_c(\lambda'  ) \Lambda(\lambda') [1-\theta_{\rho_1,0}(\lambda'_R)][1-\theta_{2r_\ast,0}(\lambda'_I)] \mathfrak Q ^\flat_{in,>,-}\nonumber\\
\times& e^{-\beta_{h-1 }  \pi i tS_0 }\left[\mathfrak C\mathfrak T\right]^{0,(h-1)}\nonumber\\
+& \iint d\overline\lambda'\wedge d\lambda' e^{2\pi i tS_0 } \widetilde {  s}_c(\lambda'  ) \Lambda(\lambda') [1-\theta_{\rho_1,0}(\lambda'_R)]\theta_{\rho_1,0}(\lambda'_I)\mathfrak Q ^\flat_{in,>,-} e^{-\beta_{h-1 }  \pi i tS_0 }\left[\mathfrak C\mathfrak T\right]^{0,(h-1)}\nonumber\\
+& \iint d\overline\lambda'\wedge d\lambda' e^{2\pi i tS_0 } \widetilde {  s}_c(\lambda'  ) \Lambda(\lambda') [1-\theta_{\rho_1,0}(\lambda'_R)][1-\theta_{\rho_1,0}(\lambda'_I)]\theta_{2r_\ast,0}(\lambda'_I) \nonumber\\
\times&\mathfrak Q ^\flat_{in,>,-} e^{-\beta_{h-1 }  \pi i tS_0 }\left[\mathfrak C\mathfrak T\right]^{0,(h-1)}\nonumber\\
=&\epsilon_\flat^{h+1} \mathcal O(t^{-1+2\epsilon_2}).\nonumber
\end{align}Here we have used \eqref{E:u21-q->0}, \eqref{E:u21-q->0-more}, and adapted the argument used to estimate $I_{Q_{in,>,-}}$ in {\bf Step 2  of   Lemma \ref{L:u20-0-small}} to estimate the term localized by $[1-\theta_{\rho_1,0}(\lambda'_R)][1-\theta_{\rho_1,0}(\lambda'_I)]\theta_{2r_\ast,0}(\lambda'_I)$.

Therefore, setting $4\epsilon_2=\epsilon$, we have established, for $n=h$,
\begin{align}
&-\dfrac 1{\pi i} \iint d\overline\lambda'\wedge d\lambda' e^{ 2\pi i tS_0 } \widetilde {  s}_c(\lambda'  )   \Lambda(\lambda')  
 \widetilde{ \mathcal CT } (2\pi i\xi'_{1,h})\widetilde{(\mathcal CT)^{h-1}  } 1\label{E:u21-first-reduction-13}\\
=&-\dfrac 1{\pi i} \iint d\overline\lambda'\wedge d\lambda' e^{ \beta_{n+1}2\pi i tS_0 } \widetilde {  s}_c(\lambda'  )   \Lambda(\lambda')  
 \mathfrak C\mathfrak T_{(n)}2\pi i\xi'_{1,n} \left[\mathfrak C\mathfrak T 1 \right] ^{0,(h-1)}+ \epsilon _\flat^{h} \epsilon_\sharp\mathcal O_{\epsilon_1}(  t^{-\frac{19}{20}+  \epsilon   }). \nonumber
\end{align}

\noindent  {$\blacktriangleright$}  $\underline{\textbf{ Step 2. $n>h$:}}$  Using \eqref{E:lambda-q>sharp-0}, Lemma \ref{L:u21-0-pre}, Theorem \ref{T:first-reduction},  adapting the argument used in Lemma \ref{L:u20-0-small} and {\bf Step 1},  we obtain
\begin{align}
&-\dfrac 1{\pi i} \iint d\overline\lambda'\wedge d\lambda' e^{ 2\pi i tS_0 } \widetilde {  s}_c(\lambda'  )   \Lambda(\lambda')  
 \widetilde{(\mathcal CT)^{n-h}  } \widetilde{ \mathcal CT } (2\pi i\xi'_{1,h})\widetilde{(\mathcal CT)^{h-1}  } 1\label{E:u21-first-reduction-14}\\
=&-\dfrac 1{\pi i} \iint d\overline\lambda'\wedge d\lambda' e^{\beta_{n+1} 2\pi i tS_0 } \widetilde {  s}_c(\lambda'  )   \Lambda(\lambda')  
 \mathfrak C\mathfrak T_{(n)}\dots  \mathfrak C\mathfrak T_{(h)}2\pi i\xi'_{1,h} \left[\mathfrak C\mathfrak T\right]^{0,(h-1)}    \nonumber\\
+& \epsilon _\flat^{n} \epsilon_\sharp\mathcal O_{\epsilon_1}( t^{-\frac{19}{20}+ \epsilon }). \nonumber
\end{align}

The proof of the theorem is complete by applying the dominated convergence theorem and Definition \ref{D:u21-CIO-1}, and  
$
\sum_{n=1}^{\infty}n r^n
<Cr$  for $0<r<1$.
 \end{proof}

\begin{proposition}\label{P:first-reduction-u12}   Suppose the assumptions of Theorem \ref{T:u} holds. 
Then, as $t\to\infty$,  
 \begin{align}
&  u_{2,1}(x)=   -\frac{1}{\pi i}\sum_{n=1}^\infty \sum_{h=1}^{n}  |\iint d\overline\lambda'\wedge d\lambda'\ e^{\beta_{n+1}2\pi i tS_0 } \widetilde {  s}_c(\lambda'  ) \Lambda(\lambda') \theta(|\lambda_R'|-t^{\color{black}-\frac 12-2\epsilon})\label{E:first-reduction-2-sum-epsilon}\\
 \times&\mathfrak C \mathfrak T_{(n)}\cdots  \mathfrak C \mathfrak T_{(h+1)}  \mathfrak C \mathfrak T_{(h)}  (\xi_h''-\xi''_{h+1}) \theta_{1,\xi''_{h+1}}( \xi_h''   )    \left[ \mathfrak C \mathfrak T 1 \right]^{0,(h-1)} +\epsilon _\flat\epsilon_\sharp \mathcal O_{\epsilon } ( t^{-\frac{19}{20}+ \epsilon  }). \nonumber
 \end{align}     
\end{proposition}

\begin{proof}   
The proof is justified by applying Theorem \ref{T:u21-first-reduction},  Lemma \ref{L:cio-0},  \eqref{E:CIO-+new-u21}, and
\begin{align*}
&\theta(|\lambda_R'|-t^{-\frac 12-2\epsilon})\left[1- \theta_{1,\xi''_{h+1}}( \xi_h''   ) \right]  \theta(t|\lambda_R'|-|x_{2,h}'|) \\
&\qquad\times|(  x_{2,h}'+3  t\lambda'_R)(\xi_h''-\xi''_{h+1})  (\xi_h''+\xi''_{h+1})|\ge  
 t ^{\frac 12+2\epsilon  }/C. 
\end{align*}  
\end{proof}

\subsection{Long time asymptotics of $u_{2,1}(x) $ when    $+1/\delta>a >+\delta>0$ }\label{SS:u+2-u21}

Throughout this subsection, we assume that $+1/\delta>a >+\delta>0$ and  \eqref{intro-ID} holds.   We define $
\psi_{r,w_0}$ and $\theta_{r,w_0}$     as in  \eqref{E:localize-fct} and    \eqref{E:wp-5}, respectively, take  $\epsilon_\flat$  and  $\epsilon_\sharp$  as in \eqref{E:large-15},  and set  $  b=  ( -r^2+{\lambda'}_R^2)^{1/2}/{2\pi}$.

From Proposition \ref{P:first-reduction-u12},  it suffices to investigate
\begin{align}
u_{2,1}^\dagger=&\sum_{n=1}^\infty \sum_{h=1}^{n} |\iint d\overline\lambda'\wedge d\lambda'\ e^{\beta_{n+1}2\pi i tS_0 } \widetilde {  s}_c(\lambda'  ) \Lambda(\lambda') \theta(|\lambda_R'|-t^{-5/9})\label{E:>-first-reduction-2-sum}\\
 \times&\mathfrak C \mathfrak T_{(n)}\cdots  \mathfrak C \mathfrak T_{(h+1)}  \mathfrak C \mathfrak T_{(h)}   (\xi_h''-\xi''_{h+1} ) \theta_{1,\xi''_{h+1}}( \xi_h''   )  \left[ \mathfrak C \mathfrak T 1 \right]^{0,(h-1)}|.\nonumber
\end{align}
We will prove $u^\dagger_{2,1}= \epsilon_\flat\epsilon_\sharp o(t^{-1 } ) $ by decomposing the estimatesinto two parts, depending on whether $||\xi_{h}''|- |\xi''_{h+1}||> t^{-6/9}$ or not. Precisely,
\begin{lemma}\label{L:+2ndReduction-P}    As $t\to\infty$, 
\begin{align}
  u_{2,1}^\dagger(x)\le  & C\sum_{n=1}^\infty  |\iint d\overline\lambda'\wedge d\lambda'\ e^{\beta_{n+1}2\pi i tS_0 } \widetilde {  s}_c(\lambda'  ) \Lambda(\lambda') \theta(|\lambda_R'|-t^{-5/9}) \label{E:2ndreduction-2-+}\\
\times&\sum_{h=1}^{n} (P_{n,h}^>+ P_{n,h}^< )\left[ \mathfrak C \mathfrak T 1 \right]^{0,(h-1)}|+ \epsilon_\flat^2o(t^{-1})  , \nonumber
 \end{align}where
  \begin{align}
 P_{n,h}^>= &    \mathfrak C \mathfrak T_{(n)}\cdots \mathfrak C \mathfrak T_{(h+1)}  \mathfrak C \mathfrak T_{(h)}\theta_{1,\xi_{h+1}'' }( \xi_h''   ) (\xi_h''-\xi''_{h+1})[1-\theta_{t^{-6/9},\xi_{h+1}''}(\xi_h'')], \label{E:--first-reduction-2}\\
 P_{n,h}^<=&   \mathfrak C \mathfrak T_{(n)}\cdots\{\ \mathfrak C \mathfrak T_{(h+1)}(-2\xi_{h+1}'' )  \theta_{1,\xi_{h+2}''}( \xi_{h+1} '' ) [1-\theta_{t^{-6/9},\xi_{h+2}''}( \xi_{h+1}'' )]
   \nonumber\\
\times& \mathfrak C \mathfrak T_{(h)}    \theta(t^{-6/9}-|\xi''_h-\xi''_{h+1}|)  \theta(|\xi_h'' +\xi''_{h+1}|-t^{-6/9})\ \}  .\nonumber 
\end{align} 
Here,  for brevity,  when $h=n$, we  identify  
\be \label{E:assertion} 
\mathfrak C \mathfrak T_{(n)}\cdots \mathfrak C \mathfrak T_{(h+1)} (-2\xi_{1,h+1}') 
     \theta_{1, \xi''_{h+2}}( \xi_{h +1}''  ) [1- \theta_{t^{-6/9},-\xi''_{h+2}}( \xi_{h+1}'' )=  \frac{ \lambda_I'}{\pi}.  \ee
\end{lemma} 
\begin{proof} From  \eqref{E:>-first-reduction-2-sum}, it reduces to studying 
\begin{align}
&  \sum_{n=1}^\infty \sum_{h=1}^{ n} |\iint d\overline\lambda'\wedge d\lambda' \ e^{\beta_{n+1}2\pi i tS_0 }  \widetilde {  s}_c(\lambda'  ) \Lambda(\lambda')\theta(|\lambda_R'|-t^{-5/9})     \nonumber\\
\times&\mathfrak C \mathfrak T_{(n)}\cdots
     \mathfrak C \mathfrak T_{(h+1)} \mathfrak C \mathfrak T_{(h)}  \theta_{1,\xi''_{h+1}}( \xi_h''   ) (\xi_h''-\xi''_{h+1})   \theta_{t^{-6/9},\xi''_{h+1}} (\xi_h'')   \left[ \mathfrak C \mathfrak T 1 \right]^{0,(h-1)}|, \nonumber 
\end{align} which is less than
\begin{align}
 & \sum_{n=1}^\infty \sum_{h=1}^{ n}    |\iint d\overline\lambda'\wedge d\lambda' \ e^{\beta_{n+1}2\pi i tS_0 }\widetilde {  s}_c(\lambda'  )\Lambda(\lambda') \theta(|\lambda_R'|-t^{-5/9})    \label{E:first-reduction-23}\\
\times&\mathfrak C \mathfrak T_{(n)}\cdots
      \mathfrak C \mathfrak T_{(h+1)} \mathfrak C \mathfrak T_{(h)}  \theta(t^{-6/9}-| \xi_h'' -\xi''_{h+1} | ) (\xi_h''-\xi''_{h+1})    \left[ \mathfrak C \mathfrak T 1 \right]^{0,(h-1)} |\nonumber\\
+& \sum_{n=1}^\infty \sum_{h=1}^{ n}   |\iint d\overline\lambda'\wedge d\lambda' \  e^{\beta_{n+1}2\pi i tS_0 }  \widetilde {  s}_c(\lambda'  )\Lambda(\lambda') \theta(|\lambda_R'|-t^{-5/9})  \nonumber\\
\times&\mathfrak C \mathfrak T_{(n)}\cdots
     \mathfrak C \mathfrak T_{(h+1)} \mathfrak C \mathfrak T_{(h)}   \theta(|\xi_h''-\xi''_{h+1}|-t^{-6/9})\theta(t^{-6/9}-| \xi_h'' +\xi''_{h+1}|  ) (\xi_h''+\xi''_{h+1}) \nonumber\\
\times&  \left[ \mathfrak C \mathfrak T 1 \right]^{0,(h-1)} |\nonumber\\
+& \sum_{n=1}^\infty \sum_{h=1}^{ n}    |\iint d\overline\lambda'\wedge d\lambda' \ e^{\beta_{n+1}2\pi i tS_0 } \widetilde {  s}_c(\lambda'  ) \Lambda(\lambda') \theta(|\lambda_R'|-t^{-5/9})  \nonumber\\
\times&\mathfrak C \mathfrak T_{(n)}\cdots
     \mathfrak C \mathfrak T_{(h+1)} \mathfrak C \mathfrak T_{(h)}  \theta(|\xi_h'' -\xi''_{h+1}|-t^{-6/9}) \theta(t^{-6/9}-|  \xi_h''+\xi''_{h+1}|   )  (-2  \xi_{ h+1}'' ) \nonumber\\
\times&   \left[ \mathfrak C \mathfrak T 1 \right]^{0,(h-1)}  |  \nonumber\\
\equiv&\sum_{n=1}^\infty \sum_{h=1}^{ n}  Q_{n,h}^{>,-}+\sum_{n=1}^\infty \sum_{h=1}^{ n}Q_{n,h}^{>,+}+\sum_{n=1}^\infty \sum_{h=1}^{ n}Q_{n,h}^< .\nonumber
\end{align}

Using \eqref{E:confine--u20}, $
\sum_{n=1}^{\infty}n r^n
=
\frac{r}{(1-r)^2}$, and   $|(\xi_h''\pm  \xi''_{h+1}) 
\theta(t^{-6/9}-|\xi_h''\pm  \xi''_{h+1}|)|_{L^1(d\xi_h'')}\le   C  (t^{-6/9\times 2})$,
we obtain
\be \label{E:AddReduction-dec-1}
\sum_{n=1}^\infty \sum_{h=1}^{ n} Q_{n,h}^{>,\pm}  \le C  \epsilon_\flat^2 t^{-6/9\times 2}  .
\ee

Applying the above argument, we have
\begin{align} 
&\sum_{n=1}^\infty \sum_{h=1}^{ n}Q_{n,h}^< \label{E:AddReduction-dec-2}\\ 
\le& \sum_{n=1}^\infty \sum_{h=1}^{ n}|\iint d\overline\lambda'\wedge d\lambda' \ e^{\beta_{n+1}2\pi i tS_0 } \widetilde {  s}_c(\lambda'  )\Lambda(\lambda')   \theta(|\lambda_R'|-t^{-5/9}) \nonumber\\
\times&\mathfrak C \mathfrak T_{(n)}\cdots
    \mathfrak C \mathfrak T_{(h+1)}(-2 \xi_{h+1}'')  \theta_{1, \xi''_{h+2}}( \xi_{h +1} ''  ) [1- \theta_{t^{-6/9}, \xi''_{h+2}}( \xi_{h+1}'' ) ]    \nonumber\\
\times&\mathfrak C \mathfrak T_{(h)}   \theta( t^{-6/9}-|\xi_h'' -\xi''_{ h+1} |) \theta(|\xi_h'' +\xi''_{ h+1}|-t^{-6/9})  \left[ \mathfrak C \mathfrak T 1 \right]^{0,(h-1)} |  +\epsilon_\flat^2  o(t^{-1})\nonumber\\
= &\sum_{n=1}^\infty  |\iint d\overline\lambda'\wedge d\lambda'\ e^{\beta_{n+1}2\pi i tS_0 } \widetilde {  s}_c(\lambda'  )\Lambda(\lambda') \theta(|\lambda_R'|-t^{-5/9})  P_{n,h}^<\left[ \mathfrak C \mathfrak T 1 \right]^{0,(h-1)}|+\epsilon_\flat^2 o(t^{-1})\nonumber.
\end{align}


\end{proof}

\begin{lemma}\label{L:+2ndReduction-u21-outside}
Suppose that $+1/\delta>a >+\delta>0$.  As $t\to\infty$,  
\begin{align}
 &  \sum_{n=1}^\infty  |\iint d\overline\lambda'\wedge d\lambda'  \ e^{\beta_{n+1}2\pi i tS_0 }\widetilde {  s}_c(\lambda'  ) \Lambda(\lambda')  \theta(|\lambda_R'|-t^{-5/9}){\color{black}\theta( a-3{\lambda'_R}^2) } \label{E:-+2ndReduction-(1)-u21-outside}\\
\times&\sum_{h=1}^{n}(P_{n,h}^>+ P_{n,h}^< )[1-\psi_{r,r}(\lambda'_R){\color{black}\psi_{ 5r,0}}(2\pi\xi_{h+1}'')]\left[ \mathfrak C \mathfrak T 1 \right]^{0,(h-1)}|=\epsilon_\flat \epsilon_\sharp  o(t^{-1}).\nonumber
 \end{align}

\end{lemma} 
\begin{proof}With the factor $\Lambda(\lambda')$, the dominated convergence theorem, Defintion \ref{D:u21-CIO-1}, and  
$
\sum_{n=1}^{\infty}n r^n
=
\frac{r}{(1-r)^2}$, the lemma can be proved by adapting the argument used in Lemma \ref{L:+3ndReduction-outside}. 

\end{proof}

\begin{lemma}\label{L:+2ndReduction-u21-inside-o}
Suppose that $+1/\delta>a >+\delta>0$.    As $t\to\infty$,  
\begin{align}
 &  \sum_{n=1}^\infty  |\iint d\overline\lambda'\wedge d\lambda'  \ e^{\beta_{n+1}2\pi i tS_0 }\widetilde {  s}_c(\lambda'  )\Lambda(\lambda')   {\color{black}\theta( a-3{\lambda'_R}^2) } \label{E:-+2ndReduction-(1)-u21-inside-o}\\
&\qquad\times\sum_{h=1}^{n} P_{n,h}^> \psi_{r,r}(\lambda'_R){\color{black}\psi_{ 5r,0}}(2\pi\xi_{h+1}'')\left[ \mathfrak C \mathfrak T 1 \right]^{0,(h-1)}|=\epsilon_\flat^2  o(t^{-1}) ,\nonumber\\
 & \sum_{n=1}^\infty  |\iint d\overline\lambda'\wedge d\lambda' \ e^{\beta_{n+1}2\pi i tS_0 } \widetilde {  s}_c(\lambda'  ) \Lambda(\lambda') {\color{black}\theta( -(a-3{\lambda'_R}^2))} \label{E:++2ndReduction-(1)-u21-inside-o}\\
&\qquad\times\sum_{h=1}^{n}P_{n,h}^>  \left[ \mathfrak C \mathfrak T 1 \right]^{0,(h-1)}|= \epsilon_\flat^2  o(t^{-1}).\nonumber
 \end{align}

\end{lemma}

\begin{proof} We  will first discard terms with rapidly decaying amplitudes. Then, through a refined decomposition, we derive the necessary estimates by leveraging the smallness of the integration domains and the   factor  or $(\xi_h''\pm \xi_{h+1}'')$. Integration by parts is not required in the proof.

To prove \eqref{E:++2ndReduction-(1)-u21-inside-o}, decompose 
\begin{align}
  &  \theta(t|\lambda_R'|-|x_{2,h}'|) \theta(||\xi_h''|- | \xi''_{h+1}||-t^{-6/9}) \label{E:+3rdReduction-1-u21}\\
=&   \theta(t|\lambda_R'|-|x_{2,h}'|) \theta(||\xi_h''|- |\xi''_{h+1}||-t^{-4.4/9}) \nonumber\\
+&  \theta(t|\lambda_R'|-|x_{2,h}'|) {\color{black}[1- \theta(||\xi_h''|- |\xi''_{h+1}||-t^{-4.4/9})]}\theta(||\xi_h''|- | \xi''_{h+1}||-t^{-6/9})
.\nonumber
\end{align}

Thanks to ${\color{black}\theta( -(a-3{\lambda'_R}^2))}$, $|\lambda_R'|>r/2>0$ as $t\gg 1$. Hence
\begin{multline}\label{E:-2ndReduction-6-(new)-1}
   \theta(t|\lambda_R'|-|x_{2,h}'|) \theta(||\xi_h''|- |\xi''_{h+1}||-t^{-4.4/9}) \theta(t|\lambda_R'|-|x_{2,h }'|)  \\
 \times    |(  x_{2,h }'+3  t\lambda'_R)(\xi_h''-\xi_{h+1}'')  (\xi_h''+\xi_{h+1}'') | \ge t^{ 1- 4.4/9\times 2 }/C . 
 \end{multline} which implies that the $L^1(d\xi''_h)$-norm of the amplitude function  $ \mathcal F^{(h)}$ on the corresponding domain of the first term is less than $ \epsilon_\flat^2  o(t^{-1})$. Together with  \eqref{E:confine--u20}, it reduces to showing
\begin{align}
 &\sum_{n=1}^{\infty} \sum_{h=1}^{n}    \iint d\overline\lambda'\wedge d\lambda' \ e^{\beta_{n+1}2\pi i tS_0 }   \widetilde {  s}_c(\lambda'  )\Lambda(\lambda')  {\color{black}\theta( -(a-3{\lambda'_R}^2))}    \label{E:+3rdReduction-2-u21}\\
\times &  P_{n,h}^>  
{\color{black}[1- \theta(||\xi_h''|- |\xi''_{h+1}||-t^{-4.4/9})]}\left[ \mathfrak C \mathfrak T 1 \right]^{0,(h-1)}= \epsilon_\flat^2 o(t^{-1}). \nonumber
\end{align}

Notice
\begin{align}
&\textit{LHS of \eqref{E:+3rdReduction-2-u21} }\label{E:+3rdReduction-5-u21}\\
\le & \sum_{n=1}^\infty\sum_{h=1}^n|\iint d\overline\lambda'\wedge d\lambda' \ e^{\beta_{n+1}2\pi i tS_0 }\widetilde {  s}_c(\lambda'  )\Lambda(\lambda'){\color{black}\theta( -(a-3{\lambda'_R}^2))} \nonumber\\
\times&\mathfrak C\mathfrak T_{(n)}\cdots \mathfrak C\mathfrak T_{(h+1)}  \psi_{ t^{-2.5/9},0}(2\pi \xi_{h+1}'')\nonumber\\
\times&\mathfrak C\mathfrak T_{(h)}{\color{black}[1- \theta(||\xi_h''|- |\xi''_{h+1}||-t^{-4.4/9})]}(\xi_h''-\xi''_{h+1}) \theta(||\xi_h''|- | \xi''_{h+1}||-t^{-6/9})\nonumber\\
\times&\left[\mathfrak C\mathfrak T1\right]^{0,(h-1)}|\nonumber\\
+&\sum_{n=1}^\infty\sum_{h=1}^n|\iint d\overline\lambda'\wedge d\lambda' \ e^{\beta_{n+1}2\pi i tS_0 }\widetilde {  s}_c(\lambda'  )\Lambda(\lambda') {\color{black}\theta( -(a-3{\lambda'_R}^2))} \nonumber\\
\times&\mathfrak C\mathfrak T_{(n)}\cdots \mathfrak C\mathfrak T_{(h+1)}[1-\psi_{t^ {-2.5/9},0}(2\pi \xi_{h+1}'')]\nonumber\\
\times&\mathfrak C\mathfrak T_{(h)}{\color{black}[1- \theta(||\xi_h''|- |\xi''_{h+1}||-t^{-4.4/9})]}(\xi_h''- \xi''_{h+1}) \theta(||\xi_h''|- | \xi''_{h+1}||-t^{-6/9})\nonumber\\
\times& \left[\mathfrak C\mathfrak T1\right]^{0,(h-1)}| 
\equiv II_1+II_2.\nonumber
\end{align} 

Using  
\begin{gather}
|(\psi_{ t^{-2.5/9},0}(2\pi \xi_{h+1}'')|_{L^1(d\xi_{h+1}'')}\le Ct^{-2.5/9},\nonumber\\
|(\xi_h''- \xi''_{h+1}) {\color{black}[1- \theta(||\xi_h''|- |\xi''_{h+1}||-t^{-4.4/9})]}|_{L^1(d\xi_h'')}\le C(t^{-4.4/9 }+| \xi''_{h+1} |)t^{-4.4/9 },\label{E:+3rdReduction-9-1}
\end{gather}   \eqref{E:confine--u20},  and  
$
\sum_{n=1}^{\infty}n r^n
=
\frac{r}{(1-r)^2}$, we obtain  
\be\label{E:+3rdReduction-9} 
\begin{split}
|II_1|\le C \epsilon_\flat^2 \left(\mathcal O(t^{-2.5/9\times 2-4.4/9 })+\mathcal O(t^{-2.5/9-4.4/9\times 2})\right).
\end{split}
\ee 

Besides, 
on the support of $(1-\psi_{ t^{-2.5/9},0}(2\pi \xi_{h+1}''))$, distance between $\pm \xi_{h+1}''$ is greater than $\mathcal O(t^{-2.5/9})$. Combining with $|\lambda_R'|>r $ on the support of ${\color{black}\theta( -(a-3{\lambda'_R}^2))}$, \label{E:+3rdReduction-5-d-new}  
\begin{multline}
 (1-\psi_{ t^{-2.5/9},0}(2\pi \xi_{h+1}''))\theta(t|\lambda_R'|-|x_{2,h}'|) \theta( -(a-3{\lambda'_R}^2))  \\
\times \theta(||\xi_h''|- |\xi''_{h+1}||-t^{-6/9})  |(  x_{2,h}'+3  t\lambda'_R)(\xi_h''+\xi''_{h+1})  (\xi_h''-\xi''_{h+1})| 
\ge   t^{1-6/9-2.5/9}/C,
\end{multline}which, together with  \eqref{E:confine--u20}  and  
$
\sum_{n=1}^{\infty}n r^n
=
\frac{r}{(1-r)^2}$, implies 
 \be\label{E:+3rdReduction-5-d} 
\begin{split}
|II_2|=  \epsilon_\flat^2  o(t^{-1}).
\end{split}
\ee
Therefore, \eqref{E:+3rdReduction-2-u21} is justified.

  Since $|\lambda_R'|>r/C$ is assured by the factor $\psi_{r,r}(\lambda'_R)$. We can prove  \eqref{E:-+2ndReduction-(1)-u21-inside-o} by analogy.

\end{proof}

\begin{lemma}\label{L:+2ndReduction-u21-inside}
Suppose that $+1/\delta>a >+\delta>0$.  As $t\to\infty$,  
\begin{align}
 &  \sum_{n=1}^\infty \sum_{h=1}^{n}  |\iint d\overline\lambda'\wedge d\lambda'\ e^{\beta_{n+1}2\pi i tS_0 } \widetilde {  s}_c(\lambda'  )\Lambda(\lambda')  {\color{black}\theta( a-3{\lambda'_R}^2)} \label{E:-+2ndReduction-(1)-u21-inside}\\
&\qquad\times P_{n,h}^<  \psi_{r,r}(\lambda'_R){\color{black}\psi_{ 5r,0}}(2\pi\xi_{h+1}'')\left[ \mathfrak C \mathfrak T 1 \right]^{0,(h-1)}|=   \epsilon_\flat\epsilon_\sharp  o (t^{-1 } )  ,\nonumber\\
 & \sum_{n=1}^\infty \sum_{h=1}^{n} |\iint d\overline\lambda'\wedge d\lambda' \ e^{\beta_{n+1}2\pi i tS_0 }\widetilde {  s}_c(\lambda'  ) \Lambda(\lambda'){\color{black}\theta( -(a-3{\lambda'_R}^2))} \label{E:++2ndReduction-(1)-u21-inside}\\
&\qquad\times P_{n,h}^<  \left[ \mathfrak C \mathfrak T 1 \right]^{0,(h-1)}|=\epsilon_\flat\epsilon_\sharp   o (t^{-1 } ).\nonumber
 \end{align}

\end{lemma} 

\begin{proof} From  \eqref{E:confine--u20} and \eqref{E:+3rdReduction-1-u21},  
\begin{align}
&\textit{LHS of \eqref{E:++2ndReduction-(1)-u21-inside}}\label{E:+3rdReduction-3-u21}\\
\le &\sum_{n=1}^{\infty} \sum_{h=1}^{n}   |\iint d\overline\lambda'\wedge d\lambda' \ e^{\beta_{n+1}2\pi i tS_0 } \widetilde {  s}_c(\lambda'  )\Lambda(\lambda'){\color{black}\theta( -(a-3{\lambda'_R}^2))}  \nonumber\\
\times&  \mathfrak C\mathfrak T_{(n)}\cdots  \mathfrak C \mathfrak T_{(h+1)} (-2 \xi_{h+1}'' ) {\color{black}\theta_{1,\xi''_{h+2}}(\xi_{h+1}''  )} \nonumber\\
\times& \theta(||\xi_{h+1}''|- |\xi''_{h+2}||-t^{-6/9}){\color{black}[1-\theta(||\xi_{h+1}''|- |\xi''_{h+2}||-t^{-4.4/9})]}
   \nonumber\\
\times&  \mathfrak C \mathfrak T_{(h)}    \theta(t^{-6/9}-|\xi_h''+\xi''_{h+1}|)\theta(|\xi_h''-\xi''_{h+1}|-t^{-6/9})   \left[ \mathfrak C \mathfrak T 1 \right]^{0,(h-1)}|. 
 \nonumber 
 \end{align}
 
Decomposing $  \xi_{h+1}''=  (\xi_{h+1}''-\xi_{h+2}'')+  \xi_{h+2}''$,  applying Lemmas \ref{L:+3ndReduction-outside}, \ref{L:+2ndReduction},  \ref{L:+2ndReduction-u21-inside-o},   an induction,  and $
\sum_{n=1}^{\infty}n r^n
=
\frac{r}{(1-r)^2}$, we have 
\be\label{E:+3rdReduction-4-u21}
\textit{RHS of \eqref{E:+3rdReduction-3-u21} }=\epsilon_\flat\epsilon_\sharp  o (t^{-1 } ).
\ee

In an entirely similar way, we can prove \eqref{E:-+2ndReduction-(1)-u21-inside}.
\end{proof}

Combining  Proposition \ref{P:first-reduction-u12}, \eqref{E:>-first-reduction-2-sum}, and Lemmas \ref{L:+2ndReduction-P}-\ref{L:+2ndReduction-u21-inside},   we conclude:
\begin{theorem}\label{T:+u-2-u21}  Suppose that $+1/\delta>a>+\delta>0$ and \eqref{intro-ID} holds.
Then, for every sufficiently small $\epsilon>0$,
as $t\to\infty$,
\be \label{E:+u-2-u21}
   u_{2,1}=    \epsilon _\flat\epsilon_\sharp o_{\epsilon } (t^{-\frac{19}{20} + \epsilon } ) .
\ee
\end{theorem}

\subsection{Long time asymptotics of $u_{2,1}(x) $ when    $-1/\delta<a <-\delta<0$ }\label{S:u-2-u21}

Throughout this subsection, we assume that $-1/\delta<a <-\delta<0$  and  \eqref{intro-ID} holds.   We define $
\psi_{r,w_0}$ and $\theta_{r,w_0}$     as in  \eqref{E:localize-fct} and    \eqref{E:wp-5}, respectively, take  $\epsilon_\flat$  and  $\epsilon_\sharp$  as in \eqref{E:large-15},  and set  $  b=  (  r^2+{\lambda'}_R^2)^{1/2}/{2\pi}$. 

Similarly, building on   Proposition \ref{P:first-reduction-u12}, it reduces to estimate
\begin{align}
u_{2,1}^\dagger=&\sum_{n=1}^\infty \sum_{h=1}^{n} |\iint d\overline\lambda'\wedge d\lambda'\ e^{\beta_{n+1}2\pi i tS_0 } \widetilde {  s}_c(\lambda'  ) \Lambda(\lambda') \theta(|\lambda_R'|-t^{\color{black}-\frac 12-2\epsilon})\label{E:<-first-reduction-2-sum}\\
 \times&\mathfrak C \mathfrak T_{(n)}\cdots  \mathfrak C \mathfrak T_{(h+1)}  \mathfrak C \mathfrak T_{(h)}   (\xi_h''-\xi''_{h+1} ) \theta_{1,\xi''_{h+1}}( \xi_h''   )  \left[ \mathfrak C \mathfrak T 1 \right]^{0,(h-1)}|.\nonumber
\end{align} 

We will prove $u^\dagger_{2,1}=\epsilon_\flat\epsilon_\sharp o(t^{-\frac {11}{12}+\epsilon})$ by decomposing the estimates  into two parts, depending on whether $||\xi_{h}''|- |\xi''_{h+1}||> t^{-1+2z}$ or not with $z=\frac16$. Precisely, 
\begin{lemma}\label{L:-2ndReduction-P} Suppose $-1/\delta<a <-\delta<0$.  As $t\to\infty$, 
\begin{align}
  u_{2,1}^\dagger(x)\le  & C\sum_{n=1}^\infty  |\iint d\overline\lambda'\wedge d\lambda' \  e^{\beta_{n+1}2\pi i tS_0 } \widetilde {  s}_c(\lambda'  )\Lambda(\lambda')\theta(|\lambda_R'|-t^{\color{black}-\frac 12-2\epsilon})    \label{E:2ndreduction-2--}\\
\times&\sum_{h=1}^{n} ( P_{n,h}^>+  P_{n,h}^< )\left[ \mathfrak C \mathfrak T 1 \right]^{0,(h-1)}|+ \epsilon_\flat^2 o_\epsilon(t^{-1})   , \nonumber
 \end{align}where $P_{n,h}^>$ and $  P_{n,h}^<$ are defined as in \eqref{E:--first-reduction-2} and $z=\frac16$. 
Here,  for brevity,  when $h=n$, we  identify  
\begin{multline}\label{E:-assertion} 
\qquad\mathfrak C \mathfrak T_{(n)}\cdots \mathfrak C \mathfrak T_{(h+1)} (-2\xi_{ h+1}'' ) \theta(t|\lambda_R'|-|x_{2,h+1}'|) \\
 \times {\color{black}\theta_{1,\xi_{ h+2}''}( \xi_{h+1 }'') }\theta(||\xi_{h+1}''|- |\xi_{h+2}''||-t^{\color{black}-1+2z})= +\frac{ \lambda_I'}{\pi}.\qquad \end{multline}
\end{lemma}   
\begin{proof}
The proof proceeds by the same argument as in Lemma \ref{L:+2ndReduction-P}.
\end{proof}

\begin{lemma}\label{L:-3ndReduction-u21-o}
 Suppose that $-1/\delta<a<-\delta<0$. As $t\to\infty$,   
\begin{align}
&\sum_{n=1}^\infty \sum_{  h=1}^{ n}   \iint d\overline\lambda'\wedge d\lambda' \ e^{\beta_{n+1}2\pi i tS_0 }\widetilde {  s}_c(\lambda'  )\Lambda(\lambda') \theta(|\lambda_R'|-t^{\color{black}-\frac 12-2\epsilon}) P_{n,h}^>   
  \left[ \mathfrak C \mathfrak T 1 \right]^{0,(h-1)}    
  {\color{black}= {\color{black}\epsilon_\flat \epsilon_\sharp} o_\epsilon(t^{-1})}. \nonumber
\end{align} 
\end{lemma}

\begin{proof}

 Applying \eqref{E:CIO-ABF-2}, and arguing as in the proofs of Lemma \ref{L:3-reduction-u20} and \ref{P:first-reduction-u12}, up a term of ${\color{black}\epsilon_\flat^{n}\epsilon_\sharp}o(t^{-1})$, it suffices to show that
\be \label{E:-3rdReduction-(1)-u21-o-sharp} 
   |\iint d\overline\lambda'\wedge d\lambda' \ e^{\beta_{n+1}2\pi i tS_0 }\widetilde {  s}_c(\lambda'  )\Lambda(\lambda') \theta(|\lambda_R'|-t^{ -\frac12-2\epsilon})   P_{n,h}^{>,\sharp}   
  \left[ \mathfrak C \mathfrak T 1 \right]^{0,(h-1)}|   
  {\color{black}= \epsilon^{n+1}_\flat o_\epsilon(t^{-1})},
\ee
where
\begin{align}
    P_{n,h}^{>,\sharp} = &    \psi_{r,0}(\lambda'_R)  \mathfrak C \mathfrak T_{(n)}\cdots \mathfrak C \mathfrak T_{(h+1)} \psi_{r,r}(2\pi \xi_{h+1}'' )\mathfrak C \mathfrak T_{(h)}      \theta_{1,\xi''_{h+1}}( \xi_{h }'')  (\xi_h''-\xi''_{h+1})  \label{E:---first-reduction-2-sharp}\\
\times& \theta(||\xi_h''|- |\xi''_{h+1}||-t^{ -1+2z}) ,\quad \textit{ if $h<n$,}\nonumber\\
  P_{n,h}^{>,\sharp} = &   \psi_{r,0}(\lambda'_R)    \psi_{r,r}(\lambda_I') \mathfrak C \mathfrak T_{(n)}      \psi_{1,\frac{\lambda'}{2\pi} }( \xi_{n }'') (\xi_n''-\frac{\lambda'_I}{2\pi} )\theta(||\xi_h''|- |\frac{\lambda'_I}{2\pi}||-t^{\color{black}-1+2z})  ,\ \  
  \textit{ if $h=n$,}\nonumber
\end{align}  and  $z=\frac 16$. {\color{black}  To this aim, consider the decomposition
\begin{align}
&  \psi_{r,r}(\pi \xi_{1,h+1}')  \theta(| \lambda' _R|-t^{ -\frac 12-2\epsilon} )\theta(||\xi_h''|- |\xi''_{h+1}||-t^{ -1+2z})  \label{E:-3ndReduction-2-u21}\\ 
=&\psi_{r,r}(\pi \xi_{1,h+1}')     \theta(|\lambda_R'|-t^{-2z+\epsilon_z}) \theta(||\xi_h''|- |\xi''_{h+1}||-t^{ -1+2z})\nonumber\\
+&\psi_{r,r}(\pi \xi_{1,h+1}')   \theta(|\lambda_R'|-t^{ -\frac 12-2\epsilon}) \theta(t^{-2z+\epsilon_z}-|\lambda_R'|) \theta(||\xi_h''|- |\xi''_{h+1}||-t^{\color{black}-1+2z}) \nonumber
\end{align} with $0<\epsilon_z\ll z$.

From $  b=  (  r^2+{\lambda'}_R^2)^{1/2}/{2\pi}$, 
\begin{multline}\label{E:-2ndReduction-6-(new)-2}
   \psi_{r,r}(\pi \xi_{1,h+1}')     \theta(|\lambda_R'|-t^{-2z+\epsilon_z}) \theta(||\xi_h''|- |\xi''_{h+1}||-t^{ -1+2z}) \theta(t|\lambda_R'|-|x_{2,h }'|)  \\
 \times    |(  x_{2,h }'+3  t\lambda'_R)(\xi_h''-\xi_{h+1}'')  (\xi_h''+\xi_{h+1}'') | \ge t^{ 1- 2z+\epsilon_z-1+2z }/C . 
 \end{multline} which implies  the $L^1(d\xi_h'')$-norm of $ \mathcal F^{(h)} $ on the corresponding domains for the first term on the right hand side of \eqref{E:-3ndReduction-2-u21} is less than $ \epsilon_\flat^{n+1}  o(t^{-1})$. It then reduces to proving:  
\begin{align}
&   |\iint d\overline\lambda'\wedge d\lambda'  \ e^{\beta_{n+1}2\pi i tS_0 } \widetilde {  s}_c(\lambda'  ) \Lambda(\lambda')    \theta(|\lambda'_I|-\frac r{10})  \theta(|\lambda_R'|-t^{\color{black}-\frac 12-2\epsilon}) \theta(t^{ -2z+\epsilon_z}-|\lambda_R'|) \nonumber\\
\times&  \mathfrak C \mathfrak T_{(n)}\cdots \mathfrak C\mathfrak T_{(h+1)}   \psi_{t^{ -y+\epsilon_y},r}(\xi''_{h+1} )\mathfrak C\mathfrak T_{(h)}  )\nonumber\\
\times&(\xi_h''-\xi''_{h+1}) \theta(||\xi_h''|- |\xi''_{h+1}||-t^{\color{black}-1+2z})\left[\mathfrak C\mathfrak T1\right]^{0,(h-1)}| \nonumber\\
+&    |\iint d\overline\lambda'\wedge d\lambda'  \  e^{\beta_{n+1}2\pi i tS_0 } \widetilde {  s}_c(\lambda'  )\Lambda(\lambda') \theta(|\lambda'_I|-\frac r{10})  \theta(|\lambda_R'|-t^{ -\frac 12-2\epsilon}) \theta(t^{ -2z+\epsilon_z}-|\lambda_R'|)\nonumber\\
\times& \mathfrak C \mathfrak T_{(n)}\cdots \mathfrak C\mathfrak T_{(h+1)} \psi_{r,r}(\xi''_{h+1} )  (1-\psi_{t^{ -y+\epsilon_y},r}(\xi''_{h+1} )) \mathfrak C\mathfrak T_{(h)}\psi_{t^{ -y-\epsilon_y},r}(\xi''_{h}) \nonumber\\
\times& (\xi_h''-\xi''_{h+1})\theta(||\xi_h''|- |\xi''_{h+1}||-t^{\color{black}-1+2z})\left[\mathfrak C\mathfrak T1\right]^{0,(h-1)} |\nonumber\\
+&    |\iint d\overline\lambda'\wedge d\lambda'  \ e^{\beta_{n+1}2\pi i tS_0 } \widetilde {  s}_c(\lambda'  )\Lambda(\lambda') \theta(|\lambda'_I|-\frac r{10})  \theta(|\lambda_R'|-t^{\color{black}-\frac 12-2\epsilon}) \theta(t^{\color{black}-2z+\epsilon_z}-|\lambda_R'|)  \nonumber\\
\times& \mathfrak C \mathfrak T_{(n)}\cdots \mathfrak C\mathfrak T_{(h+1)} {\color{black}\psi_{r,r}(\xi''_{h+1} )}  (1-\psi_{t^{\color{black}-y+\epsilon_y},r}(\xi''_{h+1}))\mathfrak C\mathfrak T_{(h)}(1-\psi_{t^{\color{black}-y-\epsilon_y},r}(\xi''_{h}))  \nonumber\\
\times& (\xi_h''-\xi''_{h+1})\theta(||\xi_h''|- | \xi''_{h+1}||-t^{\color{black}-1+2z})\left[\mathfrak C\mathfrak T1\right]^{0,(h-1)} |\nonumber\\
\equiv &I_1+I_2+I_3{\color{black}=  {\color{black}\epsilon_\flat^{n }\epsilon_\sharp}o_{\epsilon,\epsilon_1}(t^{-11/12+\epsilon})},\nonumber
\end{align}where $y$ is defined by  \eqref{E:maximise} and $0<\epsilon_y\ll y$.}

Proceeding as in the proof of Theorem \ref{T:-u-2} and choosing 
\be\label{E:-2ndReduction-5-u21-(0)-new}
y=\frac 14+\frac\epsilon 2,\quad z=\frac 16, \quad\epsilon_y,\,\epsilon_z\ll\epsilon,
\ee we obtain
\be\label{E:-2ndReduction-5-u21-(0)}
\begin{split}
|I_1|\le& C_\epsilon {\color{black}\epsilon_\flat^{n }\epsilon_\sharp} t^{-2z+\epsilon_z-y+\epsilon_y-{\color{black}\frac 12+\epsilon}}=  \epsilon^{n}_\flat \epsilon_\sharp o(t^{-1}) ,\\
|I_2|\le& \epsilon_\flat^{n+1} o_{\epsilon,y}(t^{-1})=  \epsilon^{n}_\flat \epsilon_\sharp o(t^{-1}),\\
|I_3|\le& C{\color{black}\epsilon_\flat^{n }\epsilon_\sharp}t^{-1+y+\epsilon_y-2z+\epsilon_z}=  \epsilon^{n}_\flat \epsilon_\sharp o(t^{-1})  .
\end{split}
\ee
 
\end{proof}
\begin{lemma}\label{L:-3ndReduction-u21}
 Suppose $-1/\delta<a<-\delta<0$ and $0<\epsilon\ll 1$.   As $t\to\infty$,   
\[
 \sum_{n=1}^\infty \sum_{  h=1}^{ n}  |\iint d\overline\lambda'\wedge d\lambda' \  e^{\beta_{n+1}2\pi i tS_0 }\widetilde {  s}_c(\lambda'  )\Lambda(\lambda')\theta(|\lambda_R'|-t^{\color{black}-\frac12-2\epsilon})  P_{n,h}^<     \left[ \mathfrak C \mathfrak T 1 \right]^{0,(h-1)}| {\color{black}=  {\color{black}\epsilon_\flat \epsilon_\sharp} o_\epsilon(t^{-11/12+\epsilon})}.  \]
\end{lemma}

\begin{proof}From  \eqref{E:confine--u20}, \eqref{E:-3ndReduction-2-u21},   $z=\frac 16$, and a similar argument as argument used in     Lemma \ref{L:+2ndReduction-u21-inside},  via decomposing $ 2\xi''_{h+1}= 2(\xi''_{h+1}-\xi''_{h+2})+2\xi''_{h+2}$, applying $
\sum_{n=1}^{\infty}n r^n
 <C{r} $, Lemma \ref{L:-3ndReduction-u21-o}, and Theorem \ref{T:-u-2}, we obain  
\begin{align}
&\sum_{n=1}^\infty \sum_{  h=1}^{ n}  |\iint d\overline\lambda'\wedge d\lambda' \  e^{\beta_{n+1}2\pi i tS_0 }\widetilde {  s}_c(\lambda'  )\Lambda(\lambda')\theta(|\lambda_R'|-t^{\color{black}-\frac12-2\epsilon})  P_{n,h}^<     \left[ \mathfrak C \mathfrak T 1 \right]^{0,(h-1)}|\label{E:-3ndReduction-4-u21}\\
\le & |\iint d\overline\lambda'\wedge d\lambda' \ e^{\beta_{n+1}2\pi i tS_0 }\widetilde {  s}_c(\lambda'  ) \Lambda(\lambda')\theta(|\lambda_R'|-t^{\color{black}-\frac12-2\epsilon}) \theta(t^{\color{black}-2z+\epsilon_z}-|\lambda_R'|)\nonumber\\
 \times& \mathfrak C \mathfrak T_{(n)}\cdots\mathfrak C \mathfrak T_{(h+1)}   (-2  \xi_{h+1}'')  {\color{black}\psi_{r,0}(\lambda_R')}\psi_{r,r}(\xi''_{h+1} )     \theta(||\xi_{h+1}''|- |\xi''_{h+2}||-t^{\color{black}-1+2z})  
   \nonumber\\
\times& \mathfrak C \mathfrak T_{(h)}      \theta(t^{\color{black}-1+2z}-|\xi_h''+\xi''_{h+1}|)\theta(|\xi_h''-\xi''_{h+1}|-t^{\color{black}-1+2z})\left[ \mathfrak C \mathfrak T 1 \right]^{0,(h-1)}| \nonumber\\
=&\epsilon_\flat \epsilon_\sharp o_\epsilon(t^{-11/12+\epsilon}). \nonumber
\end{align}

\end{proof}

 Therefore, combining Proposition \ref{P:first-reduction-u12}, \eqref{E:<-first-reduction-2-sum},  and Lemmas \ref{L:-2ndReduction-P}-\ref{L:-3ndReduction-u21},   we obtain: 
\begin{theorem}\label{T:-u-2-u21}
Suppose that $-1/\delta<a<-\delta<0$ and \eqref{intro-ID} holds.
Then, for every sufficiently small $\epsilon>0$,
as $t\to\infty$,
\[
 u_{2,1}{\color{black}= \epsilon_\flat\epsilon_\sharp o_\epsilon(t^{-11/12+\epsilon})}.
\] 
\end{theorem}

\begin{appendix}

\begin{section}{List of Symbols}\label{App}

{\small\begin{longtable}{|lrc|clr|}
\caption{List of Symbols}\\
\hline
Notation and Definition &Page &&& Notation and Definition &Page\\
\hline
ABF   Cauchy integrals&&&&Fourier transform &\\
$\quad   x_{1,n}',\ x_{2,n}'$, &19&&&$\quad \widehat f(\xi)$,  &6\\
$\quad \xi_1'',\ \xi_n'',\ \xi_h'',\ \xi_{1,n+1}'
$, &15,19,51&&&$\quad \phi^{\wedge_{\zeta_R'}}(\zeta_R'),\ \phi^{\wedge_{\zeta_I'}}(\zeta_I')$   &10  \\
$\quad  [\mathfrak C\mathfrak T  ]^{0,(n)},    [\mathfrak C\mathfrak T  ]^{1,(n)},   \mathfrak C\mathfrak T _{0,(n)}$, &15,19,51 &&&&\\
$\quad \mathfrak S(a;\lambda_R';\xi_1''),\, \mathfrak S^\sharp(a,t;x_1',x_2';\lambda_R';\xi_1'')$, &15&&&Inverse scattering theory&\\
$\quad \mathcal F(t;\lambda';x_2'; \xi_1''),\ \mathcal F^{(n)}(t;\lambda';x_{2,n}'; \xi_n'')$, &15,19&&&$\quad \mathcal S,\ s_c,\ \mathcal C,\ T$,   &6 \\
$\quad\beta_n$,&19&&&$\quad   \mathbf s$  &13\\
$\quad  \Xi,\ \Xi^c,\ \Sigma_\gtrless $, &14,15 &&&&\\
$\quad  \mathfrak R(\lambda'_R,\xi'),\ \Theta,\ \Theta_{amp},\ \Theta_{pole}$, & 15&&&Potentials (KPII solutions)&\\
$\quad \pm b$ stationary points for $\mathfrak G(\xi_n'')$, &16&&&$\quad u(x),\  u_0(x_1,x_2)$, &2\\
$\quad P_{n,h}^>,\ P_{n,h}^< $  &59 &&&$\quad u_1(x),\  u_{1,1}(x),\ u_{1,2}(x)$, &2,3,8,9\\
&&&&$\quad u_{2,0}(x),\  u_{2,1}(x)$,  &2,3\\
   Cauchy integrals&&&&$\quad   \mathbf v$  &13\\
   $\quad  p_\lambda(\xi),\ z_0,\ \Omega_\lambda,\ \Omega_{\lambda,\rho}$  &8,14&&&&\\
$\quad  \wp,\ \wp_1,\ \wp_2$, &14&&&Stationary theory&\\ 
$\quad  r_\ast,\ \rho_1,\ \rho_2,\ \epsilon_1,\ \epsilon_2,\ \epsilon_\flat,\ \epsilon_\sharp$, &14&&& $\quad  (t_1,t_2,t)$, &7 \\
$\quad  R_{in,\gtrless},\ R_{in,\gtrless,+},\ R_{in,\gtrless,-}\ R_{out,\gtrless} $, &20-24&&&$\quad \zeta=\zeta_R+i\zeta_I,\ \zeta'=\zeta'_R+i\zeta'_I$, & 7  \\
$\quad  P_{in,\gtrless},\ P_{in,\gtrless,+},\ P_{in,\gtrless,-}\ P_{out,\gtrless} $, &27-29&&&$\quad (\xi_1',\xi_2'),\ \partial_{\zeta_R'},\ \partial_{\zeta_I'}$, &7\\
$\quad  Q_{in,\gtrless},\ Q_{in,\gtrless,+},\ Q_{in,\gtrless,-}\ Q_{out,\gtrless} $, &33,34,40&&&$\quad \widetilde f(\zeta')$, &7\\
$\quad  \mathfrak R_{in,\gtrless,-},\ \mathfrak P_{in,\gtrless,-},\ \mathfrak Q_{in,\gtrless,-}$,  &52&&&$\quad \mathbb S_0(t_1,t_2,\zeta),\ S_0(a;\zeta')$, & 7\\
$\quad   \mathfrak Q^\flat_{in,\gtrless,-},\ \mathfrak Q^\sharp_{in,\gtrless,-} $  &53&&&$\quad \nabla S_0(a;\zeta'),\ \Delta S_0(a;\zeta')$, &9\\
&&&&$\quad a,\ \pm r$, &2,7,8\\
Coordinates &&&& &\\
$\quad x=(x_1,x_2,x_3)$, &2 &&&Special functions &\\ 
$\quad \partial_x^l=\partial_{x_1}^{l_1}\partial_{x_2}^{l_2}\partial_{x_3}^{l_3},\ |l|=l_1+l_2+l_3$, &6 &&& $\quad $Airy function $Ai(z)$, &11\\
$\quad \xi=(\xi_1,\xi_2)$, &6 &&& $\quad $Heaviside function $\theta(s)$, &11\\
$\quad C,\ \delta $  &2 &&&  $\quad \mathfrak M^{p,q}$, &2\\ 
&&&& $\quad  \Lambda(\lambda')$&52\\

  \hline
\end{longtable}}

\end{section}

\end{appendix}


\begin{thebibliography}{ABC999}
\bibitem{AYF83}
Ablowitz, M. J., Bar Yaacov, D., Fokas, A. S.:
On the inverse scattering transform for the Kadomtsev-Petviashvili equation. \textit{Stud. Appl. Math.} 69 (1983), no. 2, 135-143.

Summary

\bibitem{AKK94}
Anders, I. A., Kotlyarov, V. P., Khruslov, Ē. Ya.:
Curved asymptotic solitons of the Kadomtsev-Petviashvili equation.\textit{Teoret. Mat. Fiz.} 99  (1994), no. 1, 27-35; translation in
\textit{Theoret. and Math. Phys.} 99 (1994), no. 1, 402-408.

\bibitem{DLP25}
Donmazov, S., Liu, J., Perry, P.:
Large-time asymptotics for the Kadomtsev-Petviashvili I equation. \textit{arXiv:2409.14480}, 1-71.

\bibitem{HN14}
{\color{black}Hayashi,N.,   Naumkin, P.:
Large time asymptotics for the Kadomtsev-
Petviashvili equation. \textit{Comm. Math. Phys.} 332 (2014), no. 2, 505–533.}

\bibitem{HNS99}
Hayashi, N., Naumkin, P., Saut, J. C.:
Asymptotics for large time of global solutions to the generalized Kadomtsev-Petviashvili equation. \textit{Comm. Math. Phys.} 201 (1999), 577-590.  

\bibitem{Ki06}
Kiselev, O. M.:
Asymptotics of solutions of multidimensional integrable equations and their perturbations. \textit{J. Math. Sci.} (N.Y.) 138 (2006), no. 6, 6067-6230.

\bibitem{KS21}
Klein, C., Saut, J. C.:
Nonlinear dispersive equations-inverse scattering and PDE methods. \textit{Applied Mathematical Sciences} 209 (2021), Springer, Cham.

\bibitem{N11}
Niizato, T.:
Large time behavior for the generalized Kadomtsev-Petviashvili equations. \textit{Diff. Eq. and  Appl. } 3 (2) (2011), 299-308. 

\bibitem{W85} 
 Wickerhauser, M. V.: Nonlinear evolutions of the heat operator. \textit{Dissertation, Yale University }  (1985).
 
\bibitem{W87}
Wickerhauser, M. V.:
Inverse scattering for the heat operator and evolutions in $2+1$ variables. \textit{Comm. Math. Phys.} 108 (1987), no. 1, 67-89.

\end{thebibliography}
\end{document}